\documentclass[a4paper, 10pt]{amsart}

\usepackage{amsmath}
\usepackage{amssymb}

\usepackage{amscd,amsfonts,amsthm,bm,bbm,comment,eucal,mathrsfs,mathtools,stmaryrd,MnSymbol}
\usepackage{hyperref}
\usepackage{graphicx}
\usepackage[all]{xy}
\usepackage{tikz}
\usetikzlibrary{cd}
\usetikzlibrary{decorations.pathmorphing}
\usetikzlibrary{patterns}
\DeclareMathOperator{\Ad}{Ad}

\DeclareMathOperator{\Ber}{Ber}
\DeclareMathOperator{\Br}{Br}
\DeclareMathOperator{\BT}{BT}
\DeclareMathOperator{\BW}{BW}
\DeclareMathOperator{\can}{can}
\DeclareMathOperator{\CSAlg}{CSAlg}

\DeclareMathOperator{\diag}{diag}
\DeclareMathOperator{\bEnd}{\mathbbmss{E}nd}
\DeclareMathOperator{\End}{End}
\DeclareMathOperator{\ev}{ev}

\DeclareMathOperator{\GRet}{GRet}
\DeclareMathOperator{\GT}{GT}
\DeclareMathOperator{\bHom}{\mathbbmss{H}om}
\DeclareMathOperator{\Hom}{Hom}
\DeclareMathOperator{\hy}{hy}
\DeclareMathOperator{\id}{id}

\DeclareMathOperator{\Ind}{Ind}
\DeclareMathOperator{\Irr}{Irr}

\DeclareMathOperator{\Res}{Res}
\DeclareMathOperator{\sep}{sep}

\DeclareMathOperator{\spl}{spl}

\DeclareMathOperator{\Spec}{Spec}

\DeclareMathOperator{\qrat}{qrat}
\DeclareMathOperator{\super}{super}
\DeclareMathOperator{\sqrat}{sqrat}

\DeclareMathOperator{\GL}{GL}
\DeclareMathOperator{\PGL}{PGL}

\DeclareMathOperator{\SL}{SL}
\DeclareMathOperator{\SO}{SO}
\DeclareMathOperator{\Sp}{Sp}

\DeclareMathOperator{\Pp}{P}
\DeclareMathOperator{\Qq}{Q}
\DeclareMathOperator{\SpO}{SpO}
\DeclareMathOperator{\Uu}{U}

\newcommand{\bA}{\mathbb{A}}
\newcommand{\bC}{\mathbb{C}}

\newcommand{\bH}{\mathbb{H}}

\newcommand{\bR}{\mathbb{R}}

\newcommand{\bZ}{\mathbb{Z}}

\newcommand{\fg}{\mathfrak{g}}
\newcommand{\fgl}{\mathfrak{gl}}
\newcommand{\fh}{\mathfrak{h}}

\newcommand{\fspo}{\mathfrak{spo}}

\newcommand{\fu}{\mathfrak{u}}
\newcommand{\fv}{\mathfrak{v}}

\newcommand{\fS}{\mathfrak{S}}

\newcommand{\bmB}{\mathbbmss{B}}
\newcommand{\bmG}{\mathbbmss{G}}
\newcommand{\bmH}{\mathbbmss{H}}
\newcommand{\bmU}{\mathbbmss{U}}
\newcommand{\bmX}{\mathbbmss{X}}
\newcommand{\xcong}[1]{%
	\mathrel{\tikz[baseline=0pt] {
			\node[above] at (0,1.2ex) (a) {\(\scriptstyle #1\)};
			\draw[preaction={
				transform canvas={yshift=-.5ex},
				draw,
				decorate,
				decoration={lineto}},
			preaction={
				transform canvas={yshift=-1ex},
				draw,
				decorate,
				decoration={lineto}}]
			(a.south west) .. controls +(.25,.15) and +(-.25,-.15) .. (a.south east);
}}}
\theoremstyle{plain}
\newtheorem{thm}{Theorem}[section]
\newtheorem{cor}[thm]{Corollary}
\newtheorem{defn-prop}[thm]{Definition-Proposition}
\newtheorem{lem}[thm]{Lemma}
\newtheorem{prop}[thm]{Proposition}
\newtheorem{var}[thm]{Variant}
\theoremstyle{definition}
\newtheorem{ass}[thm]{Assumption}

\newtheorem{cons}[thm]{Construction}
\newtheorem{defn}[thm]{Definition}
\newtheorem{ex}[thm]{Example}
\newtheorem{conv}[thm]{Convention}
\newtheorem{note}[thm]{Notation}
\newtheorem{notedefn}[thm]{Notation-Definition}

\newtheorem{rem}[thm]{Remark}
\begin{document}
	\title[Classification of irreducible representations and their endomorphisms]{Classification of irreducible representations of affine group superschemes and the division superalgebras of their endomorphisms}
	\author{Takuma Hayashi}
	\address{Osaka Central Advanced Mathematical Instituteatics,
		Osaka Metropolitan University,
		3-3-138 Sugimoto, Sumiyoshi-ku Osaka 558-8585, Japan}
	\email{takuma.hayashi.forwork@gmail.com}
	\date{}
	\begin{abstract}
		In this paper, we classify irreducible representations of affine group superschemes over fields $F$ of characteristic not two in terms of those over a separable closure $F^{\sep}$ and their Galois twists. We also compute the division superalgebras of their endomorphisms. Finally, we give numerical conclusions for quasi-reductive algebraic supergroups under certain conditions, based on Shibata's Borel--Weil theory in \cite{MR4039427}.
	\end{abstract}
	
	\maketitle
	
	\section{Introduction}\label{sec:intro}
	
	The classification problem of irreducible representations over non algebraically closed fields goes back to Loewy's scheme for finite dimensional irreducible representations over the field $\bR$ of real numbers of groups in \cite{MR1500635}. In particular, it dealt with the descent problem from representations over the field $\bC$ of complex numbers to those over $\bR$. As we will briefly explain a little later, Frobenius and Schur determined Loewy's scheme for finite groups in \cite{zbMATH02646565} by a numerical formula which is nowadays called the Frobenius--Schur indicator (see \cite[Chapter 23]{MR1864147} for an exposition). Then Cartan gave a similar classification scheme for real Lie algebras in \cite{E1914} (see \cite{MR102534} for an exposition). He also determined it for real simple Lie algebras in \cite{E1914}. One can find that Loewy's classification scheme readily works for real irreducible representations of real affine group superschemes $\bmG$. To give the statement, let us collect basic notations: Let $\Irr \bmG$ (resp.\ $\Irr \bmG\otimes_\bR\bC$) denote the set of isomorphism classes of irreducible representations of $\bmG$ (resp.\ $\bmG\otimes_\bR\bC$). Set
	\[\begin{array}{cc}
		\Irr_{\mathrm{a}}\bmG=\{V\in\Irr \bmG:~V\otimes_\bR\bC\ \mathrm{is\ irreducible}\},
		&\Irr_{\mathrm{na}}\bmG=\Irr \bmG\setminus \Irr_{\mathrm{a}} \bmG.
	\end{array}\]
	We define the functor from the category of representations of $\bmG\otimes_\bR\bC$ to that of $\bmG$ by restriction of the scalar from $\bC$ to $\bR$. Define subsets $\Irr_\bullet\bmG\otimes_\bR\bC\subset \Irr \bmG\otimes_\bR\bC$ for $\bullet\in\{1,\neq 1\}$ by
	\[\Irr_1 \bmG\otimes_\bR\bC
	=\{V'\in \Irr \bmG\otimes_\bR\bC:~\Res_{\bC/\bR} V'\textrm{~is~reducible}\},\]
	\[\Irr_{\neq 1} \bmG\otimes_\bR\bC
	=\Irr\bmG\otimes_\bR\bC\setminus \Irr_{ 1} \bmG\otimes_\bR\bC.\]
	Define an equivalence relation on $\Irr_{\neq 1} \bmG\otimes_\bR\bC$ by $\bar{V}\sim V$ for $V\in \Irr_{\neq 1} \bmG\otimes_\bR\bC$, where $\bar{V}$ is the complex conjugate representation to $V$.
	
	\begin{thm}\label{thm:Cartan}
		\begin{enumerate}
			\renewcommand{\labelenumi}{(\arabic{enumi})}
			\item An irreducible representation $V$ of $\bmG\otimes_\bR\bC$ admits a real form if and only if $\Res_{\bC/\bR} V$ is reducible.
			\item The complexification determines a bijection
			$\Irr_{\mathrm{a}} \bmG\cong \Irr_1 \bmG\otimes_\bR\bC$.
			\item The restriction determines a bijection
			$\Irr_{\neq 1} \bmG\otimes_\bR\bC/\sim\cong\Irr_{\mathrm{na}} \bmG$.
		\end{enumerate}
	\end{thm}
	
	Theorem \ref{thm:Cartan} encourages us to work on the descent problem of complex irreducible representations for the classification and construction of real irreducible representations. In fact, sum up the two cases in Theorem \ref{thm:Cartan} to obtain a surjective map
	$\Irr \bmG\otimes_\bR\bC\to \Irr \bmG$, which we denote by $(-)_\bR:V\mapsto V_\bR$. This map is one-to-one on the subset of self-conjugate representations which contains
	$\Irr_1 \bmG\otimes_\bR\bC$; $(-)_\bR$ is two-to-one outside it. Therefore the real irreducible representations are obtained in the following way:
	\begin{enumerate}
		\item[1.] For each $V\in \Irr \bmG\otimes_\bR\bC$, determine whether $V$ is self-conjugate or not.
		\item[2.] If $V$ is not self-conjugate, $V$ and $\bar{V}$ determine the same real irreducible representation $\Res_{\bC/\bR} V$.
		\item[3.] Suppose $V\in \Irr \bmG\otimes_\bR\bC$ is self-conjugate. If $V$ belongs to $\Irr_1 \bmG\otimes_\bR\bC$, take its real form; Otherwise assign $\Res_{\bC/\bR} V$.
		\item[4.] These give a complete list of the irreducible representations of $\bmG$. 
	\end{enumerate}
	
	In general, a self-conjugate complex irreducible representation $V$ is called of index one if it admits a real form; Otherwise, $-1$ is assigned for its index. In the non-super (= classical) setting, all the conditions in this procedure are checked numerically at once for finite groups by the Frobenius--Schur indicator (\cite{zbMATH02646565}). In particular, it gives a formula to the index (if the given irreducible representation is self-conjugate). The problems of self-conjugacy and computation of the indices for complex irreducible representations of real simple Lie algebras were solved by Cartan in \cite{E1914}.
	
	There is another perspective to the computation of the index. Let $V$ be a complex irreducible finite dimensional representation, and $V_\bR$ be the corresponding real irreducible representation. Then $\bar{V}\cong V$ if and only if the division algebra of endomorphisms of $V_\bR$ is central. If these equivalent conditions fail, the division algebra is $\bC$. If $\bar{V}\cong V$, the central division algebra is determined by the index. Namely, the map assigning the indices for self-conjugate complex irreducible representations can be identified with the map to the Brauer group of $\bR$ assigning the central division algebras of endomorphisms. Let us note that to study the division algebras of endomorphisms of real irreducible representations is enhancement of determination of the fibers of $(-)_\bR$ in the sense that a complex irreducible representation $V$ is self-conjugate if and only if the division algebra of endomorphisms of $V_\bR$ is not $\bC$.
	
	In \cite{MR4581479}, Ichikawa and Tachikawa worked on its super enhancement for finite groups. They defined a map from the set of complex irreducible representations which are self-conjugate up to parity change to the Brauer--Wall group of $\bR$ by assigning (the opposite superalgebras to) the central division superalgebras of endomorphisms of the corresponding real irreducible representations. Precisely speaking, they took cocycles to define a twisted group algebra with a super structure. For each self-conjugate irreducible representation up to parity change, they assigned the real form of the corresponding simple summand in the twisted group algebra over $\bC$ as a reformulation of the study of the division superalgebras of endomorphisms of real irreducible representations.
	They found that this map to the Brauer-Wall group is computed by the super Frobenius--Schur indicator. They also showed that a complex irreducible representation is self-conjugate up to parity change if and only if its super Frobenius--Schur indicator is nonzero.
	
	To talk about working over general ground fields, let $G$ be a linear algebraic group over a field $F$. Let $F^{\sep}$ be a separable closure of $F$, and $\Gamma$ be the absolute Galois group of $F$. Write $\Irr G$ and $\Irr G\otimes_F F^{\sep}$ for the sets of isomorphism classes of irreducible representations of $G$ and $G\otimes_F F^{\sep}$ respectively. We notice that $\Gamma$ acts on $\Irr G\otimes_F F^{\sep}$ by the Galois twist. It should be well-known that there is a bijection $\Irr G\cong\Gamma\backslash\Irr G\otimes_F F^{\sep}$. For $V'\in\Irr G\otimes_F F^{\sep}$ $\Gamma$-invariant, Borel and Tits introduced the obstruction class to the existence of a datum of Galois descent on $V'$ in \cite[Section 12]{MR207712}, which is an element of the second Galois cohology $H^2(\Gamma,(F^{\sep})^\times)$. They also discussed how to check the $\Gamma$-invariance and to compute the 2-cocyle when $G$ is connected reductive.
	
	In this paper, we work with a general super setting. The first main results are as follows:
	
	\begin{thm}[Theorem \ref{thm:pure}, Variant \ref{var:pure}, Corollary \ref{cor:compute_F'}]\label{thmA}
		Let $\bmG$ be an affine group superscheme over a field $F$ of characteristic not two. 
		\begin{enumerate}
			\item There is a bijection between the sets of irreducible representations of $\bmG$ and Galois orbits of irreducible representations of $\bmG\otimes_F F^{\sep}$ such that it respects the parity change $\Pi$.
			\item Let $V$ be an irreducible representation of $\bmG\otimes_F F^{\sep}$, and $V_F$ be the corresponding irreducible representation of $\bmG$.
			\begin{enumerate}
				\item[(i)] The division algebra (resp.~superalgebra) of endomorphisms of $V_F$ is central if and only if the endomorphisms of $V$ are only scalar and all the Galois twists of $V$ are isomorphic to $V$ (resp.~to $V$ up to parity change).
				\item[(ii)] Let $F'$ be the even part of the center of the division superalgebra of endomorphisms of $V_F$. Suppose that $F'$ is a separable extension of $F$. Then $V_F$ is the restriction of $V_{F'}$ for a fixed embedding of $F'$ into $F^{\sep}$. Moreover, the division superalgebra of endomorphisms of $V_{F'}$ is central, and it is equal to that of $V_F$.
			\end{enumerate}
		\end{enumerate}
	\end{thm}
	
	These follow from a standard argument on the Galois descent with a careful analysis of the parity changed irreducible representations. We remark that if $F$ is perfect, the study of the division superalgebra $\bEnd_{\bmG}(V_F)$ of endomorphisms of $V_F$ is reduced to the case that $\bEnd_{\bmG}(V_F)$ is central by (ii) of (2). For numerical aspects of this reduction, $F'$ is expressed in terms of $V$ in Corollary \ref{cor:compute_F'} when $F$ is perfect. In this paper, we also discuss generalities on computation of $\bEnd_{\bmG}(V_F)$ in terms of $V$ and its Galois twists when $\bEnd_{\bmG}(V_F)$ is central (Propositions \ref{prop:even} and \ref{prop:epsilon}, Theorems \ref{thm:BT} and \ref{thm:BT_nonqrat}). In particular, Section \ref{sec:BT} is devoted to a numerical description of $\bEnd_{\bmG}(V_F)$. This leads to a numerical criterion for the pseudo-absolute irreducibility (Corollary \ref{cor:abs_irr}). This was motivated by \cite[Section 12]{MR207712}. Preliminary arguments on rationality of irreducible representations towards the continuity of certain Galois actions are given in Section \ref{sec:rationality}.

	\begin{rem}
		One can easily recognize that the results so far hold more generally for right supercomodules of supercoalgebras. We concentrate on affine group superschemes just for consistency; As we mention right below, we are particularly interested in representation theory of quasi-reductive algebraic supergroups. Here we note this because we would like to clarify the relation of \cite{MR4581479} and the present work.
		
		In the non-super setting, we can relate the affine group scheme theoretic formalism in \cite[Section 12]{MR207712} to \cite{zbMATH02646565} as follows: Let $G$ be a finite group. Then the dual $F[G]^\vee$ of the group algebra $F[G]$ is naturally equipped with the structure of a commutative Hopf algebra. The category of representations of $G$ is canonically isomorphic to that of $\Spec F[G]^\vee$, which is compatible with base changes. Put $F=\bR$. Then the self-conjugacy of a complex irreducible representation of $(\Spec \bR[G]^\vee)\otimes_\bR\bC$ and the obstruction to the existence of a descent datum for $V$ self-conjugate are determined by the Frobenius--Schur indicator.

		Recall that in \cite{MR4581479}, they constructed the structure of a superalgebra on $F[G]$ from a one-cocycle $\varphi$ of $G$. However, one can easily see that there is no superalgebra homomorphism $F[G]\to F$ unless $\varphi$ is trivial. In particular, $F[G]$ does not admit the structure of a (commutative) Hopf superalgebra if $\varphi$ is nontrivial. The superdual $F[G]^\vee$ of $F[G]$ is still canonically equipped with the structure of a supercoalgebra. Moreover, the category of left $F[G]$-supermodules is canonically isomorphic to that of right supercomodules of $F[G]^\vee$. Apply our theory to $F[G]^\vee$ to deduce the abstract classification theorem of simple left $F[G]$-supermodules and a general formula on the division superalgebras of their endomorphisms. \cite[Theorem 1.2]{MR4581479} tells us that this formula is computed explicitly by the super Frobenius--Schur indicator if $F=\bR$. Precisely speaking, put $F=\bR$, and let $V$ be a simple right $(\bR[G]^\vee\otimes_\bR\bC)$-supercomodule. Identify $V$ with a left $\bC[G]$-supermodule to define the super Frobenius--Schur indicator $\mathcal{S}(V)$. Then $V$ is self-conjugate up to parity change if and only if $\mathcal{S}(V)\neq 0$. Moreover, if these equivalent conditions hold, the division superalgebra of endomorphisms of the corresponding simple right $\bR[G]^\vee$-supercomodule is central and is determined by $\mathcal{S}(V)$.
	\end{rem}
	
	Finally, we would like to apply our theory to a super analog of reductive algebraic groups:
	
	\begin{defn}\label{defn:qred}
		\begin{enumerate}
			\item A linear algebraic supergroup over $F$ is an affine group superscheme of finite type over $F$ (cf.~\cite[Section 2.2]{MR4039427}).
			\item We call a linear algebraic supergroup $\bmG$ over $F$ (connected) quasi-reductive if it is of finite type and its even part is connected reductive in the classical sense.
		\end{enumerate}
	\end{defn}
	
	\begin{rem}
		You may wonder about the smoothness for the definition of linear algebraic supergroups. Indeed, there is a notion of (formal) smoothness for affine superschemes over $F$ (\cite[Appendix A.1]{MR3545265}). According to \cite[Proposition A.3]{MR3545265}, a linear algebraic supergroup over $F$ in the above sense is smooth if and only if so is its even part\footnote{This equivalence holds without the of finite type hypothesis in the terminology of \cite{MR3545265}.}. Though we do not concern the smoothness apparently in this paper, our linear algebraic supergroups in Sections \ref{sec:qred} and \ref{sec:qredII} are smooth; Before that, we do not even assume our affine group superschemes to be of finite type.
	\end{rem}
	
	\begin{rem}
		The original definition of the quasi-reductivity goes back to \cite{MR2849718}. Serganova introduced this notion for linear algebraic supergroups and Lie superalgebras over algebraically closed fields of characteristic zero. Subsequently, Shibata defined quasi-reductive algebraic supergroups over general base fields of characteristic not two to be linear algebraic supergroups with split reductive even part in \cite{MR4039427} to use root systems. Actually, he defined a quasi-reductive supergroup over a PID $k$ with 2 regular as an affine group superscheme of finite type over $k$ with split reductive even part such that the odd part of the cotangent module at the unit is free of finite rank. We note that the last free of finite rank condition holds in general if the ground ring is a field. In fact, the odd part of the cotangent module is always free in this case. The finite type hypothesis implies that the cotangent module is finitely generated. When we think of working on supergroups with reductive even part, they seem to suit the terminology quasi-reductive more. After the author's private communication with Shibata, we decided to change the definition as above and to call quasi-reductive algebraic supergroups in the sense of \cite{MR4039427} split quasi-reductive in this paper. If one wishes to have a notion of quasi-reductive group superschemes over general ground rings $k$ with 2 regular, a possible candidate seems to be an affine group superscheme of finite presentation over $k$ with reductive even part in the sense of \cite[D\'efinition 2.7]{MR0228502} such that the odd part of the cotangent module at the unit is projective.
	\end{rem}

	Examples of non-split quasi-reductive algebraic supergroups over $F=\bR$ are given in \cite{MR4538712, MR2276521} through the Galois descent. We will construct some examples directly as real algebraic supergroups in Section \ref{sec:qred}, based on the classification of real simple Lie superalgebras (of classical type) in \cite[TABLE 3]{MR0714220} and \cite{MR565713}. Another kind of typical examples are the so-called pure inner forms: Suppose that we are given a (quasi-reductive) linear algebraic supergroup $\bmG$ over $F$. Then each element of the first Galois cohomology $H^1(\Gamma,\bmG(F^{\sep}))=H^1(\Gamma,G(F^{\sep}))$\footnote{Sometimes it is written as $H^1(F,G)$ in literature.} gives rise to an $F$-form of $\bmG\otimes_F F^{\sep}$ through the Galois descent, where $G$ is the even part of $\bmG$.
	
	Shibata established a uniform classification theorem of irreducible representations of split quasi-reductive algebraic supergroups through the Borel--Weil construction. For our statement of the second main theorem below, let us explain his theorem briefly. Take a maximal torus $H$ of the even part $G$ of a quasi-reductive algebraic supergroup $\bmG$ and a BPS-subgroup $\bmB'\subset \bmG\otimes_F F^{\sep}$ containing $H\otimes_F F^{\sep}$. Then there is a bijection between the quotient $\Irr_\Pi \bmG\otimes_F F^{\sep}$ of the set $\Irr \bmG\otimes_F F^{\sep}$ of irreducible representations of $\bmG\otimes_F F^{\sep}$ by the parity change and a set $X^\flat(H\otimes_F F^{\sep})$\footnote{This was originally expressed $\Lambda^\flat$ in Shibata's notation.} of dominant characters of $H\otimes_F F^{\sep}$.
	
	We wish to compute the Galois twists of irreducible representations over $F^{\sep}$ in terms of $X^\flat(H\otimes_F F^{\sep})$.
	
	\begin{thm}[Corollaries \ref{cor:superpure}, \ref{cor:h_1=0case}, \ref{cor:h_1=0case_beta}, Remark \ref{rem:BT}, Theorem \ref{thm:BT}]
		Let $\bmG$ be a quasi-reductive algebraic supergroup over $F$. We fix a maximal torus $H$ of the even part $G$ of $\bmG$ and a BPS-subgroup $\bmB'\subset \bmG\otimes_F F^{\sep}$ containing $H\otimes_F F^{\sep}$. Moreover, assume the following conditions:
		\begin{enumerate}
			\renewcommand{\labelenumi}{(\roman{enumi})}
			\item For each $\sigma\in\Gamma$, there exists a lift $w_\sigma$ of the Weyl group of
			\[(G\otimes_F F^{\sep},H\otimes_F F^{\sep})\]
			such that ${}^\sigma\bmB'=w_{\sigma} \bmB'w^{-1}_{\sigma}$.
			\item The odd part of the Lie superalgebra of $\bmG$ has no nontrivial $H$-fixed vectors.
		\end{enumerate}
		Here ${}^\sigma\bmB'$ is the Galois twist of $\bmB'$ by $\sigma$. Then:
		\begin{enumerate}
			\item The quotient map $\Gamma\backslash\Irr\bmG\otimes_F F^{\sep}\to \Gamma\backslash\Irr_{\Pi}\bmG\otimes_F F^{\sep}$ is two-to-one.
			\item The action of $\Gamma$ on $X^\flat(H\otimes_F F^{\sep})$ induced from Shibata's bijection
			\[\Irr_{\Pi}\bmG\otimes_F F^{\sep}\cong X^\flat(H\otimes_F F^{\sep})\]
			agrees with Tits' $\ast$-action.
			\item For $\lambda\in X^\flat(H\otimes_F F^{\sep})$, a corresponding irreducible representation of the quasi-reductive algebraic supergroup
			$\bmG\otimes_F F^{\sep}$
			is invariant under the Galois action up to isomorphism if and only if the $\ast$-action fixes $\lambda$.
			\item Every irreducible representation $V$ of $\bmG$ admits no nontrivial map to its parity change.
			\item Let $V$ be an irreducible representation of $\bmG$ corresponding to the character $\lambda\in X^\flat(H\otimes_F F^{\sep})$ which is invariant under the $\ast$-action. Then the division algebra of the endomorphisms of $V$ agrees with that of the endomorphisms of the irreducible representation of $G$ corresponding to that of $G\otimes_F F^{\sep}$ of highest weight $\lambda$. These are computed numerically by Theorem \ref{thm:BT} (or essentially \cite[Proof of 12.6. Proposition]{MR207712}).
		\end{enumerate}
	\end{thm}
	
	We assumed (i) to relate Shibata's classifications for $\bmB'$ and its Galois twists ${}^\sigma \bmB'$. In fact, not all the BPS-subgroups are $G(F^{\sep})$-conjugate in contrast to the Borel subgroups in the non-super setting. When Shibata gave a talk in Suita Representation Theory Seminar, Yoshiki Oshima asked how Shibata's classifications for different BPS-subgroups are related. In the basic Lie superalgebra setting, \cite[Proposition 1.32]{MR3012224} gives us the answer to this question for the Borel subalgebras in the sense of Cheng and Wang. See also Proposition \ref{prop:hw_odd_reflection}. In Section \ref{sec:odd_reflection}, we use it to remove (i) for basic quasi-reductive algebraic supergroups of the main type (Definition \ref{defn:basic}). Condition (ii) is technical, but we assumed it to give a complete classification of irreducible representations of $\bmG$ and an explicit formula to the division (super)algebras of their endomorphisms. It is still interesting to remove the condition (ii) since we have $F$-forms of the split queer groups over $F^{\sep}$ for typical examples which do not satisfy (ii). If $\bmG$ is split, the classification was achieved by Shibata in \cite[Theorem 4.12]{MR4039427}. For this, he essentially related the endomorphisms of the irreducible representations of $\bmG$ to those of simple left modules over certain Clifford superalgebras. In this paper, we record it and reprove the complete classification in terms of Wall's work \cite{MR0167498} (Proposition \ref{prop:split}). In this approach, we determine the division superalgebras of endomorphims of the irreducible representations. Then \cite[Theorem 4.12]{MR4039427} is derived as its consequence. In the last two small sections \ref{sec:ast-trivial} and \ref{sec:q(p,q)}, we give explicit numerical solutions to the problems of classification of irreducible representations and determination of the division superalgebras of their endomorphisms in some special cases, particularly for some non-split real forms of the complex queer groups.
	
	\section*{Acknowledgments}
	
	Thanks Taiki Shibata for helpful discussions and careful reading of a draft of this paper. The author is also grateful to him for his generosity in allowing us to change the notion of quasi-reductive algebraic supergroups.
	
	The author is also indebted to Professor Hisayosi Matumoto for an insightful course on Lie superalgebras at University of Tokyo. He learned a lot on them from it.
	
	This work was supported by JSPS KAKENHI Grant Numbers 21J00023 and 22KJ2045.

	\section{Notation}\label{sec:notation}
	
	Let $\bZ$ denote the ring of integers, and $\bZ_{\geq 0}$ be its subset of nonnegative integers.
	
	For a nonnegative integer $m$, let $\binom{-}{m}$ denote the binomial coefficient. That is, for an integer $n$, we set
	\[\binom{n}{m}=\frac{n(n-1)\cdots(n-m+1)}{m(m-1)\cdots 1}.\]
	
	For an object $X$ of a category, we denote its identity map by $\id_X$.
	
	For $n\in\bZ_{\geq 0}$, let $I_n$ be the unit matrix of size $n$.
	
	Unless specified otherwise, $F$ will be a field of characteristic not two for a base. We write $F^\times$ for the abelian group of units of $F$. Set $(F^\times)^2$ as the subgroup of square elements of $F^\times$, i.e.,
	$(F^\times)^2=\{a\in F^\times:~\exists b\in F^\times,~a=b^2\}$.
	Let $F^{\sep}$ be a separable closure of $F$, and $\Gamma$ be the absolute Galois group of $F$. Write $\bar{F}$ for an algebraic closure of $F$ containing $F^{\sep}$. We remark that $\bar{F}=F^{\sep}$ if $F$ is perfect.
	
	We denote the Brauer group of $F$ by $\Br(F)$. For each central simple algebra $A$, we refer to its similarity (Morita equivalence) class as $[A]$.
	
	For $a,b\in F^\times$, let $(a,b)$ denote the associated quaternion algebra over $F$. Namely, $(a,b)$ is an $F$-algebra with a basis $1,i,j,ij$ which satisfies the following relations:
	\[\begin{array}{ccc}
		i^2=a,&j^2=b,&ij=-ji.
	\end{array}\]
	This is known to be a central simple algebra over $F$ (see \cite[Corollary 7.1.2]{MR4279905} for example). We write $(-1,-1)=\bH$ if $F=\bR$. For conventional reasons, we write $(a,b)=[(a,b)]\in \Br(F)$ (the Hilbert symbol).
	
	For a vector space $V$ over $F$, we write $\dim_F V$ for its dimension.
	
	For a basic formalism on affine group superschemes, we refer to \cite[Section 2]{MR4039427}. Let us collect some quickly for convenience. A vector superspace is a $\bZ/2\bZ$-graded vector space $V=V_{\bar{0}}\oplus V_{\bar{1}}$. We call $V_{\bar{0}}$ (resp.~$V_{\bar{1}}$) the even (resp.~odd) part of $V$. The degrees are traditionally called parities in this setting. For a homogeneous element $v\in V$, we denote its parity by $|v|$. Homomorphisms between vector superspaces in this paper are promised to respect the parity unless specified otherwise. The vector superspaces form a symmetric monoidal category for the tensor product $-\otimes_F-$, whose symmetry constraint
	$C_{V,W}:V\otimes_F W\cong W\otimes_F V$
	is defined by $v\otimes w\mapsto (-1)^{|v| |w|}w\otimes v$, where $v$ and $w$ are homogeneous elements of $V$ and $W$ respectively. Its (commutative) algebra objects are called (commutative) superalgebras. Let $\CSAlg_F$ denote the category of commutative superalgebras over $F$. Likewise, one can define left supermodules over them, supercoalgebras, and right supercomodules. It follows by generalities of symmetric monoidal categories that $\CSAlg_F$ is coCartesian monoidal for $-\otimes_F-$. Its cogroup objects are called commutative Hopf superalgebras. In particular, they are canonically equipped with the structure of supercoalgebras. An affine group superscheme is a representable group copresheaf over $\CSAlg_F$. Throughout this paper, we will denote copresheaves over $\CSAlg_F$ by bold capital letters. Yoneda's lemma implies that the category of affine group superschemes is equivalent to the opposite category to that of commutative Hopf superalgebras. To define a canonical quasi-inverse to the Yoneda embedding, define copresheaves $\bA^{1|0}$ and $\bA^{0|1}$ over $\CSAlg_F$ by
	\[\begin{array}{cc}
		\bA^{1|0}(R)=R_{\bar{0}},
		&\bA^{0|1}(R)=R_{\bar{1}}.
	\end{array}\]
	For a copresheaf $\bmX$ over $\CSAlg_F$, let $F[\bmX]_{\bar{0}}$ (resp.~$F[\bmX]_{\bar{1}}$) be the set of morphisms from $\bmX$ to $\bA^{1|0}$ (resp.~$\bA^{0|1}$). Then $F[\bmX]_{\bar{0}}\times F[\bmX]_{\bar{1}}$ is naturally endowed with the structure of a commutative superalgebra (cf.~\cite[The paragraph below (3) in Chapter I, Section 1.3]{MR2015057}). The assignment $\bmX\mapsto \Hom_{\CSAlg_F}(F[\bmX],(-))$ gives a left adjoint functor to the Yoneda embedding, where $\Hom_{\CSAlg_F}(-,-)$ is the Hom bi-functor of $\CSAlg_F$ (the affinization). If $\bmG$ is a group copresheaf, $F[\bmG]$ is naturally endowed with the structure of a commutative Hopf superalgebra. This gives rise to a left adjoint functor to the Yoneda embedding of the category of commutative Hopf superalgebra into that of group copresheaves (the affinization). In particular, a quasi-inverse is given by its restriction. If $\bmG$ is representable, $\bmG$ is represented by $F[\bmG]$.

	For an affine group superscheme (of finite type), we write the corresponding German small letter for its Lie superalgebra (see \cite[Section 2.2]{MR4039427}). Its Lie bracket will be denoted by $[-,-]$. We define the even part of an affine group superscheme as the restriction of the group copresheaf to the category of commutative algebras, and denote it by the corresponding capital alphabet. We remark that $G=\Spec F[\bmG]/(F[\bmG]_{\bar{1}})$, where $(F[\bmG]_{\bar{1}})$ is the ideal of $F[\bmG]$ generated by $F[\bmG]_{\bar{1}}$. If $\bmG$ is of finite type, $\fg_{\bar{0}}$ is naturally identified with the Lie algebra of $G$ (\cite[Section 2.3]{MR4039427}). For a linear algebraic supergroup $\bmG$, we denote its super-hyperalgebra (also known as the distribution superalgebra) by $\hy(\bmG)$ (see \cite[Section 2.3]{MR4039427} for details). We remark that if the characteristic of the base field $F$ is zero, $\hy(\bmG)$ coincides with the universal enveloping superalgebra of $\fg$ (see \cite[Remark 2]{MR3000482}). A representation of an affine group superscheme $\bmG$ is a homomorphism of group copresheaves from $\bmG$ to the group copresheaf of automorphisms of a vector superspace. According to Yoneda's lemma, the category of representations of $\bmG$ is equivalent to that of supercomodules over the corresponding Hopf superalgebra $F[\bmG]$. If $\bmG$ is a linear algebraic supergroup, we denote its adjoint representation by $\Ad$. For a field extension $F'/F$, the base change functor of vector superspaces is symmetric monoidal. Hence we naturally and compatibly define the base change functors $-\otimes_FF'$ for the objects above.

	For an affine group superscheme $\bmG$ over $F$ and a representation $V'$ of $\bmG\otimes_F F^{\sep}$, we denote its restriction of the scalar to $F$ by $\Res_{F^{\sep}/F} V'$.
	
	For a vector superspace $V$, we will denote its parity change by $\Pi V$, i.e.,
	\[\begin{array}{cc}
		(\Pi V)_{\bar{0}}=V_{\bar{1}},&(\Pi V)_{\bar{1}}=V_{\bar{0}}.
	\end{array}\]
	We define $\Pi$ for representations in the standard manner. For vector superspaces $V,W$ over $F$, we denote the vector space of parity-preserving linear maps from $V$ to $W$ by $\Hom_F(V,W)$. The symmetric monoidal category of vector superspaces is closed for
	\begin{flalign*}
		&\bHom_F(V,W)=\bHom_F(V,W)_{\bar{0}}\oplus \bHom_F(V,W)_{\bar{1}};\\
		&\bHom_F(V,W)_{\bar{0}}=\Hom_F(V,W),\\
		&\bHom_F(V,W)_{\bar{1}}=\Hom_F(V,\Pi W).
	\end{flalign*}
	We call $\bHom_F(V,W)$ the vector superspace of linear maps from $V$ to $W$ by convention. We will use similar terminologies later. If $V=W$, we write
	\[\begin{array}{cc}
		\End_F(V)=\Hom_F(V,V),&\bEnd_{F}(V)=\bHom_F(V,V).
	\end{array}\]
	We note that the close structure gives rise to a ``composition map''\footnote{Indeed, the category of vector superspaces is enriched by itself. Its composition law is given by this map. See \cite[Section 1.6]{MR2177301} for details.}
	\begin{equation}
		\bHom_F(V,W)\otimes_F\bHom_F(U,V)\to \bHom_F(U,W)\label{eq:enriched_composition}
	\end{equation}
	for vector superspaces $U,V,W$. For a superalgebra $A$ and left $A$-supermodules $U,V$, we define a vector superspace $\bHom_A(U,V)$ by the equalizer of the two maps $\bHom_F(U,V)\rightrightarrows\bHom_F(A\otimes_F U,W)$ determined by
	\[\begin{split}
		\bHom_F(U,V)\otimes_F A\otimes_F U
		&\overset{C_{\bHom_F(U,V),A}\otimes_FU}{\cong} A\otimes_F \bHom_F(U,V)\otimes_F U\\
		&\overset{A\otimes_F U}{\to} A\otimes_F V\\
		&\overset{a}{\to} V,
	\end{split}\]
	\[\bHom_F(U,V)\otimes_F A\otimes_F U
	\overset{a}{\to} \bHom_F(U,V)\otimes_F U
	\overset{\ev}{\to} V,
	\]
	where $\ev$ and $a$ are the evaluation and action maps respectively. Explicitly, we have 
	\[\bHom_A(U,V)_{\bar{i}}=\{f\in\bHom_F(U,V_{\bar{i}}):~f(au)=(-1)^{|f||a|}af(u)\}\]
	for $i\in\{0,1\}$, where $a\in A$ and $u\in U$ run through all homogeneous elements. This enriches the usual Hom spaces in the sense that 
	$\Hom_A(U,V)\coloneqq\bHom_A(U,V)_{\bar{0}}$ is the space of $A$-linear maps from $U$ to $V$.
	We write $\End_A(U)=\Hom_A(U,U)$, and $\bEnd_A(U)=\bHom_A(U,U)$.
	
	For representations $V,W$ of an affine group superscheme, we write $V\overset{\boldsymbol{\cdot}}{\cong} W$ if $V$ is isomorphic to $W$ or $\Pi W$.
	
	Let $\bmG$ be an affine group superscheme over $F$. For representations $V,W$ of $\bmG$, let $\Hom_{\bmG}(V,W)$ denote the vector space of homomorphisms from $V$ to $W$. We define a vector superspace $\bHom_\bmG(V,W)$ by $\bHom_\bmG(V,W)_{\bar{0}}=\Hom_\bmG(V,W)$ and $\bHom_\bmG(V,W)_{\bar{1}}=\Hom_\bmG(V,\Pi W)$. We set
	\[\begin{array}{cc}
		\End_{\bmG}(V)=\Hom_{\bmG}(V,V),&\bEnd_{\bmG}(V)=\bHom_{\bmG}(V,V).
	\end{array}\]
	If we are given a representation $V$ of $\bmG$, we will regard $\bEnd_{\bmG}(V)$ as a superalgebra for the map \eqref{eq:enriched_composition}. We denote the set of isomorphism classes of irreducible representations of $\bmG$ by $\Irr \bmG$. This is naturally equipped with an involution for $\Pi$. Let $\Irr_{\Pi}\bmG$ be the quotient of $\Irr\bmG$ by this action. For a representation $V'$ of $\bmG\otimes_F F^{\sep}$, we define its Galois twist ${}^\sigma V'$ for $\sigma\in\Gamma$ as the base change of $V'$ by $\sigma$. If $F=\bR$, we write $\bar{V}'$. We will apply similar notations to affine group superschemes and their morphisms. In this way, we regard $\Irr\bmG\otimes_F F^{\sep}$ and $\Irr_{\Pi}\bmG\otimes_F F^{\sep}$ as $\Gamma$-sets.
	
	For a homomorphism $\bmH\to \bmG$ of affine group superschemes, we denote $\Ind^{\bmG}_{\bmH}$ for the right adjoint functor to the restriction functor from the category of representations of $\bmG$ to that of $\bmH$ (see \cite[Appendix A.3]{MR4039427} for existence). When $\bmG$ is trivial, we denote $H^0(\bmH,-)=\Ind^{\bmG}_{\bmH}(-)$. For each representation $V$ of $\bmH$, $H^0(\bmH,V)$ is called the $\bmH$-invariant part of $V$. Explicitly, this is given by
	$H^0(\bmH,V)=\bHom_{\bmH}(F,V)$.
	
	For a finite extension $F'/F$ of fields of characteristic not two and an affine group superscheme $\bmG'$ over $F'$, we define a group copresheaf on the category of commutative superalgebras over $F$ by
	$(\Res_{F'/F}\bmG')(A)=\bmG'(A\otimes_F F')$.
	We call it the Weil restriction. This is right adjoint to the base change functor
	$-\otimes_F F'$.
	As in the classical setting, it is easy to show that $\Res_{F'/F} \bmG'$ is representable, and that $\Res_{F'/F} \bmG'$ is of finite type if so is $\bmG'$ (cf.~\cite[Section 7.6]{MR1045822}).
	
	We express a block diagonal matrix
	\[\left(\begin{array}{cccc}
		A_1 &  &  &  \\
		& A_2 &  &  \\
		&  & \ddots &\\
		&&& A_n
	\end{array}\right)\]
	by $\diag(A_1,A_2,\ldots,A_n)$. For a nonnegative integer $n$, let $I_n$ be the unit matrix of size $n$. Set
	\[\begin{array}{cc}
		J_n=\left(\begin{array}{cc}
			0 & I_n \\
			-I_n & 0
		\end{array}\right),&\Pi^{\spl}_n=\left(\begin{array}{cc}
			0 & I_n \\
			I_n & 0
		\end{array}\right).
	\end{array}\]

	We also follow \cite[Section 3.1]{MR2392322}, \cite[Section 3]{MR1973576}, and \cite[Section 5.3]{MR4039427} to collect some basic examples of algebraic supergroups and related notions here. Let us slightly refine notations because we will give some non-split real forms later in the standard fashion of the theory of Lie groups (see Example \ref{ex:orthosymp}). For nonegative integers $m,n$, let $\GL_{m|n}$ be the linear algebraic supergroup of automorphisms of $F^m\oplus \Pi F^n$. For each commutative superalgebra $A$, $\GL_{m|n}(A)$ is identified with the group of square matrices of size $(m+n)$ of the form
	\[\left(\begin{array}{cc}
		a & b \\
		c & d
	\end{array}\right)\]
	with the usual multiplication law, where $a,d$ are invertible square matrices with $A_{\bar{0}}$-entries of size $m$, $n$ respectively, and $b,c$ are matrices with $A_{\bar{1}}$-entries of type $m\times n$ and $n\times m$ respectively (see \cite[Lemma 3.6.1]{MR2069561} for the invertibility hypothesis). We define an anti-automorphism $(-)^{ST}$ on $\GL_{m|n}$ by
	\[\left(\begin{array}{cc}
		a & b \\
		c & d
	\end{array}\right)\mapsto\left(\begin{array}{cc}
		a^T & c^T \\
		-b^T & d^T
	\end{array}\right),\]
	where $(-)^T$ is the usual transpose. This is called the supertranspose. We define a map $\Pi:\GL_{m|n}\to\GL_{n|m}$ by
	\[\left(\begin{array}{cc}
		a & b \\
		c & d
	\end{array}\right)\mapsto\left(\begin{array}{cc}
		d & c \\
		b & a
	\end{array}\right).\]
	We call it the parity change. We also define the superdeterminant (also known as Berezinian) as
	\[\Ber:\GL_{m|n}\to \GL_{1|0};~\left(\begin{array}{cc}
		a & b \\
		c & d
	\end{array}\right)\mapsto\det(a-bd^{-1}c)\det d^{-1}.\]
	Write $\SL_{m|n}$ for its kernel. We define linear algebraic supergroups $\SpO_{2n|2m+1}$, $\SpO_{2n|2m}$ by
	\[\SpO_{2n|2m+1}(A)=\left\{g\in\SL_{2n|2m+1}(A):~
	g^{ST}J^{\spl}_{n|2m+1}g=J^{\spl}_{n|2m+1}\right\},\]
	\[\SpO_{2n|2m}(A)=\left\{g\in\SL_{2n|2m}(A):~
	g^{ST}J^{\spl}_{n|2m}g=J^{\spl}_{n|2m}\right\},\]
	where
	$J^{\spl}_{n|2m+1}=\diag(J_n,\Pi^{\spl}_m,1)$ and
	$J^{\spl}_{n|2m}=\diag(J_n,\Pi^{\spl}_m)$.
	We call them the split orthosymplectic supergroups.
	We define the $n$th split queer supergroup $\Qq_n$ as 
	\[\Qq_n(A)=\left\{\left(\begin{array}{cc}
		a & b \\
		c & d
	\end{array}\right)\in\GL_{n|n}(A):~a=d,~b=-c\right\}.\]
	We define the $n$th split periplectic supergroup $\Pp_n$ as 
	\[\Pp_n(A)=\left\{g\in\GL_{n|n}(A):~g^{ST}\Pi^{\spl}_n g=\Pi^{\spl}_n\right\}.\]

	Define an involution $\bar{}$ on $\Res_{\bC/\bR} \GL_{m|n}$ by the complex conjugation. Let $(-)^\ast$ denote the composition of the supertranspose and the complex conjugation. We call it the superadjoint. We also define another automorphism $\delta$ of $\Res_{\bC/\bR} \GL_{m|n}$ by
	\[\delta:\left(\begin{array}{cc}
		a & b \\
		c & d
	\end{array}\right)\mapsto \left(\begin{array}{cc}
		a & \sqrt{-1} b \\
		-\sqrt{-1} c & d
	\end{array}\right).\]

	For a split torus $H$, we denote its character group by $X^\ast(H)$.
	
	For $n\geq0$, let $\fS_n$ denote the $n$th symmetric group.

	\section{Division superalgebra of endomorphisms}
	
	\subsection{Preliminaries}
	
	Following \cite{MR0167498, MR1701598}, let us recall some basic notions on superalgebras:
	
	\begin{defn}
		\begin{enumerate}
			\item A superalgebra $A$ over $F$ is called central if the even elements centralizing $A$ are only scalars, i.e.,
			$\{a\in A_{\bar{0}}:~\forall b\in A,~ab=ba\}=F$.
			\item A superalgebra $A$ is called simple if it is nonzero and admits no proper two-sided superideals.
			\item A superalgebra $D$ is called a division superalgebra if every nonzero homogeneous element of $D$ is invertible.
			\item For central simple superalgebras $A,A'$, we say $A$ and $A'$ are similar if $\bEnd_A(S)\cong\bEnd_{A'}(S')$, where $S$ and $S'$ are finite dimensional simple left modules over $A$ and $A'$ respectively. This is independent of the choices of $S$ and $S'$ since every central simple superalgebra $B$ admits a unique finite dimensional simple left $B$-supermodule up to isomorphism and parity change. For each central simple algebra $A$, we denote its similarity class by $[A]$.
			\item We denote the set of similarity classes of central simple superalgebras over $F$ by $\BW(F)$. This is an abelian group for the tensor product. We call $\BW(F)$ the Brauer--Wall group.
		\end{enumerate}
	\end{defn}
	
	\begin{ex}[Clifford superalgebra]
		Let $(V,q)$ be a quadratic space, namely, a pair of a finite dimensional vector space $V$ and a quadratic form $q$ on $V$. Then define an $F$-algebra $C(V,q)$ as the algebra generated by $V$ with the relations
		$x^2=q(x)$ for all $x\in V$. We generate $V$ as a subspace of odd elements to obtain the structure of a superalgebra on $C(V,q)$. We call $C(V,q)$ the Clifford superalgebra of $(V,q)$. This superalgebra is known to be central simple if $q$ is nondegenerate (\cite[Theorem 4]{MR0167498}). In this paper, let $\bar{C}(V,q)$ denote the maximal semisimple quotient of $C(V,q)$. According to \cite[Lemma B.1]{MR4039427}, $\bar{C}(V,q)$ is a Clifford superalgebra of a nondegenerate quadratic space. 
	\end{ex}

	For later applications, let us note:
	
	\begin{prop}[{\cite[Theorem 3]{MR0167498}}, {\cite[Lemma 3.7]{MR1701598}}]
		There is a natural bijection
		\begin{equation}
			\BW(F)\cong \{+,-\}\times (F^\times)/(F^\times)^2\times \Br(F).\label{eq:bw}
		\end{equation}
		The transferred group structure on the right hand side through this bijection is given by
		\[\begin{array}{c}
			(+,a,D)(+,a',D')=(+,aa',DD'(a,a'))\\
			(+,a,D)(-,a',D')=(-,aa',DD'(a,-a'))\\
			(-,a,D)(-,a',D')=(+,-aa',DD'(a,a')).
		\end{array}\]
	\end{prop}
	
	\begin{ex}\label{ex:BW_inverse}
		For a central simple superalgebra $A$, $[A]^{-1}$ is represented by the opposite superalgebra to $A$ (\cite[(3.3.2)]{MR1701598}). 
		For $a\in F^\times/(F^\times)^2$ and $D\in\Br(F)$, we correspondingly have
		\[\begin{array}{cc}
			(+,a,D)^{-1}=(+,a,D^{-1}(a,-1)),
			&(-,a,D)^{-1}=(-,-a,D^{-1})
		\end{array}\]
		(use \cite[I, Exercises 4 of Chapter 2 and Chapter 8, Example 8.2.2]{MR4279905} for the part of the Brauer group).
	\end{ex}
	
	\begin{ex}\label{ex:cliff_op}
		Let $(V,q)$ be a quadratic space. Then $v\mapsto v$ determines an isomorphism from $C(V,-q)$ to the opposite superalgebra to $C(V,q)$. In particular, we have $[\bar{C}(V,q)]^{-1}=[\bar{C}(V,-q)]$. 
	\end{ex}
	
	\begin{ex}\label{ex:real}
		Put $F=\bR$. Then both $\Br(\bR)$ and $\bR^\times/(\bR^\times)^2$ are identified with $\{\pm 1\}$ for
		\[\begin{array}{ccc}
			\Br(F)&\cong&\{\pm 1\};\\
			\left[\bR\right]&\mapsto & 1\\
			\left[\bH\right]&\mapsto & -1,\\
			\bR^\times/(\bR^\times)^2&\cong&\{\pm 1\};\\
			1&\reflectbox{$\mapsto$} & 1\\
			-1&\reflectbox{$\mapsto$} & -1.
		\end{array}\]
		Under the first identification, we also have
		\[(a,b)=\begin{cases}
			-1&(a<0\mathrm{~and~}b<0)\\
			1&(\mathrm{otherwise}).
		\end{cases}\]
	\end{ex}

	For the convenience, we explain the construction of the bijection \eqref{eq:bw}. Let $A$ be a central simple superalgebra. We put $\epsilon_A\coloneqq +$ if $A$ is central and simple as a non-super (= ordinary) algebra; Otherwise, we put $\epsilon_A\coloneqq -$ (\cite[Lemma 4 and the sentence above Lemma 5]{MR0167498}. For the case $A=A_{\bar{0}}$, replace $A$ with $A\otimes_F \bEnd_F(V)$ for any finite dimensional vector superspace $V$ over $F$ with nonzero even and odd parts, or use \cite[Proof of Lemma 3.7]{MR1701598}). We remark that the equality $\epsilon_A=-$ holds if and only if $A_{\bar{0}}$ is a central simple algebra and $A_{\bar{1}}\neq 0$. If $\epsilon_A=+$ and $A_{\bar{1}}\neq 0$, one can and does choose $u\in A_{\bar{0}}$ such that $u^2\in F^\times$ and that the center of $A_{\bar{0}}$ is $F\oplus Fu$ as an $F$-vector space (\cite[Lemma 6]{MR0167498}). We put $u=1$ if $A=A_{\bar{0}}$. We note that if $A=A_{\bar{0}}$ and if we are given any finite dimensional vector superspace $V$ over $F$ with nonzero even and odd parts, we may put $u=1\otimes(\id_{V_{\bar{0}}},-\id_{V_{\bar{1}}})$ for $A\otimes_F \bEnd_F(V)$. In particular, we get $u^2=1$ in this case. If $\epsilon_A=-$, one can and does choose a nonzero element $u\in A_{\bar{1}}$ centralizing $A_{\bar{0}}$ (\cite[Lemma 5]{MR0167498}). Then we have $u^2\in F^\times$ (\cite[Lemma 5]{MR0167498}). The map is now given by
	\[A\mapsto\begin{cases}
		(+,u^2,[A])&(\epsilon_A=+)\\
		(-,u^2,[A_{\bar{0}}])&(\epsilon_A=-).
	\end{cases}\]
	(\cite[Lemmas 4, 5, and 6]{MR0167498}).
	
	\begin{conv}
		By abuse of notation, we write $[A]=(+,1,[A])$ for a central simple algebra $A$. This is to regard $A$ as a central simple superalgebra with trivial odd part. In terms of $\Br(F)$ and $\BW(F)$, this is to identify $(+,1,[A])$ with the image of $[A]\in\Br(F)$ along the injective map $\Br(F)\hookrightarrow \BW(F);~[B]\mapsto (+,1,[B])$.
	\end{conv}
	
	\begin{ex}[\cite{MR0167498}, see also {\cite[Section 3.9]{MR1701598}}]\label{ex:realbrauer-wall}
		Put $F=\bR$. There are eight elements in $\BW(\bR)$ by Example \ref{ex:real}. Representing central division superalgebras are determined by the following assignment:
		\[\begin{array}{c}
			[\bR]\mapsto(+,1,[\bR])=(+,1,1)\\
			\left[\bR\oplus \bR\epsilon\right]\mapsto (-,1,[\bR])=(-,1,1)\\
			\left[\bC\oplus \bC\epsilon\right]
			\mapsto (+,-1,[\bC\oplus \bC\epsilon])=(+,-1,[\bR])=(+,-1,1)\\
			\left[\bH\oplus \bH\delta\right]\mapsto (-,-1,[\bH])=(-,-1,-1)\\
			\left[\bH\right]\mapsto (+,1,\left[\bH\right])=(+,1,-1)\\
			\left[\bH\oplus \bH\epsilon\right]\mapsto (-,1,[\bH])=(-,1,-1)\\
			\left[\bC\oplus \bC\delta\right]\mapsto
			(+,-1,	\left[\bC\oplus \bC\delta\right])=(+,-1,[\bH])=(+,-1,-1)\\
			\left[\bR\oplus \bR\delta\right]\mapsto (-,-1,[\bR])=(-,-1,1),
		\end{array}\]
		where $\epsilon$ and $\delta$ are odd elements of $\epsilon^2=-\delta^2=1$. Moreover, the conjugation on the complex numbers by $\epsilon$ and $\delta$ is the complex conjugation in $\bC\oplus \bC\epsilon$ and $\bC\oplus \bC\delta$; $\epsilon$ and $\delta$ commute with $\bH$ in $\bH\oplus \bH\epsilon$ and $\bH\oplus \bH\delta$.
	\end{ex}
	
	\begin{note}
		For a nonnegative integer $n$, put $-^n=+$ (resp.~$-^n=-$) if $n$ is even (resp.~$n$ is odd).
	\end{note}
	
	\begin{ex}[{\cite[Theorem 4 and below it]{MR0167498}}]\label{ex:Clifford}
		Let $(V,q)$ be a quadratic space with $q$ nondegenerate. Then the Clifford superalgebra $C(V,q)$ is central simple superalgebra. To compute its class, identify $V$ with $F^n$ ($n\geq 0$) for an orthogonal basis of $V$ which exists to express $q$ as
		$q(x_1,x_2,\ldots,x_n)=\sum_{i=1}^n a_ix^2_i$
		for some $a_1,a_2,\ldots,a_n\in F^\times$. If we write
		$[C(F^n,q)]=(\epsilon,a,D)$,
		the right hand side is given by the following formulas:
		\[\begin{array}{ccc}
			\epsilon=-^n,&a=(-1)^{\binom{n}{2}} \prod_{i=1}^n a_i,
			&D=\prod_{i<j}(a_i,a_j)(\prod_{i=1}^n (-1,a_i))^{\binom{n-1}{2}} (-1,-1)^{\binom{n+1}{4}}.
		\end{array}\]
		The value $a$ is called the signed discriminant (determinant). We remark that the orthogonal basis was also used in \cite{MR0167498} for the proof of the fact that $C(V,q)$ is central simple.
	\end{ex}

	\begin{rem}\label{rem:D_1=0}
		For a central division superalgebra $D$, $D_{\bar{1}}=0$ if and only if the equality $\left[D\right]=(+,1,[A])$ holds
		for some central simple algebra $A$ (recall the definition of the similarity for the ``if'' direction).
	\end{rem}
	
	\subsection{Classification of irreducible representations}
	In this section, we classify irreducible representations over a field $F$ of characteristic not two in terms of those over $F^{\sep}$ and the Galois action. We also discuss the central property of the division (super)algebras of endomorphisms.
	
	Let $\bmG$ be an affine group superscheme over $F$.
	
	\begin{defn}\label{defn:pure}
		\begin{enumerate}
			\item An irreducible representation $V$ of $\bmG$ is called quasi-rational (resp.~super quasi-rational) if $\End_{\bmG} V$
			is a central division algebra (resp.~$\bEnd_{\bmG} V$ is a central division superalgebra). We denote the set of isomorphism classes of quasi-rational (resp.~super quasi-rational) irreducible representations by $\Irr_{\qrat}\bmG$ (resp.~$\Irr_{\sqrat}\bmG$).
			\item We call an irreducible representation $V'$ of $\bmG\otimes_F F^{\sep}$ weakly quasi-rational (resp.~weakly super quasi-rational) if
			${}^\sigma V'\cong V'$ (resp.~ ${}^\sigma V'\overset{\boldsymbol{\cdot}}{\cong} V'$)
			for every $\sigma\in\Gamma$. We say $V'$ is quasi-rational (resp.~super quasi-rational) if it is weakly quasi-rational (resp.~weakly super quasi-rational) and
			\[\End_{\bmG\otimes_F F^{\sep}}(V')=F^{\sep}\id_{V'}.\]
		\end{enumerate}
	\end{defn}
	
	\begin{rem}
		The condition $\End_{\bmG\otimes_F F^{\sep}}(V')=F^{\sep}\id_{V'}$ holds true if $F$ is perfect (Schur's lemma). In particular, weak (super) quasi-rationality and (super) quasi-rationality are equivalent in this case.
	\end{rem}
	
	According to Schur's lemma, $\bEnd_{\bmG} V$ (resp.~$\End_{\bmG}(V)$) is a division superalgebra (resp.~a division algebra) if $V$ is irreducible. Therefore we obtain maps
	\[\beta^{\super}:\Irr_{\sqrat} \bmG\to \BW(F);~V\mapsto \bEnd_{\bmG} V\reflectbox{$\coloneqq$}\beta^{\super}_V,\]
	\[\beta:\Irr_{\qrat} \bmG\to \Br(F);~V\mapsto \End_{\bmG} V\reflectbox{$\coloneqq$}\beta_V.\]
	
	\begin{note}
		If $V$ is a super quasi-rational irreducible representation of $\bmG$, we denote $\beta^{\super}_V=(\epsilon_V,a_V,D_V)$ under the identification of \eqref{eq:bw}.
	\end{note}
	
	Quasi-rational representations of $\bmG\otimes_F F^{\sep}$ are of our interests for the descent problem of representations over $F^{\sep}$. We aim to discuss relations of the four notions in Definition \ref{defn:pure} and the maps $\beta^{\super},\beta$. On this course, we would also like to give representation theoretic characterizations of the (super) quasi-rationality and these maps.
	
	The following three assertions are immediate from the definitions:
	
	\begin{prop}\label{prop:sqrat->qrat}
		\begin{enumerate}
			\item Quasi-rational irreducible representations of $\bmG$ are super quasi-rational.
			\item (Weakly) quasi-rational irreducible representations of $\bmG\otimes_F F^{\sep}$ are (weakly) super quasi-rational.
		\end{enumerate}
	\end{prop}
	
	\begin{prop}\label{prop:pure}
		A (weakly) super quasi-rational irreducible representation $V'$ of the affine group superscheme
		$\bmG\otimes_F F^{\sep}$
		is (weakly) quasi-rational if and only if
		\begin{enumerate}
			\item[(i)] $V'\cong \Pi V'$, or 
			\item[(ii)] no $\sigma\in\Gamma$ satisfies ${}^\sigma V'\cong\Pi V'$.
		\end{enumerate}
	\end{prop}
	
	\begin{prop}\label{prop:type}
		An irreducible representation $V$ of $\bmG$ with $\Pi V\not\cong V$ is (weakly) quasi-rational if and only if it is (weakly) super quasi-rational.
	\end{prop}
	
	Let $V$ be an irreducible representation of $\bmG$, and $V'$ be an irreducible subrepresentation of $V\otimes_F F^{\sep}$. The first basic result is as follows:
	
	\begin{thm}\label{thm:pure}
		\begin{enumerate}
			\item The assignment $V\mapsto V'$ gives rise to a well-defined bijection 
			\begin{equation}
				\Irr\bmG\cong\Gamma\backslash\Irr\bmG\otimes_F F^{\sep}.\label{eq:classification}
			\end{equation}
			This also descends to
			\begin{equation}
				\Irr_\Pi\bmG\cong\Gamma\backslash\Irr_{\Pi}\bmG\otimes_F F^{\sep}.
				\label{eq:classification_mod_switch}
			\end{equation}
			\item The following conditions are equivalent:
			\begin{enumerate}
				\item[(a)] $V$ is super quasi-rational;
				\item[(b)] $V'$ is super quasi-rational.
			\end{enumerate}
			In particular, the bijection \eqref{eq:classification_mod_switch} restricts to
			$\Irr_{\Pi,\sqrat}\bmG\cong\Irr_{\Pi,\sqrat} \bmG\otimes_F F^{\sep}$.
		\end{enumerate}
	\end{thm}
	
	\begin{lem}\label{lem:bc_of_Hom}
		Let $F'$ be a field containing $F$. For representations $U,W$ of $\bmG$ with $U$ finite dimensional, we have a canonical isomorphism
		\begin{equation}
			\bHom_{\bmG}(U,W)\otimes_F F'\cong\bHom_{\bmG\otimes_F F'}(U\otimes_F F',W\otimes_F F').\label{eq:bc_Hom}
		\end{equation}
	\end{lem}
	
	\begin{proof}
		This is proved in a similar way to \cite[I.2.10 (7)]{MR2015057}.
	\end{proof}
	
	\begin{lem}\label{lem:semisimple_descent}
		Let $F'$ be field containing $F$, and $U$ be a finite dimensional representation of $\bmG$.
		\begin{enumerate}
			\item If $U\otimes_F F'$ is semisimple, so is $U$.
			\item Assume $F'/F$ to be a possibly infinite Galois extension. If $U$ is semisimple, so is $U\otimes_F F'$. 
		\end{enumerate}
	\end{lem}
	
	In a sequel, we only use (2). Part (1) will be used in the next section.

	\begin{proof}
		Part (2) is an easy consequence of the Galois descent. We prove (1) by induction on $n=\dim_F U$. The assertion is clear if $n=0$. If $n\geq 1$, pick an irreducible subrepresentation $U_1$ of $U$. The assertion follows if $U=U_1$; Otherwise, $U/U_1$ is semisimple from the induction hypothesis since semisimplicity is closed under formations of subquotients.
		
		It remains to find a retract of the inclusion $j:U_1\hookrightarrow U$ if $U\neq U_1$. Define an $F$-vector space $\GRet(j)$ by
		\[\GRet(j)=
		\{p\in\Hom_{\bmG}(U,U_1):~\exists c\in F~\mathrm{s.t.}~p\circ j=c\id_{U_1}\}.\]
		This is identified with the fiber product of the injective map
		\[(-)\circ j:\Hom_{\bmG}(U,U_1)\to\End_{\bmG}(U_1)\]
		and the scalar map $F\hookrightarrow \End_{\bmG}(U_1);~c\mapsto c\id_{U_1}$.
		We define $\GRet(F'\otimes_F j)$ in a similar way. Then Proposition \ref{prop:type} implies $\tilde{F}\otimes_F \GRet(j)\cong \GRet(\tilde{F}\otimes_F j)$. Notice that $F'\otimes_F U_1$ is semisimple. Since $U\otimes_F F'$ is nonzero and semisimple, $\GRet(j)$ is nonzero. In fact, every subrepresentation of a semisimple representation admits a retract. We now obtain a retract by normalizing a nonzero element of $\GRet(j)$. This completes the proof.
	\end{proof}

	\begin{lem}\label{lem:base_change}
		\begin{enumerate}
			\item The representation $V\otimes_F F^{\sep}$ is completely reducible.
			\item Every irreducible summand of $V\otimes_F F^{\sep}$ is isomorphic to ${}^\sigma V'$ for some $\sigma\in\Gamma$. Conversely, a copy of ${}^\sigma V'$ appears in $V\otimes_F F^{\sep}$ for every $\sigma\in\Gamma$.
		\end{enumerate}
	\end{lem}
	
	\begin{proof}
		Part (1) is a restatement of the Lemma \ref{lem:semisimple_descent} (2) with $F'=F^{\sep}$. To prove (2), choose a nonzero homomorphism $V\otimes_F F^{\sep}\to V'$, which exists by (1). It induces an injective map $V\hookrightarrow \Res_{F^{\sep}/F} V'$. Take its base change to get
		\begin{equation}
			V\otimes_F F^{\sep}\hookrightarrow (\Res_{F^{\sep}/F} V')\otimes_F F^{\sep}
			\hookrightarrow\prod_{\sigma\in\Gamma} {}^\sigma V',\label{eq:embedding}
		\end{equation}
		where $\prod_{\sigma\in\Gamma} {}^\sigma V'$ is the set-theoretic product. Since $V$ is of finite dimension, there exists a finite subset $\Gamma'\subset\Gamma$ such that the composition of \eqref{eq:embedding} with the projection to $\oplus_{\sigma\in\Gamma'} {}^\sigma V'$ is injective. The first part of (2) now follows since this composite map is a homomorphism of representations of $\bmG\otimes_F F^{\sep}$. The latter part is verified by taking the Galois twist of the embedding $V'\hookrightarrow V\otimes_F F^{\sep}$ and the identification \[{}^\sigma(V\otimes_F F^{\sep})\cong V\otimes_F F^{\sep}.\]
	\end{proof}

	\begin{rem}
		If one hopes, he or she can prove the first part of (2) within the category of representations of $\bmG\otimes_F F^{\sep}$ by taking the product of ${}^\sigma V'$ ($\sigma\in\Gamma$) in this category. This product does not agree with the set-theoretic one, but it exists by similar arguments to \cite[Corollary 2.5.6]{MR1650134} or \cite[Proposition 1.2.2]{MR2066503}.
	\end{rem}

	\begin{lem}\label{lem:central}
		Assume $F$ to be separably closed. Let $\{V_\lambda\}_{\lambda\in\Lambda}$ be a set of irreducible representations of $\bmG$ indexed by a nonempty finite set $\Lambda$. Fix $\lambda_0\in\Lambda$. Then $\bEnd_{\bmG}(\oplus_{\lambda\in\Lambda} V_\lambda)$ is central if and only if $\End_{\bmG}(V_{\lambda_0})=F\id_{V_{\lambda_0}}$ and $V_\lambda\overset{\boldsymbol{\cdot}}{\cong} V_{\lambda_0}$ for every index $\lambda\in\Lambda$.
	\end{lem}
	
	\begin{proof}
		Set
		$\Lambda_{\lambda_0}\coloneqq\{\lambda\in\Lambda:~V_\lambda\overset{\boldsymbol{\cdot}}{\cong} V_{\lambda_0}\}\ni\lambda_0$.
		Then we have
		\[\bEnd_{\bmG}(\oplus_{\lambda\in\Lambda} V_\lambda)\cong
		\bEnd_{\bmG}(\oplus_{\lambda\in\Lambda_{\lambda_0}} V_\lambda)\times 
		\bEnd_{\bmG}(\oplus_{\lambda\in\Lambda\setminus \Lambda_{\lambda_0} } V_\lambda).
		\]
		Therefore $\bEnd_{\bmG}(\oplus_{\lambda\in\Lambda} V_\lambda)$ is central only if $\Lambda=\Lambda_{\lambda_0}$. Henceforth we assume $\Lambda=\Lambda_{\lambda_0}$.
		
		Let $D=\End_{\bmG}(V_{\lambda_0})$. Set
		$n\coloneqq |\Lambda|$ and $m\coloneqq|\{\lambda\in\Lambda:~V_\lambda\cong \Pi V_{\lambda_0}\}|$.
		\begin{description}
			\item[Case I] Assume $\Pi V_{\lambda_0}\not\cong V_{\lambda_0}$. Then we have an isomorphism
			\[\bEnd_{\bmG}(\oplus_{\lambda\in\Lambda} V_\lambda)\cong D\otimes_F \bEnd_F(F^m\oplus \Pi F^{n-m}).\]
			It is straightforward that the even part of the center of $\bEnd_{\bmG}(\oplus_{\lambda\in\Lambda} V_\lambda)$ agrees with the center of $D$. In particular, the ``if'' direction holds in the present assumption. The converse follows since $F$ is separably closed. In fact, if $\bEnd_{\bmG}(\oplus_{\lambda\in\Lambda} V_\lambda)$ is central, we have $D=F$ (\cite[Chapter X, \S5, Proposition 7]{MR554237}).
			\item[Case II] Assume $V_{\lambda_0}\cong \Pi V_{\lambda_0}$. We fix an isomorphism $\varphi: V_{\lambda_0}\cong \Pi V_{\lambda_0}$. Using $\varphi$ and the condition $\Lambda=\Lambda_{\lambda_0}$, we identify $\bEnd_{\bmG}(\oplus_{\lambda\in\Lambda} V_\lambda)$ with $\bEnd_{\bmG}(V_{\lambda_0}^n)$. Set $\epsilon=\prod_{i=1}^n \varphi:V_{\lambda_0}^n\cong \Pi V_{\lambda_0}^n$. Then we have
			\[\bEnd_{\bmG}(V_{\lambda_0}^n)=\End_{\bmG}(V_{\lambda_0}^n)\oplus \epsilon\End_{\bmG}(V_{\lambda_0}^n).\]
			One can then easily see from this presentation that $\bEnd_{\bmG}(V_{\lambda_0}^n)$ is a simple superalgebra.
			
			Suppose that $\bEnd_{\bmG}(\oplus_{\lambda\in\Lambda} V_\lambda)$ is central. Its even part is isomorphic to $D\otimes_F \End_F(F^n)$. The center of $D\otimes_F \End_F(F^n)$ is isomorphic to that of $D$. If $D\otimes_F \End_F(F^n)$ is not central, the center of $D$ is a quadratic extension of $F$ (\cite[Lemma 6]{MR0167498}). Since $F$ is a separably closed field of characteristic not two, we do not have such an extension. A similar argument to Case I then implies $D=F$.
			
			Conversely, suppose $D=F$. Then one can express $\varphi\circ\Pi\varphi=a\id_{V_{\lambda_0}}$ for some $a\in F^\times$. We choose a root $\sqrt{a}\in F$. This exists since $F$ is separably closed of characteristic not two. We replace $\varphi$ with $\sqrt{a}^{-1}\varphi$ to assume that the equality $\varphi\circ\Pi\varphi=\id_{V_{\lambda_0}}$ holds. We now get
			\[\bEnd_{\bmG}(\oplus_{\lambda\in\Lambda} V_\lambda)
			\cong \End_F(F^n)\oplus \End_F(F^n)\epsilon,
			\]
			where $\epsilon$ is a central odd element with $\epsilon^2=1$. It is straightforward that this is central.
		\end{description}
	\end{proof}

	\begin{proof}[Proof of Theorem \ref{thm:pure}]
		The map \eqref{eq:classification} is well-defined from Lemma \ref{lem:base_change} (2). To see that it is injective, let $V$ and $W$ be irreducible representations of $\bmG$. Take irreducible summands $V'\subset V\otimes_F F^{\sep}$ and $W'\subset W\otimes_F F^{\sep}$ such that $V'\cong {}^\sigma W'$ for some $\sigma\in\Gamma$. We wish to prove $\Hom_{\bmG}(V,W)$ is nonzero. We may take the base change to $F^{\sep}$ by the faithfully flat descent. Since $W\otimes_F F^{\sep}$ is completely reducible, one can choose a surjective map $V\otimes_F F^{\sep}\to V'$. We use it and Lemmas \ref{lem:base_change} (2), \ref{lem:bc_of_Hom} to get a series of injective maps
		\[\begin{split}
			0&\neq \Hom_{\bmG\otimes_F F^{\sep}}(V',{}^\sigma W')\\
			&\hookrightarrow\Hom_{\bmG\otimes_F F^{\sep}}(V',W\otimes_F F^{\sep})\\
			&\hookrightarrow \Hom_{\bmG\otimes_F F^{\sep}}(V\otimes_F F^{\sep},W\otimes_F F^{\sep})\\
			&\overset{\eqref{eq:bc_Hom}}{\cong}\Hom_{\bmG}(V,W)\otimes_F F^{\sep},
		\end{split}
		\]
		which implies that $\Hom_{\bmG}(V,W)$ is nonzero as desired.
		
		Conversely, suppose that we are given an irreducible representation $W'$ of the affine group superscheme $\bmG\otimes_F F^{\sep}$.
		Choose an irreducible subrepresentation $W$ of $\Res_{F^{\sep}/F} W'$. It extends to a surjective homomorphism $W \otimes_F F^{\sep}\to W'$. Since $W \otimes_F F^{\sep}$ is completely reducible, $W\otimes_F F^{\sep}$ contains $W'$ as a summand. This completes the proof of the first part of (1). Since the map \eqref{eq:classification} respects $\Pi$ by construction, it descends to the bijection \eqref{eq:classification_mod_switch}.
		
		We next prove (2). Since $\bEnd_{\bmG} V$ is a division superalgebra, the following conditions are equivalent:
		\begin{enumerate}
			\item[(a)] $V$ is super quasi-rational.
			\item[(a)'] $\bEnd_{\bmG}(V)$ is central. 
			\item[(a)''] $\bEnd_{\bmG\otimes_F F^{\sep}} (V\otimes_F F^{\sep})$ is central
		\end{enumerate}
		(use Lemma \ref{lem:bc_of_Hom} and the faithfully flat descent for the equivalence of (a)' and (a)''). According to Lemma \ref{lem:base_change} (1), we may write $V\otimes_F F^{\sep}=\oplus_{\lambda\in\Lambda} V'_\lambda$, where $V'_\lambda$ are irreducible representations of $\bmG\otimes_F F^{\sep}$. In view of Lemma \ref{lem:central}, (a)'' holds if and only if $V'_\lambda\overset{\boldsymbol{\cdot}}{\cong} V'$ for every $\lambda\in\Lambda$. Part (2) now follows from Lemma \ref{lem:base_change} (2).
	\end{proof}

	\begin{note}
		For an irreducible representation $U$ of $\bmG\otimes_F F^{\sep}$, we denote the corresponding irreducible representation of $\bmG$ by $U_F$.
	\end{note}
	
	\begin{var}\label{var:pure}
		\begin{enumerate}
			\item The irreducible representation $V$ of $\bmG$ is quasi-rational if and only if $V'$ is quasi-rational.
			\item The bijection \eqref{eq:classification} restricts to $\Irr_{\qrat}\bmG\cong\Irr_{\qrat}\bmG\otimes_F F^{\sep}$.
		\end{enumerate}
	\end{var}
	
	If a given irreducible representation $V$ of $\bmG$ is quasi-rational, the central division algebra $\End_{\bmG}(V)$ is recovered from $\beta^{\super}_V$ through the projection $\BW(F)\to \Br(F)$. As a summary, we obtain a commutative diagram
	\[\begin{tikzcd}
		\Irr_{\Pi,\sqrat}\bmG\otimes_F F^{\sep}\ar[r,"\sim"]
		&\Irr_{\Pi,\sqrat}\bmG\ar[r, "\beta^{\super}"]&\BW(F)\ar[d]\\
		\Irr_{\qrat}\bmG\otimes_F F^{\sep}\ar[r,"\sim"]\ar[u]
		&\Irr_{\qrat}\bmG\ar[r, "\beta"]\ar[u]&\Br(F),
	\end{tikzcd}\]
	where the right vertical arrow is defined by the projection. The middle and left vertical arrows are obtained by restriction of the the canonical quotient maps
	\[\begin{array}{cc}
		\Irr\bmG\otimes_F F^{\sep}\to \Irr_{\Pi}\bmG\otimes_F F^{\sep},
		&\Irr\bmG \to \Irr_{\Pi} \bmG.
	\end{array}\]
	
	For the relation of studies of $\bEnd_{\bmG}(V)$ and $\End_{\bmG}(V)$, we note that they are equal if $\bEnd_{\bmG}(V)_{\bar{1}}=0$ (cf.~Proposition \ref{prop:type}). Let us record a criterion for this equality:
	
	\begin{prop}\label{prop:even}
		We have an isomorphism $V\cong \Pi V$ if and only if $\Pi V'\cong {}^\sigma V'$ for some $\sigma\in\Gamma$. 
	\end{prop}
	
	\begin{proof}
		If $\Pi V'\cong {}^\sigma V'$ for some $\sigma\in\Gamma$ then we have
		\[\begin{split}
			\Hom_{\bmG}(V,\Pi V)\otimes_F F^{\sep}
			&\cong \Hom_{\bmG\otimes_F F^{\sep}}(V\otimes_F F^{\sep},\Pi V\otimes_F F^{\sep})\\
			&\reflectbox{$\hookrightarrow$} \Hom_{\bmG\otimes_F F^{\sep}}(V\otimes_F F^{\sep},\Pi V')\\
			&\reflectbox{$\hookrightarrow$} \Hom_{\bmG\otimes_F F^{\sep}}({}^\sigma V',\Pi V')\\
			&\neq 0.
		\end{split}\]
		Conversely, suppose $\Pi V\cong V$. Choose a nonzero homomorphism $V\otimes_FF^{\sep}\to V'$ to get an injective map $V\hookrightarrow \Res_{F^{\sep}/F} V'$. One can see
		\[\Hom_{\bmG\otimes_F F^{\sep}}(V\otimes_F F^{\sep},\Pi V')\neq 0\]
		from the sequence
		\[0\neq \Hom_{\bmG}(V,\Pi V)
		\hookrightarrow \Hom_{\bmG}(V,\Pi\Res_{F^{\sep}/F} V')
		\cong \Hom_{\bmG\otimes_F F^{\sep}}(V\otimes_F F^{\sep},\Pi V').\]
		The assertion now follows from Lemma \ref{lem:base_change}.
	\end{proof}

	We note that the study of the division superalgebras of endomorphisms is reduced to the super quasi-rational case:
	
	\begin{prop}\label{prop:nonsuperpure}
		Let $V$ be an irreducible representation of $\bmG$. Let $F'$ be the even part of the center of $\bEnd_{\bmG}(V)$.
		\begin{enumerate}
			\item There exists a finite dimensional super quasi-rational irreducible representation $\tilde{V}$ of $\bmG\otimes_F F'$ such that $V\cong\Res_{F'/F} \tilde{V}$ and $\bEnd_{\bmG}(V)\cong \bEnd_{\bmG\otimes_F F'}(\tilde{V})$ as superalgebras over $F'$.
			\item Assume $F'$ to be a separable extension of $F$. Fix an embedding $F'\hookrightarrow F^{\sep}$ over $F$. Choose an irreducible representation $U$ of $\bmG\otimes_F F^{\sep}$ with an injective homomorphism $\tilde{V}\hookrightarrow \Res_{F^{\sep}/F'} U$. Then we have $V=U_F$.
		\end{enumerate}
	\end{prop}
	
	We will see how to compute $\beta^{\super}$ in the next section.
	
	\begin{proof}
		It is evident that $F'$ is a field finite over $F$. Notice that $V$ is naturally equipped with the structure of a vector superspace over $F'$ through the action of $\End_{\bmG}(V)\supset F'$. We denote it by $\tilde{V}$.
		
		We identify the induced coaction $V\to V\otimes_F F[\bmG]$
		with
		$\tilde{V}\to \tilde{V}\otimes_{F'}(F'\otimes_F F[\bmG])$,
		which clearly exhibits the coaction of $F'\otimes_F F[\bmG]$ on $\tilde{V}$. Moreover, it is clear by construction that $V\cong\Res_{F'/F} \tilde{V}$. It is clear that if $\tilde{V}$ is reducible then so is $V$. This shows that $\tilde{V}$ is irreducible.
		
		It is evident that $\bEnd_{\bmG\otimes_F F'}(\tilde{V})\subset \bEnd_{\bmG}(V)$. To see the converse containment, take $\Phi\in \bEnd_{\bmG}(V)$. It follows by definition of $F'$ that $\Phi$ is $F'$-linear. In view of the definition of the structure of the right supercomodule on $\tilde{V}$, $\Phi$ lies in $\bEnd_{\bmG\otimes_F F'}(\tilde{V})$. The equality $\bEnd_{\bmG\otimes_F F'}(\tilde{V})= \bEnd_{\bmG}(V)$ now implies that $\tilde{V}$ is super quasi-rational.
		
		Finally, assume $F'$ to be separable over $F$. Fix $F'\hookrightarrow F^{\sep}$ and $U$ as in (2). Then we obtain an injective map
		$V\cong\Res_{F'/F} \tilde{V}\hookrightarrow \Res_{F'/F} \Res_{F^{\sep}/F'} U=\Res_{F^{\sep}/F} U$.
		This implies $V= U_F$.
	\end{proof}
	
	Let us record how to compute $F'$ canonically:
	
	\begin{cor}\label{cor:compute_F'}
		Assume $F$ to be perfect. Let $V'$ be an irreducible representation of $\bmG\otimes_F \bar{F}$.
		\begin{enumerate}
			\item The even part $F'$ of the center of $\bEnd_{\bmG}(V'_F)$ is isomorphic onto the $\Gamma_{V'}$-fixed point subfield $\bar{F}^{\Gamma_{V'}}$ of $\bar{F}$, where $\Gamma_{V'}=\{\sigma\in\Gamma:~{}^\sigma V'\overset{\boldsymbol{\cdot}}{\cong} V'\}$.
			\item The irreducible representation $V_{F'}$ is super quasi-rational. Moreover, we have $V'_F=\Res_{F'/F} V'_{F'}$ and $\bEnd_{\bmG}(V'_F)=\bEnd_{\bmG\otimes_F F'}(V'_{F'})$.
		\end{enumerate}
	\end{cor}
	
	\begin{proof}
		We prove (1). Let $\GT(V')$ be the image of the map
		\[\Gamma\to\Irr_{\Pi} \bmG\otimes_F\bar{F};~\sigma\mapsto {}^\sigma V'.\]
		For each $W\in \GT(V)$, let $V'_W$ denote the sum of irreducible subrepresentations of $V'_F\otimes_F \bar{F}$ which are isomorphic to $W$ up to parity change. Then we obtain a canonical isomorphism
		$\bEnd_{\bmG}(V'_F)\otimes_F \bar{F}\cong\prod_{W\in\GT(V)} \bEnd_{\bmG\otimes_F\bar{F}}(V'_W)$.
		This restricts to an isomorphism
		\begin{equation}
			F'\otimes_F\bar{F}\cong \prod_{W\in\GT(V)} \bar{F}\label{eq:Galois}
		\end{equation}
		by Lemma \ref{lem:central}.
		
		Unwinding the definitions, we see that the induced canonical Galois action on $\prod_{W\in\GT(V)} \bar{F}$ through the bijection \eqref{eq:Galois} is given by $\sigma\cdot (a_W)\coloneqq (\sigma(a_{{}^{\sigma^{-1}}W}))$. In particular, $F'$ is expressed as its fixed point subalgebra over $F$. We can easily check that the projection into the $V'$-component gives rise to an isomorphism $F'\cong\bar{F}^{\Gamma_{V'}}$ under this identification. In fact, the inverse is given by $a\mapsto (\sigma(a))_{{}^\sigma V'}$.
		
		Part (2) is a restatement of Proposition \ref{prop:nonsuperpure} through the bijection \eqref{eq:classification}.
	\end{proof}

	\begin{ex}\label{ex:nonsuperpure_real}
		Put $F=\bR$. Let $V$ be an irreducible representation of $\bmG\otimes_\bR\bC$ which is not super quasi-rational. Then we have $V_\bR=\Res_{\bC/\bR} V$ and
		\[\bEnd_{\bmG}(V_\bR)=\bEnd_{\bmG\otimes_\bR\bC} (V)\cong\begin{cases}
			\bC&(V\not\cong\Pi V)\\
			\bC\left[\epsilon\right]/(\epsilon^2-1)&(V\cong \Pi V).
		\end{cases}\]
	\end{ex}
	
	\subsection{Rationality}\label{sec:rationality}
	
	In this section, we give rationality results on irreducible representations.
	
	\begin{defn-prop}\label{defprop:absirr}
		Let $\bmG$ be an affine group superscheme over a field $F$ of characteristic not two. Then for an irreducible representation $V$ of $\bmG$, the following conditions are equivalent:
		\begin{enumerate}
			\item[(a)] $V\otimes_F F'$ is irreducible for every finite Galois extension $F'/F$ in $F^{\sep}$.
			\item[(b)] $V\otimes_F F^{\sep}$ is irreducible. 
		\end{enumerate}
		We say that $V$ is pseudo-absolutely irreducible if these equivalent conditions hold.
	\end{defn-prop}
	
	\begin{proof}
		The implication (b) $\Rightarrow$ (a) is clear. Suppose that (a) is satisfied. Lemma \ref{lem:bc_of_Hom} implies an isomorphism
		\[\End_{\bmG\otimes_{F}F^{\sep}} (V\otimes_{F} F^{\sep})
		\cong \varinjlim_{F'/F} \End_{\bmG\otimes_{F}F'} (V\otimes_{F} F'),
		\]
		where $F'$ runs through all finite Galois extension of $F$ in $F^{\sep}$. Schur's lemma then implies that $\End_{\bmG\otimes_{F}F^{\sep}} (V\otimes_{F} F^{\sep})$ is a division algebra over $F^{\sep}$. In view of Lemma \ref{lem:semisimple_descent}, we deduce that $V\otimes_{F} F^{\sep}$ is irreducible as desired.
	\end{proof}

	In \cite{shibata}, Shibata introduced the notion of absolute irreducibility: A representation is called absolutely irreducible if its base change to an arbitrary field extension is irreducible. Let us record the relation with the (pseudo-absolute) irreducibility:
	
	\begin{prop}\label{prop:abs_irr}
		Let $\bmG$ be an affine group superscheme over a field $F$ of characteristic not two. For an irreducible representation $V$ of $\bmG$, the following conditions are equivalent:
		\begin{enumerate}
			\item[(a)] $V$ is irreducible and $\End_F(V)=F\id_V$;
			\item[(b)] $V$ is pseudo-absolutely irreducible and
			\[\End_{\bmG\otimes_FF^{\sep}}(V\otimes_F F^{\sep})=F^{\sep}\id_{V\otimes_F F^{\sep}};\]
			\item[(c)] $V\otimes_F \bar{F}$ is irreducible;
			\item[(d)] $V$ is absolutely irreducible.
		\end{enumerate}
	\end{prop}
	
	\begin{lem}[The Jacobson density theorem]
		Let $A$ be a superalgebra over a field $F$ of characteristic not two, and $V$ be a finite dimensional simple left $A$-supermodule. Set $D=\bEnd_{A}(V)$. Then the action map $A\to \bEnd_{D}(V)$ is surjective.
	\end{lem}
	
	We omit its proof since it is proved in the same way as the non-super setting.
	
	\begin{proof}[Proof of Proposition \ref{prop:abs_irr}]
		The implication (d) $\Rightarrow$ (c) is clear. The implication (c) $\Rightarrow$ (b) follows from Lemma \ref{lem:bc_of_Hom} and Schur's lemma. The implication (b) $\Rightarrow$ (a) follows from Lemma \ref{lem:bc_of_Hom}.
		
		Finally, we prove (a) $\Rightarrow$ (d). Let $F'/F$ be any field extension. We wish to show that $V\otimes_F F'$ is irreducible as a representation of $\bmG\otimes_F F'$.
		Set $D=\bEnd_{\bmG}(V)$. Then $D$ is a central division superalgebra whose even part is $F$. It is easy to show that $V\otimes_F F'$ is a simple left $\bEnd_D(V)\otimes_F F'$-supermodule for the natural action (use the explicit form of $D$ to show that $D\otimes_F F'$ is a central division superalgebra of dimension at most two).
		
		We put the structure of a superalgebra on the superdual $F[\bmG]^\vee$ in a similar way to \cite[Proposition 1.1.1]{MR0252485}. Notice that $V$ is naturally endowed with the structure of a left $F[\bmG]^\vee$-supermodule (\cite[Proposition 2.1.1]{MR0252485}). In view of \cite[Theorem 2.1.3]{MR0252485}, one can apply the Jacobson density theorem to show that $F[\bmG]^\vee \to\bEnd_{D}(V)$ is surjective.
		
		We now deduce that $(F[\bmG]\otimes_F F')^\vee \to \bEnd_D(V)\otimes_F F'$ is surjective. In view of the second paragraph, $V\otimes_F F'$ is a simple left $(F[\bmG]\otimes_F F')^\vee$-supermodule. By virtue of the naturality of the action of $(F[\bmG]\otimes_F F')^\vee$, we deduce that $V\otimes_F F'$ is irreducible. This completes the proof.
	\end{proof}
	
	The next result is used in the next section for the computation of the division superalgebra of endomorphisms.
	
	\begin{prop}\label{prop:rationality}
		Let $V$ be an irreducible representation of $\bmG\otimes_F F^{\sep}$. Then there exist a finite Galois extension $F'/F$ of fields in $F^{\sep}$ and a finite dimensional (irreducible) representation $U$ of $\bmG\otimes_F F'$ such that $V\cong U\otimes_{F'} F^{\sep}$.
	\end{prop}
	
	\begin{proof}
		For each finite Galois extension $F'/F$ in $F^{\sep}$, we apply Theorem \ref{thm:pure} to $\bmG\otimes_F F'$ to get an irreducible representation $V_{F'}$ of $\bmG\otimes_F F'$. Since $V_{F}\otimes_F F^{\sep}$ and $V_{F'}\otimes_{F'} F^{\sep}$ contain $V$ as summands, we deduce
		\[\Hom_{\bmG\otimes_FF'}
		(V_F\otimes_F F',V_{F'})\otimes_{F'}F^{\sep}
		\overset{\eqref{eq:bc_Hom}}{\cong} \Hom_{\bmG\otimes_FF^{\sep}}
		(V_F\otimes_F F^{\sep},V_{F'}\otimes_{F'} F^{\sep})\neq 0.
		\]
		One can therefore find a surjective homomorphism $V_F\otimes_F F'\to V_{F'}$. In particular, we have $\dim_FV_F\geq \dim_{F'} V_{F'}$.
		
		Consider a sequence of finite Galois extensions $\cdots F^{(2)}/F^{(1)}/F^{(0)}=F$ such that $\dim_{F^{(i)}} V_{F^{(i)}}>\dim_{F^{(i+1)}} V_{F^{(i+1)}}$ for $i\geq 0$. Since $\dim_F V_F$ is finite, this sequence stops. In particular, one can find a finite Galois extension $F'/F$ in $F^{\sep}$ such that $V_{F'}$ is pseudo-absolutely irreducible. It follows by Definition-Proposition \ref{defprop:absirr} and the construction of the bijection \eqref{eq:classification} that $U=V_{F'}$ satisfies the condition. In fact, we have $V_{F'}\otimes_{F'} F^{\sep}\cong V$. This completes the proof.
	\end{proof}

	\section{Compute $\beta^{\super}_{V_F}$}\label{sec:BT}
	
	Throughout this section, let $\bmG$ be an affine group superscheme over $F$, and $V$ be super quasi-rational irreducible representation of $\bmG\otimes_FF^{\sep}$. Here we aim to discuss how to compute $\beta^{\super}_{V_F}$ in terms of $V$.
	
	\begin{prop}\label{prop:epsilon}
		We have $\epsilon_{V_F}=-$ if and only if $V$ is quasi-rational and $V\cong \Pi V$.
	\end{prop}
	
	\begin{proof}
		This is an easy consequence of Variant \ref{var:pure} and Propositions \ref{prop:pure}, \ref{prop:even}.
	\end{proof}
	
	\subsection{The second Galois cohomology and the Brauer group}
	
	We plan to give a numerical formula to $D_{V_F}$ in terms of the second Galois cohomology group $H^2(\Gamma,(F^{\sep})^\times)$. For this, we give a quick review on the relation of the Brauer group $\Br(F)$ and $H^2(\Gamma,(F^{\sep})^\times)$, based on \cite[Chapter X, \S 5]{MR554237}.
	
	\begin{cons}
		Let $A$ be a finite dimensional central simple algebra over $F$.
		Then one can find a finite dimensional vector space $V$ over $F$, a finite Galois extension $F'/F$, and an algebra isomorphism $f:\End_{F'}(V\otimes_F F')\cong A\otimes_F F'$ (\cite[Chapter X, \S 5, Proposition 7]{MR554237}). We let $\Gamma_{F'/F}$ denote the Galois group of $F'/F$.
		
		Pick $\sigma\in\Gamma_{F'/F}$. Then $f^{-1}\circ \sigma(f)$ is an algebra automorphism of $\End_{F'}(V\otimes_F F')$, where $\sigma(f)$ is defined by the base change of $f$ by $\sigma$ and the cancellation of $F'$. Therefore one can find a linear automorphism $g_\sigma$ of $V\otimes_F F'$ such that $f^{-1}\circ \sigma(f)$ is the conjugation by $g_\sigma$ (see \cite[the paragraph above Chapter X, \S 5, Proposition 8]{MR554237} if necessary). The matrices $g_\sigma$ form a 1-cocycle with coefficient the projective general linear group of $V\otimes_F F'$.
		
		We send this cocycle to $H^2(\Gamma_{F'/F},(F')^\times)$ by the coboundary operator (\cite[Appendix of Chapter VII]{MR554237}). That is, for $\sigma,\tau\in\Gamma_{F'/F}$, we have
		\[g_\sigma \sigma(g_\tau) g^{-1}_{\sigma\tau}=c^{F'}_A(\sigma,\tau)\id_{V\otimes_F F'}\]
		for some $c_A(\sigma,\tau)\in F'$
		since the left hand side is trivial in the projective general linear group. We now obtain $c^{F'}_A\in H^2(\Gamma_{F'/F},(F')^\times)$.
		
		Finally, we fix an embedding $F'\hookrightarrow F^{\sep}$ over $F$. This gives rise to a homomorphism
		$H^2(\Gamma_{F'/F},(F')^\times)\to H^2(\Gamma,(F^{\sep})^\times)$
		(\cite[Chapter X, \S3]{MR554237}, \cite[Chapter I, Section 2.2, Proposition 8]{MR1867431}).
		We denote the image of $c^{F'}_A$ by $c_A$.
	\end{cons}
	
	Let us record a basic example as a statement:
	
	\begin{prop}\label{prop:cocycle_computation}
		Let $F'/F$ be a finite Galois extension, and $V$ be a nonzero finite dimensional vector space over $F'$. For each $\sigma\in\Gamma_{F'/F}$, 
		suppose that we are given a linear isomorphism $\Phi_\sigma:{}^\sigma V\cong V$. Assume that $\Phi_\sigma\circ {}^\tau \Phi\circ \Phi^{-1}_{\sigma\tau}$ is a scalar map for every pair $(\sigma,\tau)\in \Gamma^2_{F'/F}$. We write $c(\sigma,\tau)\in F'$ for the corresponding scalar.
		
		Then $x\mapsto \Phi_\sigma\circ {}^\sigma x\circ \Phi^{-1}_\sigma$ determines a semilinear action of $\Gamma_{F'/F}$ on $\End_{F'}(V)$. Let $A\subset \End_{F'}(V)$ be the $\Gamma$-invariant part. The Galois descent implies an isomorphism $F'\otimes_F A\cong \End_{F'}(V)$ which is obtained by the scalar extension. We denoote its inverse by $f$.
		
		We fix an $F$-form $V_F$ of $V$ to identify $f$ with
		$\End_{F'}(V_F\otimes_F F')\cong F'\otimes_F A$. Then the corresponding 2-cocycle $c^{F'}_{A}$ is given by
		$c^{F'}_A(\sigma,\tau)=c(\sigma,\tau)$.
	\end{prop}
	
	\begin{proof}
		We define an action of $\Gamma_{F'/F}$ on $\End_{F'}(V)$ (resp.~$F'\otimes_F A$) by the base change by $\sigma\in\Gamma_{F'/F}$ and the cancellation of $F'$ with respect to the $F$-form $V_F$ (resp.~by twisting the first variable). If we write $\Phi^{\can}_\sigma:{}^\sigma V\cong V$ for the cancellation isomorphism with respect to $V_F$, we have
		$\sigma(x)=\Phi^{\can}_\sigma\circ{}^\sigma x\circ \Phi^{\can,-1}_\sigma$
		for $x\in \End_{F'}(V)$ and $\sigma\in\Gamma_{F'/F}$.
		
		It follows by construction of $f^{-1}$ that 
		$f(\Phi_\sigma\circ {}^\sigma x\circ\Phi^{-1}_\sigma)=\sigma(f(x))$
		for $x\in \End_{F'}(V)$ and $\sigma\in \Gamma_{F'/F}$. Using this equality, we get
		\[f^{-1}(\sigma(f)(x))=\Phi_\sigma \circ {}^\sigma(\sigma^{-1}(x))\circ \Phi^{-1}_\sigma=\Phi_\sigma\circ {}^\sigma\Phi^{\can}_{\sigma^{-1}}\circ x\circ {}^\sigma\Phi^{\can,-1}_{\sigma^{-1}}
		\circ \Phi^{-1}_\sigma.\]
		
		The cocycle $c^{F'}_A$ is now computed as
		\begin{flalign*}
			&c^{F'}_A(\sigma,\tau) \id_{\End_{F'}(V)}\\
			&=\Phi_\sigma\circ {}^\sigma\Phi^{\can}_{\sigma^{-1}}
			\circ \sigma(\Phi_\tau\circ {}^\tau\Phi^{\can}_{\tau^{-1}})
			\circ {}^{\sigma\tau}\Phi^{\can,-1}_{(\sigma\tau)^{1}}\circ \Phi^{-1}_{\sigma\tau}\\
			&=\Phi_\sigma\circ {}^\sigma\Phi^{\can}_{\sigma^{-1}}
			\circ\Phi^{\can}_\sigma \circ {}^\sigma(\Phi_\tau\circ {}^\tau\Phi^{\can}_{\tau^{-1}})\circ \Phi^{\can,-1}_\sigma
			\circ {}^{\sigma\tau}\Phi^{\can,-1}_{(\sigma\tau)^{1}}\circ \Phi^{-1}_{\sigma\tau}\\
			&=\Phi_\sigma \circ {}^\sigma\Phi_\tau\circ {}^{\sigma\tau}\Phi^{\can}_{\tau^{-1}}\circ \Phi^{\can,-1}_\sigma
			\circ {}^{\sigma\tau}\Phi^{\can,-1}_{(\sigma\tau)^{1}}\circ \Phi^{-1}_{\sigma\tau}\\
			&=\Phi_\sigma \circ {}^\sigma\Phi_\tau\circ \Phi^{-1}_{\sigma\tau}\\
			&=c(\sigma,\tau)\id_{\End_{F'}(V)}
		\end{flalign*}
		($\sigma,\tau\in\Gamma_{F'/F}$). In fact, we recall that $(\Phi^{\can}_\sigma)_{\sigma\in\Gamma_{F'/F}}$ satisfies the cocycle condition.
	\end{proof}

	\begin{thm}[{\cite[Chapter X, $\S5$]{MR554237}}]
		The assignment $A\mapsto c_A$ gives a well-defined group isomorphism 
		\begin{equation}
			\Br(F)\cong H^2(\Gamma,(F^{\sep})^\times).\label{eq:Serre}
		\end{equation}
	\end{thm}
	
	\begin{rem}\label{rem:continuity}
		Let $\PGL(V\otimes_F F')$ and $\PGL(V\otimes_F F^{\sep})$ denote the projective general linear groups of $V\otimes_F F'$ and $V\otimes_F F^{\sep}$ respectively. Then for a fixed embedding $F'\hookrightarrow F^{\sep}$, the diagram 
		\[\begin{tikzcd}
			H^1(\Gamma_{F'/F},\PGL(V\otimes_F F'))\ar[r]\ar[d]
			&H^2(\Gamma_{F'/F}, (F')^\times)\ar[d]\\
			H^1(\Gamma,\PGL(V\otimes_F F^{\sep}))\ar[r]
			&H^2(\Gamma,(F^{\sep})^\times)
		\end{tikzcd}\]
		of the canonical maps commutes (see \cite[Chapter I, Section 2.2, Proposition 8 and the paragraph above Chapter I, Section 5.7, Proposition 43]{MR1867431}). That is, take the pullback of $(g_\sigma)$ to obtain a 1-cocycle of $\Gamma$ with coefficient $\PGL(V\otimes_F F^{\sep})$. Then the coboundary operator sends it to a continuous cocycle, and in fact it coincides with $c_A$.
	\end{rem}

	\subsection{Quasi-rational case}
	
	We assume $V$ to be quasi-rational in this small section. We compute $\beta^{\super}_{V_F}$, based on ideas of \cite[Proof of 12.6. Proposition]{MR207712}. We choose a finite Galois extension $F'/F$ such that $V$ admits an $F'$-form which we denote by $V_{F'}$ (Proposition \ref{prop:rationality}).
	
	For each element $\bar{\sigma}$ of the Galois group $\Gamma_{F'/F}$ of $F'/F$, we fix an isomorphism ${}^{\bar{\sigma}} V_{F'}\cong V_{F'}$ (recall Lemma \ref{lem:bc_of_Hom}). We choose an embedding $F'\hookrightarrow F^{\sep}$ over $F$ and lift them to isomorphisms
	$\Phi_\sigma:{}^\sigma V\cong V$ for $\sigma\in\Gamma$ by the projection $\Gamma\to \Gamma_{F'/F}$ and an isomorphism $V\cong V_{F'}\otimes_{F'} F^{\sep}$.
	
	\begin{cons}\label{cons:BT}
		\begin{enumerate}
			\item For each pair $\sigma,\tau\in\Gamma$, observe that we have an equality
			$\Phi_\sigma \circ {}^\sigma\Phi_{\tau}\circ \Phi_{\sigma\tau}^{-1}=c_{\sigma,\tau}\id_{V}$
			for some element $c_{\sigma,\tau}\in (F^{\sep})^\times$ since $V$ is (super) quasi-rational. This determines a continuous 2-cocycle, which we denote by $\beta^{\BT}_{V}$, and call it the Borel--Tits cocycle. We will regard $\beta^{\BT}_{V}$ as an element of $\Br(F)$ for the isomorphism \eqref{eq:Serre}.
			\item Thanks to the fact that $\Phi_\sigma \circ {}^\sigma\Phi_{\tau}\circ \Phi_{\sigma\tau}^{-1}$ are scalar maps, the conjugation of $\Phi_\sigma$ determines a continuous Galois action on $\bEnd_{F^{\sep}}(V)$. Define a central simple superalgebra $A$ over $F$ as the corresponding $F$-form of $\bEnd_{F^{\sep}}(V)$ to this Galois action (notice that the even part of the center satisfies the base change property). Explicitly, $A$ is described as
			\begin{flalign*}
				&A_{\bar{0}}=\{f\in\End_{F^{\sep}}(V):~\forall \sigma\in\Gamma,~
				f\circ\Phi_\sigma=\Phi_{\sigma} \circ {}^\sigma f\},\\
				&A_{\bar{1}}=\{f\in\Hom_{F^{\sep}}(V,\Pi V):~\forall \sigma\in\Gamma,~
				\Pi f\circ\Phi_\sigma=\Pi\Phi_{\sigma} \circ {}^\sigma f\}.
			\end{flalign*}
		\end{enumerate}	
	\end{cons}
	
	\begin{cons}\label{cons:BT_representation}
		Let $W$ be a left $\bEnd_{F^{\sep}}(V)$-supermodule. Then $W$ is naturally endowed with the structure of a representation of $\bmG\otimes_F F^{\sep}$ in the following way: Let $R$ be a commutative superalgebra over $F^{\sep}$, and $g\in \bmG(R)$. Thanks to the property that $V$ is of finite dimension, we then obtain
		$\rho(g)\in (R\otimes_{F^{\sep}}\bEnd_{F^{\sep}}(V))_{\bar{0}}$ after a canonical identification
		$\bEnd_R(R\otimes_{F^{\sep}} V)\cong R\otimes_{F^{\sep}}\bEnd_{F^{\sep}}(V)$
		(recall the definition of a representation). Then $\rho(g)$ acts on $R\otimes_{F^{\sep}} W$ through the base change $R\otimes_{F^{\sep}}-$ of the structure of the left $\bEnd_{F^{\sep}}(V)$-module on $W$. Run through all $R$ and $g$ to get the structure of a representation of $\bmG\otimes_F F^{\sep}$ on $W$.
	\end{cons}

	\begin{rem}\label{rem:comodule}
		One can define it in a comodule theoretic way: Identify the superalgebra $\bEnd_{F^{\sep}}(V)$ with its superdual to regard $\bEnd_{F^{\sep}}(V)$ as a supercoalgebra. Then $W$ is a right supercomodule over $\bEnd_{F^{\sep}}(V)$. The structure of a representation of $\bmG\otimes_F F^{\sep}$ on $V$ gives rise to a homomorphism of supercoalgebras from $\bEnd_{F^{\sep}}(V)$ to $F[\bmG]\otimes_F F^{\sep}$. Combine them to obtain the right supercomodule corresponding to the representation constructed above.
	\end{rem}

	Since $A$ is central simple, $A$ admits a unique finite dimensional simple left supermodule $U$ up to isomorphism and parity change (see \cite[Section 3.4]{MR1701598}). We fix $U$. A similar argument to Lemma \ref{lem:base_change} shows that $U\otimes_FF^{\sep}$ is decomposed into a direct sum of a simple left supermodule $W$ over $A\otimes_FF^{\sep}\cong \bEnd_{F^{\sep}}(V)$ and its Galois twists with respect to the Galois action in Construction \ref{cons:BT} (2), which agrees with that as a representation of $\bmG\otimes_FF^{\sep}$. Since the simple finite dimensional left $\bEnd_{F^{\sep}}(V)$-modules are only $V$ and $\Pi V$ up to isomorphism, we replace $U$ with $\Pi U$ if necessary to assume $W=V$. Since $V$ is quasi-rational, we obtain $U\otimes_F F^{\sep}\cong V^m$ for some positive integer $m$.
	
	\begin{lem}\label{lem:division_algebra_description}
		Let $V$ be a quasi-rational irreducible representation of $\bmG\otimes_F F^{\sep}$. Define $A$ and $U$ be as above.
		\begin{enumerate}
			\item We have $\End_A(U)\cong\End_{\bmG}(V_F)$ and $\Hom_A(U,\Pi U)=0$.
			\item There is an isomorphism $U\cong V_F$ of representations of $\bmG$.
		\end{enumerate}	
	\end{lem}
	
	\begin{proof}
		Choose a projection $U\otimes_FF^{\sep}\to V$ as a left $\bEnd_{F^{\sep}}(V)$-supermodule to get an injective homomorphism $U\to V$ of left $A$-supermodules.
		
		Regard $U\otimes_F F^{\sep}$ as a representation of $\bmG\otimes_F F^{\sep}$ by Construction \ref{cons:BT_representation}. In virtue of the Galois descent, it gives rise to the structure of a representation of $\bmG$ on $U$. It is clear that $\End_A(U)\subset\End_{\bmG}(U)$. To see the converse containment, consider an element
		$\Phi\in\End_{\bmG}(U)$.
		To prove that $\Phi$ is $A$-linear, we may work over $F^{\sep}$. Consider the sequence
		$V^m\cong U\otimes_F F^{\sep}\overset{\Phi\otimes_F F^{\sep}}{\cong}
		U\otimes_F F^{\sep} \cong V^m$.
		Recall that the left and right bijections are $A\otimes_F F^{\sep}$-linear. The total map is $G\otimes_FF^{\sep}$-equivariant. Since
		$\End_{G\otimes_F F^{\sep}}(V)=\End_{A\otimes_F F^{\sep}}(V)
		=\End_{\bEnd_{F^{\sep}}(V)}(V)=F^{\sep}\id_V$,
		the total map is $A\otimes_F F^{\sep}$-linear. We thus obtain
		\begin{equation}
			\End_A(U)=\End_{\bmG}(U).\label{eq:End}
		\end{equation}
		Thanks to $U\otimes_F F^{\sep}\cong V^m$ and $\Hom_{\End_{F^{\sep}}(V)}(V,\Pi V)=0$, a similar base change argument shows
		$\Hom_A(U,\Pi U)=0$.
		
		Notice that $U$ is semisimple by Lemma \ref{lem:semisimple_descent}. Since $\End_{\bmG}(U)$ is a division algebra by \eqref{eq:End}, $U$ is irreducible. Since $U$ is a subrepresentation of $\Res_{F^{\sep}/F} V$ by the first paragraph, we get $U\cong V_F$ as stated in (2) (recall the construction of the bijection \eqref{eq:classification}). We obtain (1) by application of this isomorphism to \eqref{eq:End}. This completes the proof.
	\end{proof}

	\begin{thm}\label{thm:BT}
		Let $V$ be a quasi-rational irreducible representation of $\bmG\otimes_F F^{\sep}$.
		\begin{enumerate}
			\item We have $\beta_{V_F}=(\beta^{\BT}_V)^{-1}$.
			\item We have $\beta^{\super}_{V_F}=(+,1,(\beta^{\BT}_V)^{-1})$ if $V\not\cong\Pi V$.
			\item Assume $V\cong \Pi V$. 
			\begin{enumerate}
				\item[(i)] There exists a unique isomorphism $S:V\cong \Pi V$ of representations of $\bmG\otimes_FF^{\sep}$ up to $F^\times$ such that $S\circ \Phi_{\sigma}=\Pi\Phi_\sigma\circ {}^\sigma S$ for every $\sigma\in\Gamma$.
				\item[(ii)] Choose any $S$ in (i). Then $\Pi S\circ S=a\id_{V}$ for some $a\in F^\times$. Moreover, we have $a_{V_F}=a$ for this $a$.
			\end{enumerate}
		\end{enumerate}
		
	\end{thm}
	
	\begin{proof}
		To distinguish the similarity classes of $A$ in $\BW(F)$ and $\Br(F)$, we temporally refer to them as $[A]^{\super}$ and $[A]$ respectively. Lemma \ref{lem:division_algebra_description} (1) implies
		\[[A]^{\super}=[\bEnd_A(U)]^{-1}=[\End_A(U)]^{-1}=(+,1,[\End_A(U)]^{-1})\]
		in $\BW(F)$ (see \cite[Section 3.4]{MR1701598} for the first equality). It then follows by construction of the bijection \eqref{eq:bw} that $[A]=[\End_A(U)]^{-1}$ (see the $\Br(F)$-component of $[A]^{\super}$). On the other hand, it follows from Proposition \ref{prop:cocycle_computation} and Remark \ref{rem:continuity} that $\beta^{\BT}_{V}=[A]$. We have shown
		$\End_A(U)\cong\End_{\bmG}(V_F)$ (Lemma \ref{lem:division_algebra_description} (1)).
		Recall that we defined $\beta_{V_F}=[\End_{\bmG}(V_F)]\in\Br(F)$. Combine these to deduce (1).
		
		Part (2) is clear by construction of the bijection \eqref{eq:bw}.
		
		To prove (3), assume $V\cong \Pi V$. We use \cite[(3.3.3)]{MR1701598} to obtain
		\begin{equation}
			A\cong\bEnd_{\bEnd_A(U)}(U)
			=\bEnd_{\End_A(U)}(U)=\bEnd_{\End_{\bmG}(V_F)}(U)\cong\bEnd_{\End_{\bmG}(V_F)}(V_F).
			\label{eq:technical_identification}
		\end{equation}
		Observe that $\bmG\otimes_FF^{\sep}$ acts on the superalgebra $\bEnd_{F^{\sep}}(V)$ by the conjugation. It descends to an action of $\bmG$ on $A$. We notice that an element of $A\subset\bEnd_{F^{\sep}}(V)$ is $\bmG\otimes_FF^{\sep}$-equivariant if and only if it is $\bmG$-invariant with respect to this action. Similarly, the other terms of \eqref{eq:technical_identification} are equipped with actions of $\bmG$ for the conjugation. Moreover, the isomorphisms of \eqref{eq:technical_identification} are $\bmG$-equivariant. Take the odd part of the $\bmG$-invariant parts to obtain an isomorphism of the vector spaces of
		\begin{itemize}
			\item homomorphisms $V\to \Pi V$ of representations of $\bmG\otimes_FF^{\sep}$ satisfying the equality
			$S\circ \Phi_{\sigma}=\Pi\Phi_\sigma\circ {}^\sigma S$
			for every $\sigma\in\Gamma$, and of
			\item homomorphisms $V_F\to \Pi V_F$ of representations of $\bmG$ commuting with endomorphisms of $V_F$ as a representation. 
		\end{itemize}
		They are of one dimension by \cite[Lemma 5]{MR0167498}. This proves (i) of (3). Part (ii) of (3) follows by construction of the bijection \eqref{eq:bw} since \eqref{eq:technical_identification} respects the structures of superalgebras.
	\end{proof}
	
	\begin{cor}[{\cite[12.8. Proposition]{MR207712}}]\label{cor:abs_irr}
		For a quasi-rational irreducible representation $V$ of $\bmG\otimes_F F^{\sep}$, the following conditions are equivalent:
		\begin{enumerate}
			\renewcommand{\labelenumi}{(\alph{enumi})}
			\item $V$ admits an $F$-form,
			\item $V_F$ is absolutely irreducible, and
			\item $\beta_{V_F}$ is trivial.
		\end{enumerate}
		Moreover, if these equivalent conditions hold, we have
		\[\beta^{\super}_{V_F}\in\{(+,1,[F])\}\cup\{(-,a,[F])\in\BW(F):~a\in F^\times/(F^\times)^2\}.\]
	\end{cor}
	
	\begin{ex}
		Put $F=\bR$. Then $\beta^{\BT}_{V_F}$ agrees with the index. Namely, choose $\Phi:\bar{V}\cong V$. Then $\Phi\circ \bar{\Phi}=a\id_{V_F}$ for some $a\in \bR^\times$. The sign of $a$ is independent of the choice of $\Phi$. Moreover, the following conditions are equivalent:
		\begin{enumerate}
			\item[(a)] $\beta^{\BT}_{V}$ is trivial;
			\item[(b)] $a>0$;
			\item[(c)] the index of $V$ is one;
			\item[(d)] $V$ admits a real form;
			\item[(e)] $V_\bR$ is absolutely irreducible.
		\end{enumerate}
	\end{ex}
	
	\subsection{Non quasi-rational case}\label{sec:nonqrat}
	
	Let $V$ be an irreducible representation of the affine group superscheme $\bmG\otimes_{F} F^{\sep}$ which is not quasi-rational but super quasi-rational. We note that $V\not\cong\Pi V$ by Proposition \ref{prop:pure}. We already know $\epsilon_{V_F}=+$ from Proposition \ref{prop:epsilon}. 
	
	To compute $D_{V}=[\bEnd_{\bmG}(V_F)]\in\Br(F)$, choose a finite Galois extension $F'/F$ such that $V$ admits an $F'$-form which we denote by $V_{F'}$ (Proposition \ref{prop:rationality}). For each element $\bar{\sigma}$ of the Galois group $\Gamma_{F'/F}$ of $F'/F$, we fix an isomorphism ${}^{\bar{\sigma}} V_{F'}\overset{\boldsymbol{\cdot}}{\cong} V_{F'}$. We choose an embedding $F'\hookrightarrow F^{\sep}$ over $F$ and lift them to isomorphisms
	$\Phi_\sigma:{}^\sigma V\overset{\boldsymbol{\cdot}}{\cong} V$ for $\sigma\in\Gamma$ by the projection $\Gamma\to \Gamma_{F'/F}$ and an isomorphism
	$V\cong V_{F'}\otimes_{F'} F^{\sep}$.
	
	In this small section, we think that the compositions are defined in the enriched sense \eqref{eq:enriched_composition} or forget the parity in order to omit $\Pi$. For each pair $\sigma,\tau\in\Gamma$, one can write
	$\Phi_\sigma \circ {}^\sigma\Phi_{\tau}\circ \Phi_{\sigma\tau}^{-1}=c(\sigma,\tau)\id_{V}$
	for some $c(\sigma,\tau)\in (F^{\sep})^\times$ since it belongs to $\bEnd_{\bmG\otimes_F F^{\sep}}(V)= F^{\sep}\id_{V}$. Then we obtain a continuous 2-cocycle $\beta^{\BT}_V$. We again call it the Borel--Tits cocycle.
	
	\begin{cons}\label{cons:a}
		Let $V$ be an irreducible representation of $\bmG\otimes_F F^{\sep}$ which is not quasi-rational but super quasi-rational. For each $\sigma\in\Gamma$, ${}^\sigma V$ is isomorphic to exactly one of $V$ and $\Pi V$ by $V\not\cong\Pi V$. This gives rise to a map $\Gamma\to \{V,\Pi V\}\cong\bZ/2\bZ$. By virtue of Hilbert's Theorem 90, one can find $a\in F^\times$ such that the map agrees with the Galois action on $\{\pm\sqrt{a}\}$.
	\end{cons}
	
	\begin{thm}\label{thm:BT_nonqrat}
		Let $V$ be an irreducible representation of $\bmG\otimes_F F^{\sep}$ which is not quasi-rational but super quasi-rational.
		\begin{enumerate}
			\item We have $D_{V_F}=(\beta^{\BT}_V)^{-1}(-1,a_{V_F})$.
			\item We have $a_{V_F}=a$ for $a$ in Construction \ref{cons:a}.
		\end{enumerate}
	\end{thm}
	
	We omit its proof because this is verified in a similar way to the quasi-rational case. The main differences are $U\otimes_FF^{\sep}\cong V^m\oplus\Pi V^n$ with $m,n\geq 1$ and the vanishing
	$\Hom_{\bmG\otimes_F F^{\sep}}(V,\Pi V)=0$.
	In particular, the latter vanishing implies
	\[\bEnd_{\bmG\otimes_F F^{\sep}}(V)=\bEnd_{\bEnd_{F^{\sep}}(V)}(V)= F^{\sep}\id_V,\]
	which leads us to an isomorphism $\bEnd_A(U)\cong\bEnd_{\bmG}(V_F)$ of superalgebras this time. In particular, we obtain
	\begin{equation}
		[\bEnd_{\bmG}(V_F)]=[\bEnd_A(U)]=[A]^{-1}.\label{eq:End=Ainverse}
	\end{equation}
	For computation of $a_{V_F}$, we use \eqref{eq:End=Ainverse}, Example \ref{ex:BW_inverse},
	and Deligne's description in \cite[Proof of Lemma 3.7]{MR1701598}.
	
	\begin{ex}
		Put $F=\bR$. One sees from Example \ref{ex:realbrauer-wall} that central division superalgebras over $\bR$ with non-central even part have the value $-1\in\{\pm 1\}\cong\bR^\times/(\bR^\times)^2$. This implies $a_{V_\bR}=-1$.
	\end{ex}
	
	\begin{rem}\label{rem:uniform_expression}
		The descriptions of $D_{V_F}$ in Theorems \ref{thm:BT} and \ref{thm:BT_nonqrat} are compatible. In fact, we identify $\pm$ with $\pm 1$ to consider $(-\epsilon,a_{V_F})$. In the setting of Theorem \ref{thm:BT} (2), we have $a_{V_F}=1$ and thus $(-\epsilon,a_{V_F})=[F]$; In the setting of Theorem \ref{thm:BT} (3), we have $\epsilon_{V_F}=+$ from Proposition \ref{prop:epsilon} and thus $(-\epsilon_{V_F},a_{V_F})=[F]$; In the setting of Theorem \ref{thm:BT_nonqrat}, recall that we have $\epsilon_{V_F}=+$ and thus $(-\epsilon_{V_F},a_{V_F})=(-1,a_{V_F})$. Therefore in all cases, we have $D_{V_F}=(\beta^{\BT}_V)^{-1}(-\epsilon_{V_F},a_{V_F})$.
	\end{rem}
	
	\section{Quasi-reductive case I}\label{sec:qred}
	
	We assume $\bmG$ to be quasi-reductive in the sense of Definition \ref{defn:qred} from now on. In this section, we aim to classify irreducible representations of $\bmG$ and to determine the division superalgebras of their endomorphisms, based on results of \cite{MR4039427} and the preceding arguments. Let $H$ be a maximal torus of $G$. Following our notations, write $\fg$ for the Lie superalgebra of $\bmG$. Let $\fh$ be the $H$-invariant part of $\fg$. Then $\fh$ and $H$ give rise to a subgroup $\bmH\subset \bmG$ by \cite[Proposition 2.6]{MR4039427} (cf.~\cite[Lemma 3.5]{MR4039427}). Since every torus over a field is split over its finite separable extension, $H\otimes_F F^{\sep}$ and therefore $\bmG\otimes_F F^{\sep}$ are split. Let $\Delta$ (resp.~$\Delta_{\bar{0}}$, $\Delta_{\bar{1}}$) be the set of roots of $\fg\otimes_F F^{\sep}$ (resp.~$\fg_{\bar{0}}\otimes_F F^{\sep}$, $\fg_{\bar{1}}\otimes_F F^{\sep}$). We refer to elements of $\Delta_{\bar{0}}$ (resp.~$\Delta_{\bar{1}}$) as even (resp.~odd) roots. Let $Q$ be the abelian subgroup of $X^\ast(H\otimes_F F^{\sep})$ generated by $\Delta$. Choose an abelian group homomorphism $\gamma:Q\to\bR$ such that $\gamma(\alpha)\neq 0$ for every root. In practice, we will frequently construct $\bR\otimes_\bZ X^\ast(H\otimes_F F^{\sep})\to \bR$ and restrict it to $Q$. Following \cite[Sections 3.2, 3.3]{MR4039427}, we set
	\begin{align*}
		\Delta^+&=\{\alpha\in \Delta:~\gamma(\alpha)>0\},
		&\Delta^-&=\{\alpha\in \Delta:~\gamma(\alpha)<0\},\\
		\Delta^+_{\bar{0}}&=\{\alpha\in \Delta_{\bar{0}}:~\gamma(\alpha)>0\},
		&\Delta^-_{\bar{0}}&=\{\alpha\in \Delta_{\bar{0}}:~\gamma(\alpha)<0\}
		=-\Delta^+_{\bar{0}}.
	\end{align*}
	Sets of roots of the form $\Delta^+$ are called positive systems. An indecomposable element of $\Delta^+$ is called a simple root. Let $(\bmU')^+,\bmU'\subset\bmG\otimes_F F^{\sep}$ be the unipotent subgroups corresponding to $\Delta^+,\Delta^-$ respectively as in \cite[Section 3.2]{MR4039427}. Define the subgroup $\bmB'\subset \bmG\otimes_F F^{\sep}$ corresponding to $\Delta^-\cup\{0\}$ in a similar way as in \cite[Section 3.3]{MR4039427}. The subgroups obtained in this way are called BPS-subgroups in this paper\footnote{BPS stands for Borel--Penkov--Serganova.}\footnote{For generalities, we should start with a maximal torus over $F^{\sep}$. We took an $F$-form for technical reasons towards the descent problem.}. We remark that $B'$ is exhibits a Borel subgroup of $G\otimes_F F^{\sep}$ containing $H\otimes_F F^{\sep}$. 
	
	\begin{rem}
		There always exists $\gamma$ by the standard argument, i.e., we see that such $\gamma$ arises from a complementary dense open subset to finitely many hyperplanes. One can find an extension of any Borel subgroup of $G\otimes_F F^{\sep}$ containing $H\otimes_F F^{\sep}$ to that of $\bmG\otimes_F F^{\sep}$ by translating $\gamma$ by the Weyl group of $(G\otimes_F F^{\sep},H\otimes_F F^{\sep})$.
	\end{rem}
	
	In \cite{MR4039427}, BPS-subgroups are called Borel super-subgroups. This could be up to convention. We followed \cite{MR2906817} in order to avoid conflict with \cite{MR3012224}. Indeed, we will use some results in \cite{MR3012224}. On the other hand, these definitions agree in many cases:
	
	\begin{defn}\label{defn:basic}
		We say $\bmG$ is basic of the main type if the following conditions are satisfied:
		\begin{enumerate}
			\renewcommand{\labelenumi}{(\roman{enumi})}
			\item the characteristic of $F$ is zero, and
			\item $\fg\otimes_F \bar{F}$ is a product of basic simple Lie superalgebras over $\bar{F}$ of the main type in the sense of \cite[Section 2.4]{MR2336486} and $\fgl_{m|n}$ for some $m,n\geq 0$.
		\end{enumerate}
	\end{defn}
	
	In fact, the notions of our positive systems (cf.~\cite[(3.2.3), Lemma 3.2.1]{MR2906817}) and those in \cite[Section 1.3.1]{MR3012224} agree for $\bmG$ basic of the main type since the simple roots are linearly independent for any positive systems under both definitions in this case (see \cite[Proposition 1.28]{MR3012224} and \cite[Sections 2.3, 2.4]{MR2336486}). Let us remark that if $\bmG$ is basic of the main type, then $\fg\otimes_F \bar{F}$ is a product of basic Lie superalgebras in the sense of \cite[Definition 1.14]{MR3012224} up to isomorphism (unless $\fg_{\bar{1}}=0$). To compare the lists, see \cite[Remarks 1.6 and 1.9, Sections 1.1.5 and 1.1.7]{MR3012224} if necessary.
	
	\subsection{Split case}\label{sec:split}
	
	Let us recall Shibata's Borel--Weil theory for split quasi-reductive algebraic supergroups. For this, assume $H$ to be split in this section. Through the canonical bijection $X^\ast(H)\cong X^\ast(H\otimes_F F^{\sep})$, we obtain $F$-forms $\bmB,\bmU^+\subset \bmG$ of $\bmB',(\bmU')^+$ respectively. We call $\bmB$ a BPS-subgroup.
	
	According to \cite[Proposition 4.4]{MR4039427}, there is a bijection
	\begin{equation}
		\Irr_{\Pi} \bmH\cong X^\ast(H).\label{eq:IrrH}
	\end{equation}
	For each irreducible representation of $\bmH$, the corresponding character is given by its unique weight (\cite[Remark 4.5]{MR4039427}). The converse direction will be reviewed later when $F$ is separably closed. The projection $\bmB\to \bmH$ induces a $\Pi$-preserving bijection $\Irr \bmH\cong \Irr \bmB$ by \cite[Proof of Proposition 4.9]{MR4039427}. Henceforth we identify irreducible representations of $\bmH$ and $\bmB$ for the projection. For an irreducible representation $\fu$ of $\bmB$, let $V_{F,\bmB}(\fu)$ be the socle of $\Ind^{\bmG}_{\bmB}\fu$ as a representation of $\bmG$ (cf.~\cite[Proposition 4.17]{MR4039427}). We put the subscript $F$ in order to avoid the conflict of notations later; we will mainly concern the case $F=F^{\sep}$ after the base change. We will remove this subscript in the next section. We frequently omit the index $\bmB$ to write $V_{F}(\fu)$ when the BPS-subgroup $\bmB$ is clear from the context. Observe that $V_F(\Pi\fu)\cong \Pi V_F(\fu)$ by
	\begin{equation}
		\Ind^{\bmG}_{\bmB}\Pi\fu
		\cong\Pi \Ind^{\bmG}_{\bmB}\fu.\label{eq:parity_change}
	\end{equation}
	The socle $V_F(\fu)$ is a unique irreducible subrepresentation of $\Ind^{\bmG}_{\bmB} \fu$ unless
	$\Ind^{\bmG}_{\bmB}\fu$ vanishes
	(\cite[Propositions 4.4 and 4.11]{MR4039427}). For later applications, let us note that $\fu$ is recovered from $V(\fu)$ unless $V(\fu)=0$:
	
	\begin{prop}[{\cite[Lemma 4.10, Proof of Proposition 4.11, Proposition 4.15]{MR4039427}}]\label{prop:Vknowsu}
		For an irreducible representation $\fu$ of $\bmB$, the counit $\Ind^{\bmG}_{\bmB}\fu\to \fu$ restricts to an isomorphism
		\begin{equation}
			H^0(\bmU^+,V_F(\fu))\cong\fu\label{eq:canonical_splitting}
		\end{equation}
		of representations of $\bmH$ unless $V_F(\fu)= 0$. In particular, $\fu$ is a direct summand of $V_F(\fu)$ as an $\bmH$-module.
	\end{prop}

	\begin{cor}\label{cor:uknowsv}
		Let $\fu,\fv$ be irreducible representations of $\bmB$ with $V_F(\fu)\neq 0$. Then we have $\fu\cong\fv$ if and only if $V_F(\fu)\cong V_F(\fv)$.
	\end{cor}
	
	\begin{proof}
		The ``only if'' direction follows by definition of $V_F(\fu),V_F(\fv)$. Conversely, suppose $V_F(\fu)\cong V_F(\fv)$. Then $V_F(\fv)$ is nonzero since so is $V_F(\fu)$. Proposition \ref{prop:Vknowsu} now implies an isomorphism
		$\fu\cong H^0(\bmU^+,V_F(\fu))\cong H^0(\bmU^+,V_F(\fv))\cong\fv$
		as $\bmH$-modules. Since irreducible representations of $\bmB$ are obtained by the projection $\bmB\to\bmH$, this is an isomorphism of $\bmB$-modules. This completes the proof.
	\end{proof}

	We now choose a corresponding irreducible representation $\fu_F(\lambda)\in \Irr_{\Pi} \bmB$ to each $\lambda\in X^\ast(H)$. Following \cite[Section 4.4]{MR4039427}, we set
	\[X^\flat(H)=\{\lambda\in X^\ast(H):~\Ind^{\bmG}_{\bmB}\fu_F(\lambda)\neq 0\}.\]
	The condition $\Ind^{\bmG}_{\bmB}\fu_F(\lambda)\neq 0$ is independent of choice of $\fu_F(\lambda)$ by \eqref{eq:parity_change} since the whole choices of $\fu_F(\lambda)$ are $\fu_F(\lambda)$ and $\Pi \fu_F(\lambda)$. We note that $X^\flat(H)$ is independent of separable extension of base fields in the following sense:
	
	\begin{prop}
		We identify $X^\ast(H\otimes_F F^{\sep})$ with $X^\ast(H)$ through the canonical bijection. Take $\lambda\in X^\ast(H)$. Choose $\fu_{F^{\sep}}(\lambda)$. Then we have $\Ind^{\bmG}_{\bmB}\fu_F(\lambda)\neq 0$ if and only if $\Ind^{\bmG\otimes_F F^{\sep}}_{\bmB\otimes_F F^{\sep}}\fu_{F^{\sep}}(\lambda)\neq 0$.
	\end{prop}
	
	\begin{proof}
		One can prove in a similar way to \cite[Proposition I.4.13]{MR2015057} that
		\begin{equation}
			(\Ind^{\bmG}_{\bmB}\fu_F(\lambda))\otimes_F F^{\sep}
			\cong \Ind^{\bmG\otimes_F F^{\sep}}_{\bmB\otimes_F F^{\sep}}(\fu_F(\lambda)\otimes_F F^{\sep}).\label{eq:bc_Ind}
		\end{equation}
		On the other hand, we see the weight to get
		\[\fu_F(\lambda)\otimes_F F^{\sep}\cong
		\fu_{F^{\sep}}(\lambda)^{m_0}\oplus \Pi \fu_{F^{\sep}}(\lambda)^{m_1}
		\]
		for some $m_0,m_1\geq 0$ with $m_0+m_1\geq 1$ (Lemma \ref{lem:base_change} (1)). We thus obtain
		\[(\Ind^{\bmG}_{\bmB}\fu_F(\lambda))\otimes_F F^{\sep}
		\cong
		\Ind^{\bmG\otimes_F F^{\sep}}_{\bmB\otimes_F F^{\sep}}\fu_{F^{\sep}}(\lambda)^{m_0}
		\oplus \Pi \Ind^{\bmG\otimes_F F^{\sep}}_{\bmB\otimes_F F^{\sep}}\fu_{F^{\sep}}(\lambda)^{m_1}.\]
		The assertion now follows by the faithfully flat descent. This completes the proof.
	\end{proof}
	
	\begin{rem}\label{rem:one-dim}
		If $\fh_{\bar{1}}=0$, $\fu_F(\lambda)$ is one dimensional by the proof of \cite[Proposition 4.2]{MR4039427}, and therefore $\fu_F(\lambda)\otimes_F F^{\sep}$ is irreducible; Otherwise, $\fu_F(\lambda)\otimes_F F^{\sep}$ is not irreducible in general. Its obstruction comes from the rationality problem of irreducible representations of $\bmH\otimes_F F^{\sep}$ (see Proposition \ref{prop:split} and Example \ref{ex:Clifford}).
	\end{rem}

	According to \cite[Theorem 4.12]{MR4039427}, we have a bijection
	
	\begin{equation}
		X^\flat(H)\cong\Irr_{\Pi}\bmG ;~
		\lambda\mapsto V_F(\fu_F(\lambda)),\label{eq:superhwtheory}
	\end{equation} 
	which is canonical in the sense that this map is independent of the choice of $\fu_F(\lambda)$ for each $\lambda$. In view of \cite[Lemma 4.10 and Proof of Proposition 4.11]{MR4039427}, the converse direction is constructed as follows: For each irreducible representation $V$ of $\bmG$, $H^0(\bmU^+,V)$ is an irreducible representation of $\bmH$ (Proposition \ref{prop:Vknowsu}). The element of $X^\flat(H)$ corresponding to $V$ is given by the unique $H$-weight of the $H^0(\bmU^+,V)$ (recall the construction of the bijection \eqref{eq:IrrH}).
	
	\begin{rem}
		For explicit description of $X^\flat(H)$ in special cases, see \cite[Proposition 2.2]{zbMATH03603460}, \cite[Theorem 2]{MR0831047}, \cite[Theorem 4.5]{zbMATH01985814}, \cite[Theorem 6.11]{MR1973576}, \cite[Lemma 2.3]{zbMATH05117269}, \cite[Theorem 5.3]{MR2392322}, \cite[Theorems 2.1, 3.5, 4.8]{MR3936085}, and \cite[Section 5]{MR4039427}.
	\end{rem}
	
	The bijection \eqref{eq:superhwtheory} gives a classification of irreducible representations of $\bmG$ up to parity change. For the complete classification, we are still interested in whether $V_F(\fu(\lambda))\cong \Pi V_F(\fu(\lambda))$, or equivalently, $\fu_F(\lambda)\cong \Pi \fu_F(\lambda)$. Shibata judged it by using the module theory of Clifford superalgebras (\cite[Theorem 4.12]{MR4039427}). Here we reprove it as a consequence of determination of the central division superalgebras of irreducible representations of $\bmG$. For $\lambda\in X^\ast(H)$, define a quadratic form $q^\lambda_F$ on $\fh_{\bar{1}}$ by
	\begin{equation}
		q^\lambda_F(x,y)=\frac{1}{2}\lambda(\left[x,x\right]).\label{eq:q^lambda_F}
	\end{equation}
	Choose a nondegenerate orthogonal subspace $(\fh_{\bar{1}})_s$ to the radical of $q^\lambda_F$ in $\fh_{\bar{1}}$. Let $d_{\lambda,F}$ denote the dimension of $(\fh_{\bar{1}})_s$. We also set $\delta_{\lambda,F}\in F^\times/(F^\times)^2$ as the signed discriminant of $q^\lambda_F|_{(\fh_{\bar{1}})_s}$ if $(\fh_{\bar{1}})_s\neq 0$; We set $\delta_{\lambda,F}=0$ if $(\fh_{\bar{1}})_s=0$ by convention. These are independent of the choice of $(\fh_{\bar{1}})_s$.
	
	\begin{prop}[{\cite[Theorem 4.12]{MR4039427}}]\label{prop:split}
		Let $\bmG$ be a split quasi-reductive algebraic supergroup over $F$, and $H$ be a split maximal torus of the even part $G$ of $\bmG$. We fix a BPS-subgroup to define $X^\flat(H)$. 
		\begin{enumerate}
			\item Every irreducible representation of $\bmG$ is super quasi-rational.
			\item We have
			$\beta^{\super}_{V_F(\fu_F(\lambda))}=[\bar{C}(\fh_{\bar{1}},q^{-\lambda}_F)]$
			for every $\lambda\in X^\flat(H)$.
			\item Let $\lambda\in X^\flat(H)$. Then $V_F(\fu_F(\lambda))\not\cong \Pi V_F(\fu_F(\lambda))$ if and only if $d_{\lambda,F}$ is even and $\delta_{\lambda,F}\in (F^\times)^2\cup\{0\}$.
		\end{enumerate}
	\end{prop}
	
	\begin{proof}
		We first prove an isomorphism
		$\bEnd_{\bmG}(V_F(\fu_F(\lambda)))\cong \bEnd_{\bmH}(\fu_F(\lambda))$. For this, recall that $V_F(\fu_F(\lambda))$ is realized as the unique irreducible subrepresentation of $\Ind^{\bmG}_{\bmB} \fu_F(\lambda)$. Hence any endomorphism of $\fu_F(\lambda)$ extends to that of $\Ind^{\bmG}_{\bmB} \fu_F(\lambda)$ and then retricts to that of $V_F(\fu_F(\lambda))$. Conversely, $\fu_F(\lambda)$ is recovered from $V_F(\fu_F(\lambda))$ as the $\mathbbmss{U}^+$-invariant part (Proposition \ref{prop:Vknowsu}). Hence any endomorphism of $V_F(\fu_F(\lambda))$ restricts to that of $\fu_F(\lambda)$. We have similar correspondences for the odd parts. These give the isomorphism.
		
		We next prove (1) and (2).
		Observe that $\fu_F(\lambda)$ can be regarded as an irreducible representation of $C(\fh_{\bar{1}},q^\lambda_F)$ and thus of $\bar{C}(\fh_{\bar{1}},q^\lambda_F)$ by \cite[Lemma 4.1]{MR4039427}. This implies
		an isomorphism $\bEnd_{\bmH}(\fu_F(\lambda))
		\cong \bEnd_{\bar{C}(\fh_{\bar{1}},q^\lambda_F)}(\fu_F(\lambda))$ (recall that $C(\fh_{\bar{1}},q^\lambda_F)$ is a subquotient of the dual superalgebra to $F[\bmH]$).
		Since $\bar{C}(\fh_{\bar{1}},q^\lambda_F)$ is central simple, $\bEnd_{\bmH}(\fu_F(\lambda))$ and $\bEnd_{\bar{C}(\fh_{\bar{1}},q^\lambda_F)}(\fu_F(\lambda))$ are central division superalgebras. In particular (1) follows. Moreover, we obtain the equalities
		\[\left[\bEnd_{\bmH}(\fu_F(\lambda))\right]
		=\left[\bEnd_{\bar{C}(\fh_{\bar{1}},q^\lambda_F)}(\fu_F(\lambda))\right]
		=[\bar{C}(\fh_{\bar{1}},q^\lambda_F)]^{-1}
		=[\bar{C}(\fh_{\bar{1}},q^{-\lambda}_F)]\]
		in $\BW(F)$ (recall Example \ref{ex:BW_inverse} for the last equality). This shows (2).
		
		Finally, (3) follows from Remark \ref{rem:D_1=0}, Example \ref{ex:Clifford}, and \cite[Theorem 3]{MR0167498}.
	\end{proof}
	
	\begin{cor}\label{cor:schur's_lemma}
		Suppose that $F$ is separably closed. Then we have an equality
		$\End_{\bmG}(V)=F\id_V$
		for any irreducible representation $V$ of a quasi-reductive algebraic supergroup $\bmG$ over $F$. In particular, the notions of weak (super) quasi-rationality and (super) quasi-rationality in Definition \ref{defn:pure} (2) are equivalent for quasi-reductive algebraic supergroups over a field of characteristic not two.
	\end{cor}
	
	\begin{cor}\label{cor:absolute_irreducibility}
		Consider the setting of Proposition \ref{prop:split}. Let $\lambda\in X^\flat(H)$. Assume that $\fu_F(\lambda)\otimes_F F^{\sep}$ is irreducible. Then $V_F(\fu_F(\lambda))$ is absolutely irreducible.
	\end{cor}
	
	\begin{proof}
		We use the isomorphism \eqref{eq:bc_Ind} to obtain an isomorphism
		\[(\Ind^{\bmG}_{\bmB}\fu_F(\lambda))\otimes_F F^{\sep} \cong 
		\Ind^{\bmG\otimes_F F^{\sep}}_{\bmB\otimes_F F^{\sep}}(\fu_{F^{\sep}}(\lambda)).\]
		We obtain an isomorphism $V_F(\fu_F(\lambda))\otimes_F F^{\sep}\cong V_{F^{\sep}}(\fu_{F^{\sep}}(\lambda))$ by restriction. In fact, $V_F(\fu_F(\lambda))\otimes_F F^{\sep}$ is semisimple (Lemma \ref{lem:semisimple_descent} (2)). Since $V_{F^{\sep}}(\fu_{F^{\sep}}(\lambda))$ is the socle of $\Ind^{\bmG\otimes_F F^{\sep}}_{\bmB\otimes_F F^{\sep}}(\fu_{F^{\sep}}(\lambda))$, the restriction to $V_F(\fu_F(\lambda))\otimes_F F^{\sep}$ is onto $V_{F^{\sep}}(\fu_{F^{\sep}}(\lambda))$ as desired.
	\end{proof}
	
	\begin{cor}[{\cite[Proposition 5.6]{shibata}}]\label{cor:split_h1=0}
		Let $\lambda\in X^\flat(H)$ with $\delta_\lambda=0$. Then $V_F(\fu_F(\lambda))$ is absolutely irreducible. In particular, every irreducible representation of $\bmG$ is absolutely irreducible if $\fh_{\bar{1}}=0$.
	\end{cor}

	\begin{ex}[{\cite[Theorem 1.18]{MR3012224}}]
		We have $\fh_{\bar{1}}=0$ if the characteristic of $F$ is zero and $\fg\otimes_F \bar{F}$ is basic in the sense of \cite[Definition 1.14]{MR3012224}.
	\end{ex}
	
	\begin{ex}\label{ex:type_sep_clo_case}
		Assume $F$ to be separably closed. Then we observe
		$(F^\times)^2=F^\times$
		since the characteristic of $F$ is not two. Hence for $\lambda\in X^\flat(H)$,
		\[V_F(\fu_F(\lambda))\cong \Pi V_F(\fu_F(\lambda))\]
		if and only if $d_{\lambda,F}$ is odd. In particular, these equivalent conditions fail if $\fh_{\bar{1}}=0$. One also has a simple description of $\delta_{\lambda,F}$:
		\[\delta_{\lambda,F}=\begin{cases}
			1&(q^\lambda_F\neq 0)\\
			0&(q^\lambda_F=0).
		\end{cases}\]
	\end{ex}
	
	\begin{ex}\label{ex:Q_n}
		Let $n\geq 1$. Put $\bmG=\Qq_n$ (see Section \ref{sec:notation}). Let $H\subset \Qq_n$ be the subgroup of diagonal matrices.
		Identify $X^\ast(H)\cong \bZ^n$ in the standard way. Take the standard positive system as in \cite[Section 1]{MR0831047}. If we write $p$ for the characteristic of $F$, we have
		\begin{flalign*}
			&X^\flat(H)\\
			&=\{\lambda=(\lambda_i)\in\bZ^n:~\lambda_1\geq \lambda_2\geq\cdots\geq \lambda_n,~\mathrm{if}~i\in\{1,2,\ldots,n-1\}~
			\mathrm{satisfies}~\lambda_i=\lambda_{i+1},~p|\lambda_i\}
		\end{flalign*}
		after \cite[Theorem 6.11]{MR1973576} (and \cite[Section 6]{MR0831047} for $p=0$). Here if $p=0$, the condition $p|\lambda_i$ means $\lambda_i=0$.
		
		Identify the Lie superalgebra $\fgl_{n|n}$ of $\GL_{n|n}$ with that of square matrices of size $2n$ of two diagonal blocks of square matrices of size $n$ for the even part as usual (see \cite[Section 1.1.2]{MR3012224} for example). Then we have
		\[\fh_{\bar{1}}=\left\{\left(\begin{array}{cc}
			0 & y \\
			y & 0
		\end{array}\right)\in\fgl_{n|n}(\bR):~y~\mathrm{is~diagonal}\right\}.\]
		For $1\leq i\leq n$, write $E_{ii}$ for the square matrix unit of size $n$ with $1$ at the $(i,i)$-entry, and
		\[\eta_i=\left(\begin{array}{cc}
			0 & E_{ii} \\
			E_{ii} & 0
		\end{array}\right).\]
		Then for $\lambda=(\lambda_i)\in\bZ^n$, the quadratic form $q^\lambda_F$ is given by
		$q^\lambda_F(\sum_{i=1}^na_i\eta_i)=\sum_{i=1}^n \lambda_i a^2_i$.
		See \cite[Lemma B.1]{MR4039427}, Example \ref{ex:Clifford}, and \cite[Theorem 3]{MR0167498} for computation of the similarity class
		$[\bar{C}(\fh_{\bar{1}},q^\lambda_F)]$.
	\end{ex}
	
	\begin{ex}\label{ex:Q_n_real}
		Consider the setting of Example \ref{ex:Q_n} with $F=\bR$. Take $\lambda\in X^\flat(H)$. Write
		\[\begin{array}{cc}
			n_+\coloneqq |\{i\in\{1,2,\ldots,n\}:~\lambda_i>0\}|,
			&n_-\coloneqq |\{i\in\{1,2,\ldots,n\}:~\lambda_i<0\}|.
		\end{array}\]
		Then we have
		\begin{align*}
			d_\lambda&=n_++n_-,
			&\epsilon_{V_\bR(\fu_\bR(\lambda))}&=-^{d_\lambda},\\
			a_{V_\bR(\fu_\bR(\lambda))}&=(-1)^{\binom{d_\lambda}{2}+n_+},
			&D_{V_\bR(\fu_\bR(\lambda))}&=(-1)^{\binom{n_+}{2}+\binom{d_\lambda-1}{2}n_++\binom{d_\lambda+1}{4}}.
		\end{align*}
	\end{ex}
	
	\begin{ex}
		Consider the setting of Example \ref{ex:Q_n}. Assume $F$ to be a finite field of characteristic $p>0$. Take $\lambda\in X^\flat(H)$. Write
		\[\begin{array}{cc}
			S\coloneqq \{i\in\{1,2,\ldots,n\}:~\lambda_i\not\equiv0\pmod{p}\},
			&m=|S|.
		\end{array}\]
		Then we have
		\[\begin{array}{cccc}
			d_\lambda=m,&\epsilon_{V_F(\fu_F(\lambda))}=-^{d_\lambda},
			&a_{V_F(\fu_F(\lambda))}=(-1)^{\binom{d_\lambda+1}{2}} \prod_{i\in S} \lambda_i,
			&D_{V_F(\fu_F(\lambda))_F}=[F]
		\end{array}\]
		(see \cite[Chapter X, \S 7]{MR554237} for the last equality).
	\end{ex}

	Finally, we discuss how the bijection \eqref{eq:superhwtheory} changes by replacing $\bmB$ with another BPS-subgroup. The easy part is the translation by $\bmG(F)=G(F)$:
	
	\begin{cons}\label{cons:pullback}
		Take $w\in G(F)$. Then for a representation $\fu$ of $w\bmB w^{-1}$, let ${}^w \fu$ denote the representation of $\bmB$ defined by the pullback through the isomorphism
		$\bmB\cong w \bmB w^{-1};~b\mapsto w bw^{-1}$.
		We will apply similar notations to homomorphisms.
	\end{cons}

	\begin{prop}\label{prop:G-conjugate}
		Let $w\in G(F)$. Then for an irreducible representation $\fu$ of $w\bmB w^{-1}$, there are isomorphisms
		\begin{equation}
			\Ind^{\bmG}_{w\bmB w^{-1}} \fu\cong
			\Ind^{\bmG}_{\bmB} {}^{w}\fu,
			\label{eq:w-twist_Ind}
		\end{equation}
		\begin{equation}
			V_{F,w\bmB w^{-1}}(\fu)\cong V_{F,\bmB}({}^w\fu)
			\label{eq:w-twist_V(u)}.
		\end{equation}
	\end{prop}
	
	\begin{proof}
		The translation by $w$ gives rise to the first isomorphism. Take the socles to obtain \eqref{eq:w-twist_V(u)}.	
	\end{proof}
	
	For further analysis, suppose that $\bmG$ is basic of the main type in order to make use of the odd reflections. In fact, all the positive systems are transferred to each other by the Weyl group of $(G,H)$ and odd reflections (\cite[Proposition 1.32, Section 1.4.3]{MR3012224}). There an odd root is called isotropic if $2\alpha\not\in\Delta$ (\cite[(1.18)]{MR3012224}). For this reason, we may restrict ourselves now to odd reflections.
	
	Consider an isotropic simple odd root $\alpha$. Write $s_{-\alpha}$ for the odd reflection with respect to $-\alpha$ (\cite[Lemma 1.30]{MR3012224}). Let $\bmB_\alpha$ denote the BPS-subgroup attached to $s_{-\alpha}\Delta^-=(\Delta^-\setminus\{-\alpha\})\cup\{\alpha\}$. Define $\bmU^+_\alpha$ in a similar way.
	
	Let $\lambda\in X^\ast(H)$. According to \cite[Theorem 1.18]{MR3012224}, we have $\bH=H$ since $\fh_{\bar{1}}=0$. Regard the field $F$ as a vector superspace with trivial odd part. We define 
	the structure of a $\bmB$-module on $F$ by $\lambda$. We denote it by $F_\lambda$ and put $\fu_F(\lambda)=F_\lambda$.
	Let $\pi_{\lambda,\bmB}:\Ind^{\bmG}_{\bmB} F_\lambda\to F_\lambda$ denote the counit.
	
	Recall that if $\lambda\in X^\flat(H)$, $\pi_{\lambda,\bmB}$ restricts to an isomorphism
	\[H^0(\bmU^+,V_{F,\bmB}(F_\lambda))\overset{\eqref{eq:canonical_splitting}}{\cong} F_\lambda.\]
	We denote the preimage of $1\in F$ by $v_{\lambda,F}$.
	
	Write $(-,-)$ for the product of the bilinear forms in \cite[Theorem 1.18 (5)]{MR3012224}. 
	
	\begin{prop}\label{prop:hw_odd_reflection}
		Let $\bmG$ be a split basic quasi-reductive algebraic supergroup of the main type over a field $F$ of characteristic zero, $H$ be a split maximal torus of the even part $G$, and $\bmB$ be a BPS-subgroup containing $H$. Let $\lambda\in X^\flat(H)$, and $\alpha$ be an isotropic simple odd root. Define $\pi_{\lambda,\bmB}$ as above.
		
		\begin{enumerate}
			\item If $(\lambda,\alpha)=0$ then the composite map
			$V_{F,\bmB}(F_\lambda)\hookrightarrow \Ind^{\bmG}_{\bmB} F_\lambda
			\overset{\pi_{\lambda,\bmB}}{\to} F_\lambda$
			is $\bmB_\alpha$-equivariant. Moreover, it induces an isomorphism
			$V_{F,\bmB}(F_\lambda)\cong V_{F,\bmB_\alpha}(F_\lambda)$.
			\item Suppose $(\lambda,\alpha)\neq 0$. Choose a root vector $E_{\alpha}$ of root $\alpha$.
			\begin{enumerate}
				\item[(i)] The composite linear map
				\[V_{F,\bmB}(F_\lambda)
				\overset{E_\alpha}{\to} \Pi V_{F,\bmB}(F_\lambda)
				\xrightarrow{\Pi \pi_{\lambda,\bmB}|_{V_{F,\bmB}(F_\lambda)}}
				\Pi F\]
				of vector superspaces is $\bmB_\alpha$-equivariant for $\lambda-\alpha$ as the structure of a $\bmB_\alpha$-module on the target. Moreover, it induces an isomorphism
				\begin{equation}
					V_{F,\bmB}(F_\lambda)\cong\Pi
					V_{F,\bmB_\alpha}(F_{\lambda-\alpha}).
					\label{eq:odd_reflection}
				\end{equation}
				\item[(ii)] Choose a root vector $E_{-\alpha}$ of root $-\alpha$. Then apply (i) to
				\[\begin{array}{cccc}
					\bmB_\alpha,&\lambda-\alpha,&-\alpha,&E_{-\alpha}
				\end{array}\]
				to obtain an isomorphism 
				\begin{equation}
					V_{F,\bmB_\alpha}(F_{\lambda-\alpha})\cong \Pi V_{F,\bmB}(F_\lambda).
					\label{eq:odd_reflection_2}
				\end{equation}
				Its parity change is inverse to \eqref{eq:odd_reflection} if we normalize $E_{-\alpha}$ so that we have
				$\lambda([E_\alpha,E_{-\alpha}])=1$.
			\end{enumerate}
		\end{enumerate}
	\end{prop}
	
	\begin{proof}
		For (1), assume $(\lambda,\alpha)=0$. Put the canonical structure of a representation of $\bmG$ on the superdual $V_{F,\bmB}(F_\lambda)^\vee\coloneqq \bHom_F(V_{F,\bmB}(F_\lambda),F)$. Then one can regard $\pi_{\lambda,\bmB}|_{V_{F,\bmB}(F_\lambda)}$ as an element of
		$\Hom_{\bmB}(F_{-\lambda},V_{F,\bmB}(F_\lambda)^\vee)$. Combine \cite[Lemma 1.40]{MR3012224} with the interpretation into the Harish-Chandra modules in \cite[Proposition 5.4]{MR3605977} to obtain
		\[\begin{split}
			\pi_{\lambda,\bmB}|_{V_{F,\bmB}(F_\lambda)}
			&\in \Hom_{\bmB}(F_{-\lambda},V_{F,\bmB}(F_\lambda)^\vee)\\
			&=\Hom_{\bmB_\alpha}(F_{-\lambda},V_{F,\bmB}(F_\lambda)^\vee)\\
			&\cong \Hom_{\bmB_\alpha}(V_{F,\bmB}(F_\lambda),F_\lambda)
		\end{split}.\]
		Follow the last canonical bijection to deduce (1).
		
		A similar argument shows (i) of (2). Unwinding the definitions, we see that the map \eqref{eq:odd_reflection} sends $E_{-\alpha}v_{\lambda}$ to $\lambda([E_\alpha,E_{-\alpha}])v_{\lambda-\alpha}\neq 0$ by
		\[E_\alpha E_{-\alpha}v_{\lambda}=
		[E_\alpha,E_{-\alpha}] v_\lambda=\lambda([E_\alpha,E_{-\alpha}])v_\lambda\]
		(recall $E_\alpha v_\lambda=0$).
		Since \eqref{eq:odd_reflection} is $\fg$-equivariant, it sends
		$v_{\lambda}=\frac{1}{\lambda([E_\alpha,E_{-\alpha}])}E_\alpha E_{-\alpha}v_\lambda$
		to $E_\alpha v_{\lambda-\alpha}$.
		Similarly, the isomorphism \eqref{eq:odd_reflection_2} sends $v_{\lambda-\alpha}$ to $E_{-\alpha} v_{\lambda}$. Therefore the composite map
		\[V_{F,\bmB}(F_\lambda)\overset{\eqref{eq:odd_reflection}}{\cong}\Pi
		V_{F,\bmB_\alpha}(F_{\lambda-\alpha})
		\overset{\Pi \eqref{eq:odd_reflection_2}}{\cong} V_{F,\bmB}(F_\lambda)
		\]
		sends $v_\lambda$ to
		$E_\alpha E_{-\alpha} v_\lambda=\lambda([E_\alpha,E_{-\alpha}])v_{\lambda}$.
		Part (ii) is now clear.
	\end{proof}

	\subsection{General case I}\label{sec:general}
	
	In this section, we study irreducible representations of $\bmG$ without the split hypothesis. We wish to do this by combination of the bijection \eqref{eq:superhwtheory} for $F=F^{\sep}$ and the Galois descent. In fact, though we no longer have the Borel--Weil construction over $F$, we are still able to apply Shibata's theory in the previous section to $\bmG\otimes_F F^{\sep}$. That is, for each irreducible representation $\fu$ of $\bmH\otimes_F F^{\sep}$, we can define $V(\fu)=V_{\bmB'}(\fu)\coloneqq V_{F^{\sep},{\bmB'}}(\fu)$. We have a bijection
	\begin{equation}
		\Irr_{\Pi} \bmH\otimes_F F^{\sep}\cong X^\ast(H\otimes_F F^{\sep})\label{eq:IrrH'}
	\end{equation}
	and a subset $X^\flat(H\otimes_F F^{\sep})\subset X^\ast(H\otimes_F F^{\sep})$. Then we obtain a bijection
	\begin{equation}
		X^\flat(H\otimes_F F^{\sep})\cong\Irr_{\Pi} \bmG\otimes_F F^{\sep}.\label{eq:superhwtheory'}
	\end{equation}
	
	We would like to analyze how \eqref{eq:superhwtheory'} and $V(\fu)$ change by the Galois twists. We choose a finite Galois extension $F'/F$ in $F^{\sep}$ such that $H\otimes_F F'$ is split. Correspondingly, we get an $F'$-form of $B'$ which we denote by $B'_{F'}$. For each element $\sigma$ of the Galois group of $F'/F$, choose $w_{\bar{\sigma}}\in G(F')$ such that
	\[\begin{array}{cc}
		w_{\bar{\sigma}} (H\otimes_F F')w^{-1}_{\bar{\sigma}}=H\otimes_F F',
		&{}^{\bar{\sigma}} B'_{F'}=w_{\bar{\sigma}} B'_{F'}w^{-1}_{\bar{\sigma}}.
	\end{array}\]
	We lift them to define $w_\sigma\in G(F^{\sep})$ for each element $\sigma\in\Gamma$.
	It gives rise to Tits' $\ast$-action $\sigma\ast\lambda\coloneqq w_\sigma^{-1}{}^\sigma\lambda$ on $X^\ast(H\otimes_F F^{\sep})$ (see \cite[Section 2.3]{MR0224710}). We remark that this action depends on the choice of the positive system of roots of
	\[(G\otimes_F F^{\sep},H\otimes_F F^{\sep}).\]
	
	\begin{notedefn}\label{notedefn}
		When $F=\bR$ and $\sigma$ is the nontrivial element of $\Gamma$, we denote $w_{\sigma}=w$. We call $\sigma\ast(-)$ for this $\sigma$ the $\ast$-involution.
	\end{notedefn}
	
	\begin{ass}\label{ass}
		We have ${}^\sigma \bmB'=w_\sigma \bmB'w^{-1}_\sigma$ for every $\sigma\in\Gamma$.
	\end{ass}
	
	It is clear that Assumption \ref{ass} is equivalent to $\sigma\ast\Delta^-=\Delta^-$ for every $\sigma\in\Gamma$. This holds, for example, if the $\ast$-action is trivial. We remark that the triviality of the $\ast$-action is independent of choice of $H$ and $\bmB'$. To check Assumption \ref{ass}, let us record that this is independent of the choice of $H$ and $B'$ in the following sense:
	
	\begin{prop}
		Let $\tilde{H}$ be a maximal torus of $G$, and $\tilde{\bmB}'$ be a BPS-subgroup containing $\tilde{H}\otimes_F F^{\sep}$. Assume that there exists $g\in G(F^{\sep})$ such that
		\[\begin{array}{cc}
			g (H\otimes_F F^{\sep})g^{-1}=\tilde{H}\otimes_F F^{\sep},
			& g \bmB' g^{-1}=\tilde{\bmB}'.
		\end{array}\]
		Then $(H,\bmB')$ satisfies Assumption \ref{ass} if and only if so does $(\tilde{H},\tilde{\bmB}')$.
	\end{prop}
	
	\begin{proof}
		It suffices to prove the ``only if'' direction. If we are given $w_\sigma\in G(F^{\sep})$ normalizing $H\otimes_F F^{\sep}$ and satisfying
		${}^\sigma \bmB'=w_\sigma \bmB' w^{-1}_\sigma$,
		set $\tilde{w}_\sigma=\sigma(g) w_\sigma g^{-1}$. Then we have
		\[\begin{split}
			\tilde{w}_\sigma(\tilde{H}\otimes_F F^{\sep})\tilde{w}^{-1}_\sigma
			&=\sigma(g) w_\sigma g^{-1}(\tilde{H}\otimes_F F^{\sep}) gw^{-1}_\sigma \sigma(g)^{-1}\\
			&=\sigma(g) w_\sigma (H\otimes_F F^{\sep})w^{-1}_\sigma \sigma(g)^{-1}\\
			&=\sigma(g) (H\otimes_F F^{\sep}) \sigma(g)^{-1}\\
			&={}^\sigma (g (H\otimes_F F^{\sep}) g^{-1})\\
			&={}^\sigma (\tilde{H}\otimes_F F^{\sep})\\
			&\cong\tilde{H}\otimes_F F^{\sep},
		\end{split}\]
		\[\begin{split}
			{}^\sigma \tilde{\bmB}'&=\sigma(g) {}^{\sigma}\bmB'\sigma(g)^{-1}\\
			&=\sigma(g) w_\sigma \bmB'w^{-1}_\sigma \sigma(g)^{-1}\\
			&=\sigma(g)w_\sigma g^{-1} \tilde{\bmB}' gw^{-1}_\sigma \sigma(g)^{-1}\\
			&=\tilde{w}_\sigma \tilde{\bmB}'\tilde{w}^{-1}_\sigma.
		\end{split}\]
	\end{proof}
	
	\begin{rem}
		Even if $(H,\bmB')$ satisfies Assumption \ref{ass}, it is not necessarily true that another pair $(H,\bmB'')$ of the maximal torus $H$ and a BPS-subgroup satisfies Assumption \ref{ass}. For instances, see Examples \ref{ex:u(p,q|r,s)} and \ref{ex:spo(m|2p+1,2q+1)}.
	\end{rem}
	
	Hence we may fix $H$ and $B'$. For the BPS-subgroups, we may only think of those whose even part is $B'$. If $\fg\otimes_F F^{\sep}\cong\fgl,\fspo$, such BPS-subgroups are parameterized by the so-called $\epsilon\delta$-sequences (\cite[Section 1.3]{MR3012224}). Therefore it is enough to find an $\epsilon\delta$-sequence fixed by the induced $\Gamma$-action which we will call the Galois action (or the complex conjugate action, Galois involution if $F=\bR$). Indeed, since $\epsilon\delta$-sequences parameterize the conjugacy classes of positive systems by the Weyl group of the even part (\cite[Proposition 1.27]{MR3012224}), this action is only determined by the usual $\Gamma$-action on $X^\ast(H\otimes_F F^{\sep})$. The action of the Weyl group action does not affect the sequences. However, we still have to take $w_\sigma$ into account when we want to find the positive system in the conjugacy class corresponding to $B'$. Let us also note that thinking of the $\ast$-action is sometimes helpful for computation of the induced action for this reason.
	
	\begin{ex}\label{ex:res}
		Let $\tilde{F}/F$ be a finite separable extension of fields of characteristic not two, and $\tilde{\bmG}$ be a split quasi-reductive algebraic supergroup over $\tilde{F}$. Take a split maximal torus $\tilde{H}\subset \tilde{G}$. Take any BPS-subgroup $\tilde{\bmB}\subset \tilde{\bmG}$ containing $\tilde{H}$. Then $\Res_{\tilde{F}/F}\tilde{\bmG}$ is a quasi-reductive algebraic supergroup over $F$. One can see that $(\Res_{\tilde{F}/F} \tilde{\bmB})\otimes_F F^{\sep}$ is a BPS-subgroup. This clearly satisfies Assumption \ref{ass} (put the unit for $w_{\bar{\sigma}}$).
	\end{ex}
	
	\begin{ex}[Indefinite unitary supergroup]\label{ex:u(p,q|r,s)}
		Put $F=\bR$. Let $p,q,r,s$ be nonnegative integers. Set
		$n=p+q+r+s$ and $I_{p,q|r,s}=\diag(I_p,-I_q,I_r,-I_s)$.
		Define a linear algebraic supergroup $\Uu(p,q|r,s)$ by
		\[\Uu(p,q|r,s;A)=\{g\in \GL_{p+q|r+s}(A\otimes_\bR\bC):~(\delta g)^\ast I_{p,q|r,s}g=I_{p,q|r,s}\}.\]
		This is a real form of $\GL_{p+q|r+s}$ for the restriction of the counit of the adjunction $(-\otimes_\bR\bC,\Res_{\bC/\bR})$, i.e.,
		\begin{equation}
			\Uu(p,q|r,s)\otimes_\bR\bC\cong\GL_{p+q|r+s};~g=(g_{ij}+g'_{ij}\otimes\sqrt{-1})\mapsto (g_{ij}+g'_{ij}\sqrt{-1}),\label{eq:u(p,q|r,s)}
		\end{equation}
		where $g_{ij},g'_{ij}$ are the real and imaginary parts in the $(i,j)$-entry. In particular, it is a quasi-reductive algebraic supergroup over $\bR$. The even part is given by the product	
		$\Uu(p,q)\times\Uu(r,s)$
		of indefinite unitary groups. See \cite[Section 3.4]{MR2069561} for the linear algebraic background on the appearance of $\delta$. Take the subgroup of diagonal matrices for $H$. Identify $H\otimes_\bR\bC$ with the split torus $H^{\spl}\subset\GL_{p+q|r+s}$ of diagonal matrices for \eqref{eq:u(p,q|r,s)}. This gives rise to an isomorphism
		$X^\ast(H\otimes_\bR\bC)\cong X^\ast(H^{\spl})$.
		Write $\{e_i\}_{1\leq i\leq n}$ for the standard basis of $X^\ast(H^{\spl})$. Then we have
		\[\Delta=\{\pm(e_i-e_j)\in X^\ast(H^{\spl}):~1\leq i<j\leq n\}\]
		\[\Delta_{\bar{0}}=\{\pm(e_i-e_j)\in X^\ast(H^{\spl}):~1\leq i<j\leq p+q,
		~p+q+1\leq i<j\leq n\}.\]
		Following \cite{MR3012224}, consider the standard positive system
		\[\Delta^+_{\bar{0}}=\{e_i-e_j\in X^\ast(H^{\spl}):~1\leq i<j\leq p+q,
		~p+q+1\leq i<j\leq n\}\]
		of the even part. Since the complex conjugation acts on $X^\ast(H\otimes_\bR\bC)$ as $-1$, we have $\bar{\Delta}^+=-\Delta^+$. For the same reason, $w$ is the longest element at the level of the Weyl group $W$ of $\Delta_{\bar{0}}$. Under the identification of $W$ with $\fS_{p+q}\times\fS_{r+s}$, we have
		\[w=\left(\left(\begin{array}{cccc}
			1&2&\cdots&p+q\\
			p+q&p+q-1&\cdots&1
		\end{array}\right),\left(\begin{array}{cccc}
			p+q+1&p+q+2&\cdots&n\\
			n&n-1&\cdots&p+q+1
		\end{array}\right)\right).\]
		It is straightforward to see that the complex conjugate action on the set of $\epsilon\delta$-sequences is given by reversing the sequences, for instance,
		$\epsilon\delta\delta\epsilon\epsilon\mapsto \epsilon\epsilon\delta\delta\epsilon$.
		Therefore $\Uu(p,q|r,s)$ satisfies Assumption \ref{ass} for a certain positive system if and only if $(p+q)(r+s)$ is even. 
	\end{ex}
	
	\begin{ex}[Unitary periplectic supergroup]\label{ex:p(n)}
		Put $F=\bR$. For $n\geq 0$, define a linear algebraic supergroup $\Pp(n)$ by
		$\Pp(n;A)=\{g\in\GL_{n|n}(A\otimes_\bR\bC):~\Pi g=(g^\ast)^{-1}\}$.
		This is a real form of $\GL_{n|n}$. We identify $\Res_{\bC/\bR}\GL_n\cong G$ for $g\mapsto \diag(g,(g^\ast)^{-1})$.
		Let $H^{\spl}$ be the base change to $\bC$ of the subgroup of diagonal matrices in $\GL_n$ over $\bR$. We chose this real structure for the computation right below on the $\ast$-involution. We set $H=\Res_{\bC/\bR} H^{\spl}$. Identify $H\otimes_\bR\bC$ with the subgroup of diagonal matrices of $\GL_{n|n}$ to write $X^\ast(H\otimes_\bR\bC)\cong\bZ^{2n}$ in the standard way. We regard
		$\bZ^{2n}=(\bZ^n)^2$
		for separating tuples into the first and last $n$ entries. In virtue of a minor modification of \cite[Proof of Proposition 3.2.2]{MR4627704}, the induced complex conjugate action is given by
		$(\lambda,\lambda')\mapsto (-\lambda',-\lambda)$.
		We set $\gamma:\bR^{2n}\cong \bR\otimes_\bZ X^\ast(H\otimes_\bR\bC)\to\bR$ as
		\[\gamma(i)=\begin{cases}
			i&(1\leq i\leq n),\\
			3n-i+1&(n+1\leq i\leq 2n).
		\end{cases}\]
		Then the set of positive roots is given by
		\[\Delta^+=\{e_i-e_j\in\bZ^{2n}:~1\leq i<j\leq 2n,~i\leq n\}\cup
		\{e_j-e_i\in\bZ^{2n}:~n+1\leq i<j\leq 2n\},
		\]
		where $\{e_i\}$ is the standard basis of $\bZ^{2n}$.
		If we write $B$ for the subgroups of upper triangular matrices in $\GL_n$ over $\bC$, we have
		$B'=(\Res_{\bC/\bR} B)\otimes_\bR\bC$.
		Therefore $w$ can be the unit, and the $\ast$-involution is equal to the complex conjugation.
		One can easily check that $\Delta^+$ is closed under the formation of the complex conjugation. Hence $\Pp(n)$ satisfies Assumption \ref{ass} for the present $H$ and $\gamma$.
	\end{ex}

	\begin{ex}\label{ex:0q(n)}
		Put $F=\bR$. For $n\geq 0$, define a linear algebraic supergroup ${}^0\Qq(n)$ by
		${}^0\Qq(n;A)=\{g\in\GL_{n|n}(A\otimes_\bR\bC):~\Pi \bar{g} =g\}$.
		This is a real form of $\GL_{n|n}$. We identify $\Res_{\bC/\bR}\GL_n\cong G$ for $g\mapsto \diag(g,\bar{g})$.
		Take the same $H$ and $B$ as Example \ref{ex:p(n)}. We identify $X^\ast(H\otimes_\bR\bC)$ with $(\bZ^n)^2$ in a similar way to Example \ref{ex:p(n)}. Then the complex conjugate action is given by
		\begin{equation}
			(\lambda,\lambda')\mapsto (\lambda',\lambda)\label{eq:switch}
		\end{equation}
		by \cite[Proposition 3.2.2]{MR4627704}. Let $\{e_i\}$ be the standard basis of $\bZ^{2n}$. The $\ast$-involution attached to $(H\otimes_\bR\bC,(\Res_{\bC/\bR} B)\otimes_\bR\bC)$ is equal to the complex conjugation. Since
		$\overline{e_1-e_{n+1}}=-(e_1-e_{n+1})$,
		${}^0\Qq(n)$ does not satisfy Assumption \ref{ass} for any $H$ unless $n=0$.
	\end{ex}

	\begin{ex}[Quaternion general linear supergroup]\label{ex:uast(2m|2n)}
		Put $F=\bR$. Let $m,n$ be nonnegative integers. Set $J_{m|n}=\diag(J_m,J_n)$.
		Define a linear algebraic supergroup $\Uu^\ast(2m|2n)$ by
		$\Uu^\ast(2m|2n;A)=\{g\in \GL_{2m|2n}(A\otimes_\bR\bC):~\bar{g} J_{m|n}=J_{m|n}g\}$.
		This is a real form of $\GL_{2m|2n}$ in a similar way to Example \ref{ex:u(p,q|r,s)}. In particular, it is a quasi-reductive algebraic supergroup over $\bR$. The even part is the product $\Uu^\ast(2m)\times\Uu^\ast(2n)$ of quaternion general linear groups.
		In this case, $\ast$-action is trivial for some and therefore any $H$ and $\gamma$ by \cite[Example 4.1.4]{MR4627704}. 
	\end{ex}
	
	\begin{ex}[Quaternion periplectic supergroup]\label{ex:p^ast}
		Put $F=\bR$. Let $n$ be nonnegative integer. Set
		$\Pp^\ast(2n)=\Uu^\ast(2n|2n)\cap \Res_{\bC/\bR} \Pp_{2n}$. This is a real form of the complex periplectic supergroup $\Pp_{2n}$. Its even part is isomorphic to $\Uu^\ast(2n)$. Hence the $\ast$-action is trivial for any $H$ and $\gamma$.
	\end{ex}
	
	\begin{ex}\label{ex:assumption}
		Put $F=\bR$. Then Assumption \ref{ass} holds if
		\begin{enumerate}
			\renewcommand{\labelenumi}{(\roman{enumi})}
			\item the Lie group $H(\bR)$ is compact, and
			\item the longest element of the Weyl group of $G\otimes_\bR\bC$ acts on $X^\ast(H\otimes_\bR\bC)$ as $-1$.
		\end{enumerate}
	\end{ex}
	
	\begin{ex}[Orthosymplectic supergroup]\label{ex:orthosymp}
		Put $F=\bR$. Take nonnegative integers $p,q,n$. Set $I_{n|p,q}=\diag(J_n,I_p,-I_q)$.
		Define a linear algebraic supergroup $\SpO(2n|p,q)$ by
		$\SpO(2n|p,q;A)=\{g\in \SL_{2n|p+q}(A):~g^{ST} I_{n|p,q}g= I_{n|p,q}\}$.
		Then $\SpO(2n|p,q)$ is a quasi-reductive algebraic supergroup over $\bR$ since this is a real form of $\SpO_{2n|p+q}$. Moreover, $\SpO(2n|p,q)$ satisfies the conditions of Example \ref{ex:assumption} for certain $H$ if and only if either
		\begin{enumerate}
			\renewcommand{\labelenumi}{(\roman{enumi})}
			\item $p+q$ is odd,
			\item $p,q$ are even and $p+q\equiv 0\pmod 4$.
		\end{enumerate}
		If $p,q$ are odd and $p+q\equiv 2\pmod 4$, then $\SpO(2n|p,q)$ does not satisfy the conditions of Example \ref{ex:assumption}, but the $\ast$-action is trivial for any $H$ and $\gamma$ (\cite[Example 4.1.10]{MR4627704}).
	\end{ex}

	\begin{ex}[Quaternion orthosymplectic supergroup]\label{ex:osp^ast}
		Put $F=\bR$. Take nonnegative integers $p,q,r$. Set
		\[J_{p,q|2r}=\left(\begin{array}{cccccc}
			0 & 0 & I_p & 0 & 0 & 0 \\
			0 & 0 & 0 & -I_q & 0 & 0 \\
			-I_p & 0 & 0 & 0 & 0 & 0 \\
			0 & I_q & 0 & 0 & 0 & 0 \\
			0 & 0 & 0 & 0 & 0 & I_r \\
			0 & 0 & 0 & 0 & -I_r & 0
		\end{array}\right).\]
		Define a linear algebraic supergroup $\SpO^\ast(p,q|2r)$ by
		\[\SpO^\ast(p,q|2r;A)=\{g\in \SpO_{2p+2q|2r}(A\otimes_\bR\bC):~
		\bar{g}J_{p,q|2r}=J_{p,q|2r}g\}.\]
		Then $\SpO^\ast(p,q|2r)$ is a real form of $\SpO_{p+q|2r}$. The even part is isomorphic to $\Sp(p,q)\times \SO^\ast(2r)$ of the indefinite unitary symplectic group\footnote{For the definition of $\Sp(p,q)$ in this paper, we adopt $\Sp(p,q)'$ in \cite[Example 3.3.9]{MR4627704}.} and the quaternion special orthogonal group. In particular, $\SpO^\ast(p,q|2r)$ satisfies the conditions of Example \ref{ex:assumption} if and only if $r$ is even.
	\end{ex}

	We shall see whether the last two groups can satisfy Assumption \ref{ass} in terms of $\epsilon\delta$-sequences:
	
	\begin{ex}\label{ex:SpO_advanced}
		Put $F=\bR$. Take nonnegative integers $m,p,q$ with $n\coloneqq p+q$ odd. Set $\bmG=\SpO(2m|2p,2q)$. Take the product of the maximal tori of diagonal matrices of $\Sp_{2m}$\footnote{Here $\Sp_{2m}$ is the split symplectic group of rank $m$ in the usual sense. In the notation of \cite[Section 1.5]{MR4627704}, this agrees with $\Sp_m$.} and of the indefinite special orthogonal group $\SO(2p,2q)$ in \cite[Example 3.4.8]{MR4627704} for a maximal torus $H$ of $G$. Identify $X^\ast(H\otimes_\bR\bC)$ with $\bZ^{m+n}$ by \cite[Section 3.2]{MR4627704}. Write $\{e_i\}$ for the standard basis of $\bZ^{m+n}$. Consider the standard positive system
		\[\begin{split}
			\Delta^+_{\bar{0}}
			&=\{e_i\pm e_j\in X^\ast(H\otimes_\bR\bC):~1\leq i<j\leq m,
			~m+1\leq i<j\leq m+n\}\\
			&\cup\{2e_i\in X^\ast(H\otimes_\bR\bC):~1\leq i\leq m\}
		\end{split}\]
		of the even part. Then $w$ is given by the product of that in \cite[Example 3.4.8]{MR4627704} and the unit matrix $I_{2n}$. In fact, the action of $\Gamma$ and $w$ on $X^\ast(H\otimes_\bR\bC)$ are given by
		\[\bar{\lambda}
		=(\lambda_1,\lambda_2,\ldots,\lambda_m,-\lambda_{m+1},-\lambda_{m+2},\ldots,-\lambda_{m+n})\]
		\[w\lambda=(\lambda_1,\lambda_2,\ldots,\lambda_m,-\lambda_{m+1},
		-\lambda_{m+2},\ldots,-\lambda_{m+n-1},\lambda_{m+n})\]
		Therefore $\Gamma$ switches the sign of $\epsilon$ in $\epsilon\delta$ sequences if $\delta$ lies at the right end; $\Gamma$ fixes the other $\epsilon\delta$ sequences.
		For example, $\delta\delta\cdots \delta\epsilon\epsilon\cdots\epsilon$
		is fixed by $\Gamma$. This corresponds to the standard positive system in \cite[Section 1.3.4]{MR3012224}. Explicitly, the set $\Pi$ of simple roots is given by
		\[\begin{split}
			\Pi&=\{e_i-e_{i+1}\in X^\ast(H\otimes_\bR\bC):~1\leq i\leq m-1,~m+1\leq i\leq m+n-1\}\\
			&\cup\{e_m-e_{m+1},e_{m+n-1}+e_{m+n}\}.
		\end{split}\]
	\end{ex}
	
	\begin{ex}\label{ex:spo(m|2p+1,2q+1)}
		Put $F=\bR$. Take nonnegative integers $m,p,q$ with $n\coloneqq p+q$ odd. Set $\bmG=\SpO(2m|2p+1,2q+1)$. Take the product of the maximal tori of diagonal matrices of $\Sp_{2m}$ and of the indefinite special orthogonal group $\SO(2p+1,2q+1)$ in \cite[Example 3.4.9]{MR4627704} for a maximal torus $H$ of $G$. Identify $X^\ast(H\otimes_\bR\bC)$ with $\bZ^{m+n+1}$ by \cite[Section 3.2 and Example 3.3.5]{MR4627704}. Write $\{e_i\}$ for the standard basis of $\bZ^{m+n+1}$. Consider the standard positive system
		\[\begin{split}
			\Delta^+_{\bar{0}}
			&=\{e_i\pm e_j\in X^\ast(H\otimes_\bR\bC):~1\leq i<j\leq m,
			~m+1\leq i<j\leq m+n+1\}\\
			&\cup\{2e_i\in X^\ast(H\otimes_\bR\bC):~1\leq i\leq m\}
		\end{split}\]
		of the even part. Then $w$ is given by the diagonal matrix $w$ whose diagonal entries are given by
		\[(\overbrace{1,1,\ldots,1}^{2m},\overbrace{1,-1,1,-1,\ldots,1,-1}^{2p},
		1,1,\overbrace{1,-1,\ldots,1,-1}^{2q-2},1,1).\]
		In fact, the action of $\Gamma$ and $w$ on $X^\ast(H\otimes_\bR\bC)$ are given by
		\[\bar{\lambda}
		=(\lambda_1,\lambda_2,\ldots,\lambda_m,-\lambda_{m+1},-\lambda_{m+2},\ldots,
		-\lambda_{m+p},\lambda_{m+p+1},-\lambda_{m+p+1},\ldots,-\lambda_{m+n+1})\]
		\[w\lambda=(\lambda_1,\ldots,\lambda_m,-\lambda_{m+1},
		-\lambda_{m+2},\ldots,-\lambda_{m+p},\lambda_{m+p+1},-\lambda_{m+p+2},\ldots,
		-\lambda_{m+n},\lambda_{m+n+1}).\]
		Therefore $\Gamma$ switches the sign of $\epsilon$ in $\epsilon\delta$ sequences if $\delta$ lies at the right end; $\Gamma$ fixes the other $\epsilon\delta$ sequences. The standard positive system in \cite[Section 1.3.4]{MR3012224} is again fixed by $\Gamma$. 
	\end{ex}
	
	\begin{ex}
		Put $F=\bR$. Take nonnegative integers $p,q,r$. Set $n=p+q$ and $\bmG=\SpO^\ast(p,q|4r+2)$. Take the product of the maximal tori in \cite[Examples 3.4.7 and 3.4.10]{MR4627704} for a maximal torus $H$ of $G$. Identify $X^\ast(H\otimes_\bR\bC)$ with $\bZ^{m+n}$ by \cite[Section 3.2]{MR4627704}. Write $\{e_i\}$ for the standard basis of $\bZ^{n+2r+1}$. Consider the standard positive system
		\[\begin{split}
			\Delta^+_{\bar{0}}
			&=\{e_i\pm e_j\in X^\ast(H\otimes_\bR\bC):~1\leq i<j\leq m,
			~m+1\leq i<j\leq m+n\}\\
			&\cup\{2e_i\in X^\ast(H\otimes_\bR\bC):~1\leq i\leq m\}
		\end{split}\]
		of the even part. Since $H(\bR)$ is compact, $w$ is the longest element in the Weyl group. See \cite[Examples 4.1.7 and 4.1.13]{MR4627704} for a lift $w$ in $G(\bR)$. We now obtain an equality
		$w^{-1}\bar{\lambda}=(\lambda_1,\lambda_2,\ldots,\lambda_{n+2r},-\lambda_{m+2r+1})$
		for $\lambda=(\lambda_i)\in X^\ast(H\otimes_\bR\bC)$. The same argument as Example \ref{ex:SpO_advanced} again implies that the standard positive system in \cite[Section 1.3.4]{MR3012224} is fixed by the $\ast$-involution.
	\end{ex}
	
	There is a completely different situation:
	
	\begin{ex}
		Any $F$-form of the split queer supergroup over $F^{\sep}$ satisfies the assumption since the sets of roots in the even and odd parts are equal in this case.
	\end{ex}
	
	\begin{ex}[Indefinite unitary queer supergroup]
		Put $F=\bR$. Let $p,q$ be nonnegative integers. Set
		$\Qq(p,q)\coloneqq \Uu(p,q|p,q)\cap \Res_{\bC/\bR}\Qq_{p+q}$.
		This is a real form of $\Qq_{p+q}$. Its even part is the indefinite unitary group $\Uu(p,q)$.
	\end{ex}
	
	\begin{ex}[Quaternion queer supergroup]
		Put $F=\bR$. Let $n$ be a nonnegative integer. Define a real form $\Qq^\ast(2n)$ of $\Qq_{2n}$ by
		\[\Qq^\ast(2n;A)=\Uu^\ast(2n|2n) \cap \Res_{\bC/\bR}\Qq_{2n}.\]
		Its even part is the quaternion general linear group $\Uu^\ast(2n)$.
	\end{ex}
	
	Henceforth suppose that Assumption \ref{ass} holds. Let $\fu$ be an irreducible representation of $\bmB'$. We observe that ${}^\sigma\fu$ is an irreducible representation of the BPS-subgroup
	${}^\sigma\bmB'=w_\sigma\bmB'w^{-1}_{\sigma}$.

	For each $\lambda\in X^\ast(H\otimes_F F^{\sep})$, let us choose a corresponding irreducible representation $\fu(\lambda)=\fu_{F^{\sep}}(\lambda)$ of $\bmH\otimes_F F^{\sep}$ according to the bijection \eqref{eq:IrrH'}.
	
	\begin{thm}\label{thm:Galois_action}
		Let $\bmG$ be a quasi-reductive algebraic supergroup over a field $F$ of characteristic not two with a maximal torus $H\subset G$. Suppose that we are given a BPS-subgroup $\bmB'\subset \bmG\otimes_F F^{\sep}$ and $w_\sigma\in G(F^{\sep})$ in Assumption \ref{ass} (see also the paragraph above Notation-Definition \ref{notedefn} for $w_\sigma$).
		\begin{enumerate}
			\item For an irreducible representation $\fu$ of $\bmB'$, there are isomorphisms
			\begin{equation}
				{}^\sigma\Ind^{\bmG\otimes_F F^{\sep}}_{\bmB'} \fu\cong
				\Ind^{\bmG\otimes_F F^{\sep}}_{\bmB'} {}^{w_\sigma}({}^\sigma\fu),
				\label{eq:Galois_twist_Ind}
			\end{equation}
			\begin{equation}
				{}^\sigma V(\fu)\cong V({}^{w_\sigma}({}^\sigma\fu)).
				\label{eq:Galois_twist_V(u)}
			\end{equation}
			\item For $\sigma\in\Gamma$ and $\lambda\in X^\flat(H\otimes_F F^{\sep})$, we have
			${}^{w_\sigma}({}^\sigma\fu(\lambda))\overset{\boldsymbol{\cdot}}{\cong}
			\fu(\sigma\ast\lambda)$.
		\end{enumerate} 
	\end{thm}
	
	\begin{proof}
		The first isomorphism \eqref{eq:Galois_twist_Ind} in (1) is formal:
		\[{}^\sigma\Ind^{\bmG\otimes_F F^{\sep}}_{\bmB'} \fu
		\cong \Ind^{\bmG\otimes_F F^{\sep}}_{{}^\sigma\bmB'} {}^\sigma\fu
		=\Ind^{\bmG\otimes_F F^{\sep}}_{w_\sigma\bmB'w^{-1}_{\sigma}} {}^\sigma\fu
		\overset{\eqref{eq:w-twist_Ind}}{\cong}
		\Ind^{\bmG\otimes_F F^{\sep}}_{\bmB'} {}^{w_\sigma}({}^\sigma\fu(\lambda)).\]
		Take their socles to obtain \eqref{eq:Galois_twist_V(u)} since the Galois twists respect socles.
		
		Part (2) is straightforward (see the weight).
	\end{proof}
	
	\begin{cor}\label{cor:superpure}
		Consider the setting of Theorem \ref{thm:Galois_action}.
		\begin{enumerate}
			\item The subset $X^\flat(H\otimes_F F^{\sep})\subset X^\ast(H\otimes_F F^{\sep})$ is $\Gamma$-invariant.
			\item The map \eqref{eq:superhwtheory'} is $\Gamma$-equivariant, where $\Gamma$ acts on $\Irr_{\Pi} \bmG\otimes_F F^{\sep}$ by the Galois twist.
			\item The irreducible representation $V(\fu(\lambda))$ is super quasi-rational if and only if ${}^\sigma\lambda=w_\sigma \lambda$ for all $\sigma\in\Gamma$.
		\end{enumerate}
	\end{cor}
	
	\begin{proof}
		The assertions (1) and (2) follow from Theorem \ref{thm:Galois_action}. Part (3) is an immediate consequence of (2) and Corollary \ref{cor:schur's_lemma}.
	\end{proof}

	For a complete classification of irreducible representations of $\bmG$, we are interested in judging ${}^{w_\sigma}({}^\sigma\fu)\cong\fu$ and ${}^{w_\sigma}({}^\sigma\fu)\cong\Pi\fu$ (cf.~Corollary \ref{cor:uknowsv}). In particular, we have to handle with the twist by $w_\sigma$ in addition to the Galois twist. This obstructs us to achieve an inductive argument in \cite[Proof of Theorem 4]{MR0167498}. We note that the first isomorphism in question is also a key observation to the descent problem. We wish to deal with these questions by application of \cite[Proposition 4.2]{MR4039427}. For $\lambda\in X^\ast(H\otimes_F F^{\sep})$, set $q^\lambda=q^\lambda_{F^{\sep}}$, $d_\lambda=d_{\lambda,F^{\sep}}$, $\delta_\lambda=\delta_{\lambda,F^{\sep}}$ as before. That is, $q^\lambda$ is a quadratic form on $\fh_{\bar{1}}\otimes_F F^{\sep}$ defined by
	$q^\lambda(x)=\frac{1}{2}\lambda([x,x])$.
	The nonnegative integer $d_{\lambda}$ is the dimension of a complementary subspace to the radical of $q^\lambda$ in $\fh_{\bar{1}}\otimes_F F^{\sep}$.
	According to Example \ref{ex:type_sep_clo_case}, $\delta_\lambda$ is simply given by
	\begin{equation}
		\delta_{\lambda}=\begin{cases}
			1&(q^\lambda\neq 0)\\
			0&(q^\lambda= 0).\label{eq:delta_lambda}
		\end{cases}
	\end{equation}
	
	\begin{lem}\label{lem:delta_lambda}
		For $\lambda\in X^\flat(H\otimes_F F^{\sep})$ and $\sigma\in\Gamma$, $\delta_{\sigma\ast\lambda}$ is nonzero if and only if so is $\delta_\lambda$.
	\end{lem}
	
	\begin{proof}
		This is evident by \eqref{eq:delta_lambda} and definitions.
	\end{proof}
	
	\begin{conv}\label{conv}
		Let $\lambda\in X^\ast(H\otimes_F F^{\sep})$ with $\delta_\lambda=0$. According to the proofs of \cite[Proposition 4.2]{MR4039427} or of Proposition \ref{prop:split} and Corollary \ref{cor:split_h1=0}, the irreducible representations of $\bmH\otimes_F F^{\sep}$ corresponding to $\lambda$ are of one dimension. Henceforth we choose the even one for $\fu(\lambda)$.
	\end{conv}

	\begin{thm}\label{thm:u(lambda)_easy_case}
		Consider the setting of Theorem \ref{thm:Galois_action}.
		Let $\lambda\in X^\flat(H\otimes_F F^{\sep})$. Define $\fu(\lambda)$, $d_\lambda$, $\delta_\lambda$ as above.
		\begin{enumerate}
			\item If $\delta_\lambda=0$ then we have
			${}^{w_\sigma}({}^\sigma\fu(\lambda))\cong\fu(\sigma\ast\lambda)
			\not\cong\Pi \fu(\sigma\ast\lambda)$
			for all $\sigma\in\Gamma$.
			\item If $\delta_\lambda\neq 0$ and $d_\lambda$ odd then we have
			${}^{w_\sigma}({}^\sigma\fu(\lambda))\cong\fu(\sigma\ast\lambda)
			\cong\Pi \fu(\sigma\ast\lambda)$.
		\end{enumerate}
	\end{thm}
	
	\begin{proof}
		Thanks to our choice of $\fu(\lambda)$ in Convention \ref{conv}, (1) follows from Lemma \ref{lem:delta_lambda}. Part (2) follows from Theorem \ref{thm:Galois_action} (2) and \cite[Proposition 4.2]{MR4039427}. 
	\end{proof}
	
	\begin{cor}\label{cor:compute_a}
		Consider the setting of Theorem \ref{thm:Galois_action}.
		Let $\lambda$ be a $\Gamma$-invariant element of $X^\flat(H\otimes_F F^{\sep})$.
		\begin{enumerate}
			\item If $\delta_\lambda=0$, $V(\fu(\lambda))$ is quasi-rational. Moreover, we have
			\[\begin{array}{cc}
				\epsilon_{V(\fu(\lambda))_F}=+,&a_{V(\fu(\lambda))_F}=1.
			\end{array}\]
			\item If $d_\lambda$ is odd (in particular, $\delta_\lambda\neq 0$), $V(\fu(\lambda))$ is quasi-rational. Moreover, we have $\epsilon_{V(\fu(\lambda))_F}=-$.
		\end{enumerate}
	\end{cor}
	
	\begin{cor}\label{cor:h_1=0case}
		Consider the setting of Theorem \ref{thm:Galois_action}. Suppose $\fh_{\bar{1}}=0$.
		\begin{enumerate}
			\item The quotient map $\Gamma\backslash\Irr\bmG\otimes_F F^{\sep}\to \Gamma\backslash\Irr_{\Pi}\bmG\otimes_F F^{\sep}$ is two-to-one.
			\item The two classes in each fiber of (1) are distinguished by $H^0((\bmU')^+,-)$. To be precise, we denote the set of irreducible representations $V$ of $\bmG\otimes_F F^{\sep}$ with $H^0((\bmU')^+,V)_{\bar{1}}=0$ by $\Irr_{\bar{0}}\bmG\otimes_F F^{\sep}$. Then $\Irr_{\bar{0}}\bmG\otimes_F F^{\sep}$ is a $\Gamma$-subset of $\Irr\bmG\otimes_F F^{\sep}$. Moreover, the map of (1) restricts to a bijection
			\[\Gamma\backslash\Irr_{\bar{0}}\bmG\otimes_F F^{\sep}
			\cong \Gamma\backslash\Irr_{\Pi}\bmG\otimes_F F^{\sep}.\]
			\item For $\lambda\in X^\flat(H\otimes_F F^{\sep})$, the following conditions are equivalent:
			\begin{enumerate}
				\renewcommand{\labelenumi}{(\roman{enumi})}
				\item $V(\fu(\lambda))$ is quasi-rational;
				\item $V(\fu(\lambda))$ is super quasi-rational;
				\item $\sigma\ast\lambda=\lambda$ for every $\sigma\in\Gamma$.
			\end{enumerate}
		\end{enumerate}
	\end{cor}
	
	This gives a complete classification of irreducible representations of $\bmG$ under the hypothesis $\fh_{\bar{1}}=0$ (modulo the description of $X^\flat(H\otimes_F F^{\sep})$). See Examples \ref{ex:res}, \ref{ex:p(n)}, \ref{ex:uast(2m|2n)}, \ref{ex:p^ast}, \ref{ex:assumption}, \ref{ex:orthosymp}, and \ref{ex:osp^ast} for instances of $\bmG$ satisfying this condition. The split periplectic supergroups $\Pp_n$ also satisfy the condition $\fh_{\bar{1}}=0$.
	
	To complete the classification problem when $\fh_{\bar{1}}\neq 0$, it remains to discuss the case that $d_\lambda$ is positive even. As we can see in Example \ref{ex:Q_n_real}, whether
	\[{}^{w_\sigma}({}^\sigma\fu(\lambda))\cong\fu(\sigma\ast\lambda)\]
	depends on $\lambda$ in general. When $\delta_\lambda=1$, $a_{V(\fu(\lambda))_F}$ also depends on $\lambda$ in general. We should need a deeper analysis to deal with them. Since we no longer have the Clifford algebra over $F$, we may have to work with the ``twisted'' descent problem by $w_\sigma$. We leave this hard part to latter sections.
	
	To close this section, let us see how we can compute $D_{V(\fu(\lambda))_F}$. Pick a character $\lambda\in X^\flat(H\otimes_F F^{\sep})$
	with $V(\fu(\lambda))$ super quasi-rational. Recall that we chose a finite Galois extension $F'/F$ above Notation-Definition \ref{notedefn}. For the continuity of the cocycle, we may enlarge $F'$ if necessary to assume $\fu_{F'}(\lambda)\otimes_{F'} F^{\sep}\cong \fu(\lambda)$ through a similar argument to Proposition \ref{prop:rationality}.
	
	For each element $\bar{\sigma}$ of the Galois group of $F'/F$, one can and does choose an isomorphism
	\[\begin{array}{cc}
		\varphi_{\bar{\sigma}}:{}^{w_{\bar{\sigma}}}({}^{\bar{\sigma}}\fu_{F'}(\lambda))
		\cong \fu_{F'}(\lambda)
		&(\mathrm{resp.~}\varphi_{\bar{\sigma}}:
		{}^{w_{\bar{\sigma}}}({}^{\bar{\sigma}}\fu_{F'}(\lambda))
		\overset{\boldsymbol{\cdot}}{\cong} \fu_{F'}(\lambda))
	\end{array}\]
	if $V(\fu(\lambda))$ is quasi-rational (resp.~not quasi-rational) by Theorem \ref{thm:Galois_action} (1) and Corollary \ref{cor:uknowsv}. We lift them to define an isomorphism
	\[\begin{array}{cc}
		\varphi_\sigma:{}^{w_\sigma}({}^\sigma\fu(\lambda))\cong \fu(\lambda)&(\mathrm{resp.~}\varphi_\sigma:{}^{w_\sigma}({}^\sigma\fu(\lambda))\overset{\boldsymbol{\cdot}}{\cong} \fu(\lambda))
	\end{array}\]
	for $\sigma\in\Gamma$ and $V(\fu(\lambda))$ quasi-rational (resp.~not quasi-rational).
	
	Henceforth to think of the two cases simultaneously, let us omit $\Pi$. In a sequel, the symbol $\overset{\boldsymbol{\cdot}}{\cong}$ is replaced with $\cong$ when $V(\fu(\lambda))$ is quasi-rational.
	
	The chosen isomorphisms give rise to a sequence
	\[{}^\sigma\Ind^{\bmG\otimes_F F^{\sep}}_{\bmB'} \fu(\lambda)
	\overset{\eqref{eq:Galois_twist_Ind}}{\cong} 
	\Ind^{\bmG\otimes_F F^{\sep}}_{\bmB'} {}^{w_\sigma}({}^\sigma\fu(\lambda)) 
	\xcong{\overset{\Ind^{\bmG\otimes_F F^{\sep}}_{\bmB'}\varphi_\sigma}{\boldsymbol{\cdot}}}
	\Ind^{\bmG\otimes_F F^{\sep}}_{\bmB'} \fu(\lambda).
	\]
	It restricts to ${}^\sigma V(\fu(\lambda))\overset{\boldsymbol{\cdot}}{\cong} V(\fu(\lambda))$, which we denote by $\Phi_\sigma$. The commutative diagram
	\begin{equation}
		\begin{tikzcd}
			V_{\bmB'}(\fu(\lambda))\ar[r, "\overset{V_{\bmB'}(\varphi^{-1}_{\sigma\tau})}{\boldsymbol{\cdot}}"]
			&V_{\bmB'}({}^{w_{\sigma\tau}}({}^{\sigma\tau}\fu(\lambda)))\ar[r, "w^{-1}_{\sigma\tau}"]
			\ar[rd, "w_{\sigma\tau}^{-1}\sigma(w_\tau)"']
			\ar[dd, "w_{\sigma\tau}^{-1}\sigma(w_\tau)w_\sigma"']
			&V_{{}^{\sigma\tau}\bmB'}({}^{\sigma\tau}\fu(\lambda))\ar[d, "\sigma(w_\tau)"]\\
			&&V_{{}^{\sigma}\bmB'}({}^{\sigma(w_\tau)}({}^{\sigma\tau}\fu(\lambda)))\ar[d, equal]\\
			&V_{\bmB'}({}^{w_\sigma}({}^\sigma({}^{w_{\tau}}({}^{\tau}\fu(\lambda)))))
			\ar[d, "V_{\bmB'}({}^{w_\sigma}({}^\sigma\varphi_\tau))\boldsymbol{\cdot}"']
			&V_{{}^{\sigma}\bmB'}({}^\sigma({}^{w_{\tau}}({}^{\tau}\fu(\lambda))))
			\ar[d, "\boldsymbol{\cdot} V_{{}^\sigma\bmB'}({}^\sigma\varphi_\tau)"]\ar[l, "w_\sigma"']\\
			V_{\bmB'}(\fu(\lambda))&V_{\bmB'}({}^{w_{\sigma}}({}^{\sigma}\fu(\lambda)))
			\ar[l, "\overset{V_{\bmB'}(\varphi_\sigma)}{\boldsymbol{\cdot}}"']
			&V_{{}^{\sigma}\bmB'}({}^{\sigma}\fu(\lambda))\ar[l, "w_\sigma"']
		\end{tikzcd}\label{diag:BT}
	\end{equation}
	implies
	\[\Phi_\sigma\circ{}^{\sigma}\Phi_\tau\circ\Phi^{-1}_{\sigma\tau}
	=\lambda(w_{\sigma\tau}^{-1}\sigma(w_\tau)w_\sigma)
	V_{\bmB'}(\varphi_\tau\circ {}^{w_\sigma}({}^\sigma\varphi_\tau)\circ \varphi^{-1}_{\sigma\tau}).
	\]
	Here the formula
	$\varphi_\tau\circ {}^{w_\sigma}({}^\sigma\varphi_\tau)\circ \varphi^{-1}_{\sigma\tau}$ makes a sense in the following way: Firstly, think of 
	$\varphi^{-1}_{\sigma\tau}:\fu(\lambda)\overset{\boldsymbol{\cdot}}{\cong}
	{}^{w_{\sigma\tau}}({}^{\sigma\tau}\fu(\lambda))$.
	Observe that ${}^\sigma \varphi_\tau$ is a homomorphism between the representations ${}^{\sigma}({}^{w_\tau}({}^\tau \fu(\lambda)))$ and ${}^{\sigma}\fu(\lambda)$ of ${}^\sigma\bmB'=w_\sigma \bmB'w^{-1}_{\sigma}$. Take the pullback by $w_\sigma$ to obtain
	${}^{w_\sigma}({}^{\sigma}({}^{w_\tau}({}^\tau \fu(\lambda))))\overset{\boldsymbol{\cdot}}{\cong}
	{}^{w_\sigma}({}^\sigma\fu(\lambda))$.
	One can naturally identify the domain with
	${}^{w_{\sigma\tau}}({}^{\sigma\tau}\fu(\lambda))$
	by restricting $\bmB'$ to $\bmH\otimes_F F^{\sep}$
	since we have $\sigma(w_\tau)w_\sigma=w_{\sigma\tau}$
	at the level of the Weyl group (cf.~\cite[Section 2.1]{MR4627704}). Here we used the fact that $H$ centralizes $\bmH$ to show that the twists by $w_\sigma\sigma(w_\tau)$ and $w_{\sigma\tau}$ are equal. This fact follows by definition of $\bH$.
	
	One can write
	$\varphi_\tau\circ {}^{w_\sigma}({}^\sigma\varphi_\tau)\circ \varphi^{-1}_{\sigma\tau}=
	c_{\lambda}(\sigma,\lambda)\id_{\fu(\lambda)}$
	for some element $c_{\lambda}(\sigma,\lambda)\in (F^{\sep})^\times$. The Borel--Tits cocycle is now computed as
	\[\beta^{\BT}_{V(\fu(\lambda))}(\sigma,\tau)=\lambda(w_{\sigma\tau}^{-1}\sigma(w_\tau)w_\sigma)
	c_\lambda(\sigma,\tau).\]
	By abuse of notation, let us write
	\[\begin{array}{cc}
		\beta^{\BT}_\lambda(\sigma,\tau)\coloneqq \lambda(w_{\sigma\tau}^{-1}\sigma(w_\tau)w_\sigma),
		&\beta^{\BT}_{\fu(\lambda)}(\sigma,\tau)=c_\lambda(\sigma,\tau).
	\end{array}\]
	to define 2-cocycles $\beta^{\BT}_\lambda$ and $\beta^{\BT}_{\fu(\lambda)}$. We shall summarize the preceding argument as a statement:
	
	\begin{thm}[Product formula]\label{thm:prod_formula}
		Consider the setting of Theorem \ref{thm:Galois_action}.
		Let $\lambda\in X^\flat(H\otimes_FF^{\sep})$ with $V(\fu(\lambda))$ super quasi-rational. Then we have an equality
		$\beta^{\BT}_{V(\fu(\lambda))}=\beta^{\BT}_\lambda \beta^{\BT}_{\fu(\lambda)}$
		in $H^2(\Gamma,(F^{\sep})^\times)$.
	\end{thm}
	
	\begin{rem}\label{rem:BT}
		A similar argument for $G$ implies that $\beta^{\BT}_\lambda$ exhibits the Borel--Tits cocycle of the irreducible representation of $G\otimes_F F^{\sep}$ of highest weight $\lambda$. Morally speaking, the product formula therefore says that $\beta^{\BT}_{V(\fu(\lambda))}$ is determined by the contribution from representation theory of the even part and from $\fu(\lambda)$. For explicit computation of $\beta^{\BT}_\lambda$ in the case $F=\bR$, see \cite{E1914} and its subsequent works (e.g.~\cite{MR0209401,MR1611385,MR2041548,MR4627704}). 
	\end{rem}
	
	\begin{ex}
		Consider the setting of Theorem \ref{thm:prod_formula}. Assume $\delta_\lambda=0$. Then $\beta^{\BT}_{\fu(\lambda)}$ is a boundary cocycle since $\dim_{F^{\sep}}\fu(\lambda)=1$. In particular, we get $\beta^{\BT}_{V(\fu(\lambda))}=\beta^{\BT}_\lambda$.
	\end{ex}
	
	\begin{cor}\label{cor:h_1=0case_beta}
		Consider the setting of Theorem \ref{thm:prod_formula}. Suppose $\fh_{\bar{1}}=0$. Let $\lambda\in X^\flat(H\otimes_F F^{\sep})$ with $V(\fu(\lambda))$ super quasi-rational. Then we have
		\[\beta^{\super}_{V(\fu(\lambda))_F}=[\End_{\bmG}(V(\fu(\lambda))_F)]
		=(+,1,(\beta^{\BT}_{\lambda})^{-1}).\]
	\end{cor}
	
	\section{Quasi-reductive case II}\label{sec:qredII}
	
	In the last section, we explained how to classify irreducible representations of quasi-reductive algebraic supergroups under Assumption \ref{ass}. We partially determined the division superalgebras of their endomorphisms. In this section, we discuss the remaining issues on these topics in examples. Throughout this section, we consider the same setting as that at the beginning of Section \ref{sec:qred}. Choose $w_\sigma$ as in Section \ref{sec:general}.

	\subsection{Basic case}\label{sec:odd_reflection}
	
	In this small section, we assume $\bmG$ to be basic of the main type in order to remove Assumption \ref{ass}. In virtue of a minor modification of \cite[Proof of Proposition 1.32]{MR3012224}, one can find a sequence $\alpha(\sigma,1),\alpha(\sigma,2)\ldots,\alpha(\sigma,\nu_\sigma)$ of isotropic odd roots such that for each $0\leq i\leq \nu_\sigma-1$, $\alpha(\sigma,i+1)$ is simple in the positive system 
	$s_{\alpha(\sigma,i)}s_{\alpha(\sigma,i-1)}\cdots s_{\alpha(\sigma,1)}
	w^{-1}_\sigma{}^\sigma\Delta^+$
	and
	$s_{\alpha(\sigma,\nu_\sigma)}s_{\alpha(\sigma,\nu_\sigma-1)}\cdots s_{\alpha(\sigma,1)}
	w^{-1}_\sigma{}^\sigma\Delta^+=\Delta^+$.
	
	\begin{ex}\label{ex:u(odd|odd)}
		Put $F=\bR$. Let $p,q,r,s$ be nonnegative integers with $(p+q)(r+s)$ odd. Set $m\coloneqq \frac{p+q-1}{2}$ and $n\coloneqq \frac{r+s-1}{2}$.
		Put $\bmG=\Uu(p,q|r,s)$. We may assume
		\[2m+1=p+q \geq r+s=2n+1.\]
		Let $H,\{e_i\},\Delta^+_{\bar{0}},w$ be as in Example \ref{ex:u(p,q|r,s)}. Consider the positive system $\Delta^+$ containing $\Delta^+_{\bar{0}}$ and corresponding to the following $\epsilon\delta$-sequence:
		\[\overbrace{\epsilon\delta\epsilon\delta\cdots \epsilon\delta}^{2n}
		\overbrace{\epsilon\epsilon\cdots \epsilon}^{m-n} \epsilon\delta \overbrace{\epsilon\epsilon\cdots\epsilon}^{m-n} 
		\overbrace{\delta\epsilon \delta\epsilon\cdots \delta\epsilon}^{2n}.
		\]
		Set $\alpha\coloneqq e_{m+1}-e_{2m+n+2}$. This root corresponds to the middle term $\epsilon\delta$ in the sequence above. Recall that $w^{-1}\bar{\Delta}^+$ is the reverse to the above:
		\[\overbrace{\epsilon\delta\epsilon\delta\cdots \epsilon\delta}^{2n}
		\overbrace{\epsilon\epsilon\cdots \epsilon}^{m-n} \delta\epsilon \overbrace{\epsilon\epsilon\cdots\epsilon}^{m-n} 
		\overbrace{\delta\epsilon \delta\epsilon\cdots \delta\epsilon}^{2n}.
		\]
		Since odd reflections switch $\epsilon$ and $\delta$, we have $s_\alpha w^{-1}\bar{\Delta}^+=\Delta^+$.
	\end{ex}
	
	\begin{ex}\label{ex:{}^0q(n)}
		Put $F=\bR$. Let $n$ be a positive integer.
		Set $\bmG={}^0\Qq(n)$. Let $H,\{e_i\},B$ be as in Example \ref{ex:0q(n)}. The positive system $\Delta^+_{\bar{0}}$ attached to
		\[(H\otimes_\bR\bC,(\Res_{\bC/\bR} B)\otimes_\bR\bC)\]
		is $\{e_i-e_j\in \bZ^{2n}:~1\leq i<j\leq n,~n+1\leq i<j\leq 2n\}$.
		In this case, $w$ can be the unit.
		Consider the following $\epsilon\delta$-sequence:
		$\epsilon\delta\epsilon\delta\cdots \epsilon\delta$.
		
		Recall that the complex conjugate action on $X^\ast(H\otimes_\bR\bC)$ switches the even and odd parts by \eqref{eq:switch}. This implies that $\Gamma$ switches $\epsilon$ and $\delta$. One can therefore recover the sequence above from its Galois involution by taking the odd reflections corresponding to $\delta\epsilon$ in the $(2i-1,2i)$th terms in order ($1\leq i\leq n$). Namely, we set
		$\alpha_i=e_i-e_{i+n}$ ($1\leq i\leq n$).
		Then we have
		$s_{\alpha_n}\cdots s_{\alpha_2}s_{\alpha_1}w^{-1}\bar{\Delta}^+=\Delta^+$.
	\end{ex}
	
	Let $\lambda\in X^\flat(H\otimes_F \bar{F})$ and $\sigma\in\Gamma$. Since $\fh_{\bar{1}}=0$, we may put $\fu(\lambda)=\bar{F}_\lambda$ as in the end of Section \ref{sec:split}. We would like to compare ${}^\sigma V(\bar{F}_\lambda)$ with $V(\bar{F}_\lambda)$ and $\Pi V(\bar{F}_\lambda)$. We note that ${}^\sigma V(\bar{F}_\lambda)$ is isomorphic to at most one of them since $\fh_{\bar{1}}=0$ (Example \ref{ex:type_sep_clo_case}). 
	
	Take a bilinear form $(-,-)$ on $\fh\otimes_F\bar{F}$ as in \cite[Theorem 1.18 (7)]{MR3012224}. To compute ${}^\sigma V(\bar{F}_\lambda)$, we define a sequence
	$(\lambda(\sigma,i))_{0\leq i\leq \nu_\sigma}$ by
	$\lambda(\sigma,0)=w^{-1}_\sigma{}^\sigma\lambda$ and
	\[\lambda(\sigma,i+1)=\begin{cases}
		\lambda(\sigma,i)&((\alpha(\sigma,i+1),\lambda(\sigma,i))=0)\\
		\lambda(\sigma,i)-\alpha(\sigma,i+1)&((\alpha(\sigma,i+1),\lambda(\sigma,i))\neq 0).
	\end{cases}\]
	Set $I(\lambda,\sigma)
	=\{i\in\{1,2,\ldots,\nu_\sigma\}:~(\alpha(\sigma,i),\lambda(\sigma,i-1))\neq 0\}$.
	We choose a root vector $E_{\alpha(\sigma,i)}$ of root $\alpha(\sigma,i)$ for each $i\in I(\lambda)$.
	
	Proposition \ref{prop:G-conjugate} and iterated use of Proposition \ref{prop:hw_odd_reflection} imply:
	
	\begin{thm}\label{thm:galois_twist_basic}
		The data above determine an isomorphism
		\[\Phi_\sigma:{}^\sigma V(\bar{F}_\lambda)\cong\begin{cases}
			V(\bar{F}_{\lambda(\sigma,\nu_\sigma)})&(|I(\lambda,\sigma)|\in 2\bZ_{\geq 0})\\
			\Pi V(\bar{F}_{\lambda(\sigma,\nu_\sigma)})&(|I(\lambda,\sigma)|\in 2\bZ_{\geq 0}+1).\\
		\end{cases}\]
	\end{thm}
	
	\begin{cor}
		\begin{enumerate}
			\item The irreducible representation $V(\bar{F}_\lambda)$ is super quasi-rational if and only if $\lambda(\sigma,\nu_\sigma)=\lambda$ for every $\sigma\in\Gamma$.
			\item The irreducible representation $V(\bar{F}_\lambda)$ is quasi-rational if and only if we have
			$\lambda(\sigma,\nu_\sigma)=\lambda$
			and $|I(\lambda,\sigma)|$ is even for every element $\sigma\in\Gamma$. If these equivalent conditions are satisfied, we have $a_{V(\bar{F}_\lambda)_F}=1$.
		\end{enumerate}
	\end{cor}
	
	\begin{ex}
		Put $F=\bR$. If $V(\bC_\lambda)$ is not super quasi-rational then Examples \ref{ex:nonsuperpure_real} and \ref{ex:type_sep_clo_case} imply
		$\bEnd_{\bmG}(V(\bC_\lambda)_\bR)\cong\bC$.
	\end{ex}
	
	\begin{rem}
		Put $F=\bR$. Let $\sigma$ be the nontrivial element of the absolute Galois group $\Gamma$. In Example \ref{ex:u(odd|odd)}, no $\lambda$ can satisfy $\lambda(\sigma,1)=\lambda$ with $|I(\lambda,\sigma)|$ odd. In particular, every super quasi-rational representation is quasi-rational in this case. In fact, write $\lambda=\sum_{i=1}^{p+q+r+s}\lambda_i e_i\in X^\ast(H\otimes_\bR\bC)$. Then
		we have
		\[\lambda(\sigma,0)=-\sum_{i=1}^{p+q}\lambda_{p+q+1-i}e_i
		-\sum_{i=1}^{r+s} \lambda_{p+q+r+s+1-i} e_{i+p+q}.\]
		If $\lambda(\sigma,1)=\lambda$ and $|I(\lambda,\sigma)|$ is odd, we have
		\[\lambda=-\sum_{i=1}^{p+q}\lambda_{p+q+1-i}e_i
		-\sum_{i=1}^{r+s} \lambda_{p+q+r+s+1-i} e_{i+p+q}
		-e_{m+1}+e_{2m+n+2}.\]
		We compare the coefficients of $e_{m+1}$ to get
		$\lambda_{m+1}=-\lambda_{m+1}-1$, which cannot hold since $\lambda_{m+1}$ is an integer (recall $p+q+1-(m+1)=m+1$).
		
		On the other hand, consider the case $\bmG={}^0\Qq(1)$ of Example \ref{ex:{}^0q(n)}. In this case, we have
		$X^\flat(H\otimes_\bR\bC)=X^\ast(H\otimes_\bR\bC)$ (see \cite[Theorem 5.5]{MR4039427} for example).
		Put $\lambda=e_2$. Then we have
		$\lambda(\sigma,0)=e_1$ and $\lambda_{(\sigma,1)}=e_1-(e_1-e_2)=e_2=\lambda$
		since $(\lambda(\sigma,0),\alpha(\sigma,1))=(e_1,e_1-e_2)=1\neq 0$. In particular, we have $\lambda(1,\sigma)=\lambda$ and $|I(\sigma,1)|=1$ (recall $\nu_\sigma=1$). We can find more examples for ${}^0\Qq(n)$ with general $n$ by a similar computation.
	\end{rem}
	
	In the non quasi-rational but super quasi-rational case, $a_{V(\bar{F}_\lambda)_F}$ is determined by Theorem \ref{thm:galois_twist_basic} through Deligne's description explained in Section \ref{sec:nonqrat}.
	
	Assume $V(\bar{F}_\lambda)$ to be super quasi-rational. To determine $\bEnd_{\bmG}(V(\bar{F}_\lambda)_F)$, it remains to compute $D_{V(\bar{F}_\lambda)_F}$. For this, take the normalized root vector $E_{-\alpha(\sigma,i)}$ in the sense of Proposition \ref{prop:hw_odd_reflection} (2) (ii).
	Define elements
	$u^+(\lambda,\sigma), u(\lambda,\sigma)\in\hy(\fg\otimes_F \bar{F})$
	by $u^+(\lambda,\sigma)=\prod_{i\in I(\lambda,\sigma)} E_{\alpha(\sigma,i)}$ and
	$u(\lambda,\sigma)=\prod_{i\in I(\lambda,\sigma)} E_{-\alpha(\sigma,i)}$
	where the products are taken in the increasing and decreasing orders of indices from left to right respectively.
	
	Take any nonzero element $v\in H^0((\bmU')^+,V(\bar{F}_\lambda))$. Then unwinding the definitions, we obtain
	\begin{flalign*}
		&(\Phi_\sigma\circ{}^\sigma\Phi_\tau\circ\Phi^{-1}_{\sigma\tau})(v)\\
		&=(\sigma\tau\ast\lambda)(w_{\sigma\tau}^{-1}\sigma(w_\tau)w_\sigma)
		u(\lambda,\sigma\tau)\Ad(w_\sigma)^{-1}(\sigma(u^+(\lambda,\tau)))u^+(\lambda,\sigma) v
	\end{flalign*}
	for $\sigma,\tau\in\Gamma$ (write down a similar diagram to \eqref{diag:BT}). One may write
	\[u(\lambda,\sigma\tau)\Ad(w_\sigma)^{-1}(\sigma(u^+(\lambda,\tau)))u^+(\lambda,\sigma) v
	=c_\lambda(\sigma,\tau)v\]
	for a unique element $c_\lambda(\sigma,\tau)\in\bar{F}^\times$.
	
	\begin{thm}
		The Borel--Tits cocycle of $V(\bar{F}_\lambda)$ is computed by
		\[\beta^{\BT}_{V(\bar{F}_\lambda)}(\sigma,\tau)
		=(\sigma\tau\ast\lambda)(w_{\sigma\tau}^{-1}\sigma(w_\tau)w_\sigma)c_\lambda(\sigma,\tau).\]	
	\end{thm}

	\subsection{Transfer method}\label{sec:ast-trivial}
	
	Consider the setting of Theorem \ref{thm:Galois_action}.
	We resume the study of case $\delta_\lambda\neq 0$ in Section \ref{sec:general}. Here we deal with an easy case:
	
	\begin{ass}\label{ass_asttriv}
		\begin{enumerate}
			\renewcommand{\labelenumi}{(\roman{enumi})}
			\item The $\ast$-action on $X^\ast(H\otimes_F F^{\sep})$ is trivial.
			\item We are given a split quasi-reductive group superscheme $\bmG^{\spl}$ and its split maximal torus $H^{\spl}$, equipped with compatible isomorphisms
			\[\begin{array}{cc}
				\bmG\otimes_F F^{\sep}\cong \bmG^{\spl}\otimes_F F^{\sep},
				&	H\otimes_F F^{\sep}\cong H^{\spl}\otimes_F F^{\sep}.
			\end{array}\]
			Let $\fh^{\spl}_{\{1\}}$ denote the $H^{\spl}$-invariant part of $\fg_{\{1\}}$. Define $\bmH^{\spl}\subset\bmG^{\spl}$ in a similar way to $\bmH$.
			\item The isomorphism
			$\fh_{\bar{1}}\otimes_F F^{\sep}\cong\fh^{\spl}_{\bar{1}}\otimes_F F^{\sep}$
			induced from (ii) is $\Gamma$-equivariant with respect to the semilinear $\Gamma$-actions defined by
			\[\begin{array}{cc}
				x\otimes c\mapsto \Ad(w^{-1}_\sigma) (x\otimes\sigma(c))
				&(x\in \fh_{\bar{1}},~c\in F^{\sep},~\sigma\in\Gamma)\\
				x\otimes c\mapsto x\otimes\sigma(c)&(x\in \fh^{\spl}_{\bar{1}},~c\in F^{\sep},~\sigma\in\Gamma).
			\end{array}\]
		\end{enumerate}
	\end{ass}
	
	Suppose that Assumption \ref{ass_asttriv} holds. We would like to compute $\beta^{\super}$ by transferring the problems of $\bmG$ to those of $\bmG^{\spl}$. Define a BPS-subgroup of $\bmG\otimes_F F^{\sep}$ by the pullback of $\bmB'$. It is defined over $F$ in $\bmG^{\spl}$ since $H^{\spl}$ is split. Define the $\ast$-action on $X^\ast(H^{\spl}\otimes_F F^{\sep})$ for this choice. The key fact is that the induced bijection
	\begin{equation}
		X^\ast(H\otimes_F F^{\sep})\cong X^\ast(H^{\spl}\otimes_F F^{\sep})
		\label{eq:char_group}
	\end{equation}
	is $\Gamma$-equivariant. This holds since the $\ast$-actions on both sides are trivial. Therefore a similar equivariance statement to (iii) for $\fh_{\bar{0}}$ holds. Take $\lambda\in X^\ast(H\otimes_F F^{\sep})$. If we write $\fu^{\spl}(\lambda)$ for the pullback of $\fu(\lambda)$ to an irreducible representation of $\bmH^{\spl}\otimes_F F^{\sep}$, (iii) implies that the pullback of ${}^{w_\sigma}({}^\sigma \fu(\lambda))$ is ${}^\sigma\fu^{\spl}(\lambda)$.
	
	For $\lambda\in X^\ast(H\otimes_F F^{\sep})$, define a quadratic form $q^{\lambda,\spl}_F$ on $\fh^{\spl}_{\bar{1}}$ as usual:
	\[q^{\lambda,\spl}_F(x)=\frac{1}{2}\lambda([x,x]).\]
	We define $d_{\lambda,F}^{\spl}$, and $\delta^{\spl}_{\lambda,F}$ accordingly. Thanks to the isomorphism
	\[\fh\otimes_F F^{\sep}\cong\fh^{\spl}\otimes_F F^{\sep},\]
	we have $d_{\lambda,F}^{\spl}=d_\lambda$ since the radicals of symmetric bilinear forms on finite dimensional vector spaces are compatible with base change of fields.
	
	\begin{prop}\label{prop:classification_ast_trivial}
		Consider the setting of Theorem \ref{thm:Galois_action}.
		Moreover, suppose that Assumption \ref{ass_asttriv} is satisfied.
		For $\lambda\in X^\flat(H\otimes_F F^{\sep})$, the following conditions are equivalent:
		\begin{enumerate}
			\renewcommand{\labelenumi}{(\alph{enumi})}
			\item There exists $\sigma\in\Gamma$ such that $\Pi \fu(\lambda)\cong {}^{w_\sigma}({}^\sigma \fu(\lambda))$,
			\item there exists $\sigma\in\Gamma$ such that $\Pi V(\fu(\lambda))\cong {}^\sigma V(\fu(\lambda))$,
			\item $\Pi V(\fu(\lambda))_F\cong V(\fu(\lambda))_F$,
			\item there exists $\sigma\in\Gamma$ such that $\Pi \fu^{\spl}(\lambda)\cong {}^\sigma \fu^{\spl}(\lambda)$,
			\item $\Pi \fu^{\spl}(\lambda)_F\cong \fu^{\spl}(\lambda)_F$,
			\item $\Pi V_F(\fu^{\spl}(\lambda)_F)\cong V_F(\fu^{\spl}(\lambda)_F)$, and
			\item $d_{\lambda,F}^{\spl}$ is odd or $\delta^{\spl}_{\lambda,F}\not\in (F^\times)^2\cup\{0\}$.
		\end{enumerate}
	\end{prop}
	
	\begin{proof}
		The equivalence of (a) and (d) follows by the preceding arguments. Proposition \ref{prop:even} implies (b) $\iff$ (c) and (d) $\iff$ (e). Combine Corollary \ref{cor:uknowsv} and Theorem \ref{thm:Galois_action} (1) to deduce (a) $\iff$ (b). Corollary \ref{cor:uknowsv} implies (e) $\iff$ (f). Finally, the equivalence of (f) and (g) follows from Proposition \ref{prop:split} (3).
	\end{proof}
	
	This gives a complete classification of irreducible representations of $\bmG$.
	
	\begin{cor}\label{cor:transfer_qrat}
		Consider the setting of Proposition \ref{prop:classification_ast_trivial}.
		\begin{enumerate}
			\item Every irreducible representation of $\bmG\otimes_F F^{\sep}$ is super quasi-rational. 
			\item For $\lambda\in X^\flat(H\otimes_F F^{\sep})$, $V(\fu(\lambda))$ is quasi-rational if and only if $d_{\lambda,F}^{\spl}$ is odd or
			$\delta^{\spl}_{\lambda,F}\in (F^\times)^2\cup\{0\}$.
		\end{enumerate}
	\end{cor}
	
	\begin{proof}
		Part (1) follows since the $\ast$-action is trivial. If $d_{\lambda,F}^{\spl}=d_\lambda$ is odd, $V(\fu(\lambda))$ is quasi-rational by (1), Proposition \ref{prop:pure}, and Example \ref{ex:type_sep_clo_case}. To complete the proof of (2), we assume $d_{\lambda,F}^{\spl}$ to be even. Then for each $\sigma\in\Gamma$, ${}^\sigma V(\fu(\lambda))$ is isomorphic to exactly one of $V(\fu(\lambda))$ and $\Pi V(\fu(\lambda))$ by (1) and Example \ref{ex:type_sep_clo_case}. Proposition \ref{prop:classification_ast_trivial} now implies that $V(\fu(\lambda))$ is quasi-rational if and only if $\delta^{\spl}_{\lambda,F}\in (F^\times)^2\cup\{0\}$.
	\end{proof}
	
	\begin{prop}
		Consider the setting of Proposition \ref{prop:classification_ast_trivial}. Then
		we have
		\[\beta^{\super}_{V(\fu(\lambda))_F}=(+,1,(\beta^{\BT}_\lambda)^{-1})
		[\bar{C}(\fh^{\spl},q^{-\lambda,\spl}_F)].\]
	\end{prop}
	
	\begin{proof}
		We have $\epsilon_{V(\fu(\lambda))_F}=\epsilon_{V_F(\fu^{\spl}(\lambda))}$ from Proposition \ref{prop:epsilon} and Corollary \ref{cor:transfer_qrat}.
		
		The product formula and the pullback argument to $\bmG^{\spl}$ imply
		\[\beta^{\BT}_{V(\fu(\lambda))}=\beta^{\BT}_\lambda \beta^{\BT}_{\fu^{\spl}(\lambda)}=\beta^{\BT}_\lambda \beta^{\BT}_{V(\fu^{\spl}(\lambda))}.\]
		In fact, one can prove $\beta^{\BT}_{\fu(\lambda)}=\beta^{\BT}_{\fu^{\spl}(\lambda)}$ by construction. Similarly, the equality
		$a_{V(\fu(\lambda))_F}=a_{V_F(\fu^{\spl}(\lambda)_F)}=a_{\fu^{\spl}(\lambda)_F}$ follows from Theorems \ref{thm:BT} (3) (ii) and \ref{thm:BT_nonqrat} (2). We identify $\pm$ with $\pm 1$ to deduce
		\[\begin{split}
			\beta_{V(\fu(\lambda))_F}
			&=(\epsilon_{V_F(\fu^{\spl}(\lambda)_F)},a_{V_F(\fu^{\spl}(\lambda)_F)},
			D_{V(\fu(\lambda))_F})\\
			&=(+,1,(\beta^{\BT}_\lambda)^{-1})
			(\epsilon_{V_F(\fu^{\spl}(\lambda)_F)},a_{V_F(\fu^{\spl}(\lambda)_F)},
			\beta^{\BT}_\lambda D_{V(\fu(\lambda))_F}),
		\end{split}\]
		\[\begin{split}
			\beta^{\BT}_\lambda D_{V(\fu(\lambda))_F}
			&=\beta^{\BT}_\lambda(\beta^{\BT}_{V(\fu(\lambda))})^{-1}
			(-\epsilon_{V(\fu(\lambda))_F},a_{V(\fu(\lambda))_F})\\
			&=(\beta^{\BT}_{V(\fu^{\spl}(\lambda))})^{-1}
			(-\epsilon_{V_F(\fu^{\spl}(\lambda)_F)},a_{V_F(\fu^{\spl}(\lambda)_F)})\\
			&=D_{V_F(\fu^{\spl}(\lambda)_F)},
		\end{split}\]
		\[(-\epsilon_{V_F(\fu^{\spl}(\lambda)_F)},a_{V_F(\fu^{\spl}(\lambda)_F)},
		D_{V_F(\fu^{\spl}(\lambda)_F)})
		=\beta^{\super}_{V_F(\fu^{\spl}(\lambda)_F)}
		=[\bar{C}(\fh^{\spl},q^{-\lambda,\spl}_F)].
		\]
		This completes the proof.
	\end{proof}

	\begin{ex}
		Put $\bmG=\Qq^\ast(2n)$ with $n\geq 1$. Take the subgroup of diagonal matrices for $H$. Identify $\bmG\otimes_\bR\bC\cong \Qq_{2n}$ to take the BPS-subgroup as in Example \ref{ex:Q_n}. Put $w=\diag(J_n,J_n)\in \bmG(\bC)$.
		This setting enjoys Assumption \ref{ass_asttriv}. To apply the preceding arguments, identify $X^\ast(H\otimes_\bR\bC)$ with $\bZ^{2n}$ as in Example \ref{ex:Q_n}. Take $\lambda\in X^\flat(H\otimes_\bR \bC)$. Due to our present choice of the positive system, $\lambda$ satisfies
		\[\lambda_1\geq \lambda_2\geq\cdots\geq\lambda_n\geq\lambda_{2n}\geq \lambda_{2n-1}
		\geq\cdots\geq \lambda_{n+1}\]
		and the unstationary condition away from zero. Since the $\ast$-involution is trivial, all the irreducible representations $V$ of $\bmG$ are super quasi-rational. To describe $\beta^{\super}_{V(\fu(\lambda))}$,
		set
		\[\begin{array}{ccc}
			n_+=|\{i\in\{1,2,\ldots,2n\}:~\lambda_i>0\}|,
			&n_-=|\{i\in\{1,2,\ldots,2n\}:~\lambda_i<0\}|,
			&|\lambda|\coloneqq \sum_{i=1}^{2n} \lambda_i.
		\end{array}\]
		Then we have
		\begin{align*}
			d_\lambda&=n_++n_-,&\epsilon_{V_\bR(\fu_\bR(\lambda))}&=-^{d_\lambda},\\
			a_{V_\bR(\fu_\bR(\lambda))}&=(-1)^{\binom{d_\lambda}{2}+n_+},
			&D_{V_\bR(\fu_\bR(\lambda))}
			&=(-1)^{|\lambda|+\binom{n_+}{2}+n_+\binom{d_\lambda-1}{2}+\binom{d_\lambda+1}{4}}.
		\end{align*}
		Here we used \cite[Example 4.1.4]{MR4627704} for computation of $\beta^{\BT}_\lambda$.
	\end{ex}

	\subsection{General case II}\label{sec:q(p,q)}
	
	Finally, we discuss how to work with the general setting of Theorem \ref{thm:Galois_action}. For this, let us recall \cite[Remark 4.7]{MR4039427} to give an explicit construction of $\fu(\lambda)$. We caution that this works over separably closed fields since so does \cite[Proof of Proposition B.4]{MR4039427} readily.
	
	Firstly, we think of general $\lambda\in X^\ast(H\otimes_F F^{\sep})$. Choose a maximal totally isotropic subspace $(\fh_{\bar{1}}\otimes_F F^{\sep})^\dagger$ to set
	$(\fh\otimes_F F^{\sep})^\dagger\coloneqq \fh_{\bar{0}}\oplus
	(\fh_{\bar{1}}\otimes_F F^{\sep})^\dagger$.
	We remark that $(\fh_{\bar{1}}\otimes_F F^{\sep})^\dagger=\fh_{\bar{1}}\otimes_F F^{\sep}$ if $\delta_\lambda=0$. On the other hand, this noncanonical choice will cause difficulties for the study of the case $\delta\neq 0$.
	
	Define $(\bmH\otimes_F F^{\sep})^\dagger$ as in \cite[Remark 4.7]{MR4039427}. Then we may set
	\[\fu(\lambda)\coloneqq
	\hy(\bmH\otimes_F F^{\sep})\otimes_{\hy((\bmH\otimes_F F^{\sep})^\dagger)} F^{\sep}_\lambda,\]
	where $F^{\sep}_\lambda=F^{\sep}$ is an $\hy((\bmH\otimes_F F^{\sep})^\dagger)$-module for the character $\lambda$ and the trivial action of $(\fh_{\bar{1}}\otimes_F F^{\sep})^\dagger$ (descend it to an action of a Clifford superalgebra through \cite[Lemma 4.1]{MR4039427}).
	We remark that this construction is compatible with Convention \ref{conv}. We also note that this construction is after \cite[\S 9.2]{MR2849718} when $F$ is algebraically closed of characteristic zero. We now get
	\[{}^{w_\sigma}({}^\sigma\fu(\lambda))\cong\hy(\bmH\otimes_F F^{\sep})\otimes_{\hy(w^{-1}_\sigma{}^\sigma(\bmH\otimes_F F^{\sep})^\dagger w_{\sigma})} F^{\sep}_{\sigma\ast\lambda}.\]
	
	Suppose that Assumption \ref{ass} holds. Recall that we wish to know whether $\fu(\lambda)\cong {}^{w_\sigma}({}^\sigma \fu(\lambda))$ and $\fu(\lambda)\cong \Pi {}^{w_\sigma}({}^\sigma \fu(\lambda))$ hopefully explicitly. These are equivalent to finding $(\fh_{\bar{1}}\otimes_F F^{\sep})^\dagger$-invariant even and odd elements of ${}^{w_\sigma}({}^\sigma\fu(\lambda))$ respectively. The difficulty is that
	$w^{-1}_\sigma{}^\sigma(\bmH\otimes_F F^{\sep})^\dagger w_{\sigma}
	\neq (\bmH\otimes_F F^{\sep})^\dagger$
	in general due to our noncanonical choice of $(\fh_{\bar{1}}\otimes_F F^{\sep})^\dagger$. We found a lucky example where we can avoid this issue:
	
	\begin{ex}
		Put $\bmG=\Qq(p,q)$ with $p+q\geq 1$. Take the subgroup of diagonal matrices for $H$. Identify $X^\ast(H\otimes_\bR \bC)\cong\bZ^{p+q}$ and $\bmG\otimes_\bR\bC\cong \Qq_{p+q}$ to take BPS-subgroup as in Example \ref{ex:Q_n}. The $\ast$-involution on $X^\ast(H\otimes_\bR \bC)$ is given by
		\[(\lambda_i)\mapsto (-\lambda_{p+q+1-i}).\]
		Take an element $\lambda\in X^\flat(H\otimes_\bR \bC)$.
		\begin{description}
			\item[Case I] Assume $\bar{\lambda}= w\lambda$. Then $V(\fu(\lambda))$ is super quasi-rational. Let us set $n_{\pm},n_0,\eta_i$ as in Example \ref{ex:Q_n_real}. Then we have $n_+=n_-$ and thus $d_\lambda=2n_{+}$. In particular, $d_\lambda$ is even. This implies $\Pi V(\fu(\lambda))\not\cong V(\fu(\lambda))$ (Example \ref{ex:type_sep_clo_case}).
			
			Let $(\fh_{\bar{1}}\otimes_\bR\bC)^\dagger$ be the linear span of
			\[\begin{array}{cc}
				\sqrt{-\lambda}_{p+q+1-i}\eta_i+\sqrt{\lambda}_i\eta_{p+q+1-i}&(1\leq i\leq n_{+})\\
				\eta_i&(n_{+}+1\leq i\leq n_{+}+n_0).
			\end{array}\]
			This is a maximal isotropic subspace. Since it is closed under the actions of the complex conjugation and the longest element of the Weyl group of $G\otimes_\bR\bC$, we have
			$w^{-1}\overline{\bmH\otimes_\bR \bC}^\dagger w
			= (\bmH\otimes_\bR \bC)^\dagger$. This implies $\overline{V(\fu(\lambda))}\cong V(\fu(\lambda))$ (recall \eqref{eq:Galois_twist_V(u)} and the preceding arguments in this small section).
			
			By virtue of Proposition \ref{prop:even} and the equality $w^{-1}\overline{\bmH\otimes_\bR \bC}^\dagger w
			= (\bmH\otimes_\bR \bC)^\dagger$, we now deduce
			\[\beta^{\super}_{V(\fu(\lambda))_\bR}=[\End_{\bmG}(V(\fu(\lambda))_\bR)]
			=(+,\beta^{\BT}_{V(\fu(\lambda))},1)=(+,1,1)=[\bR]\]
			(see \cite[Example 4.1.3]{MR4627704} for the second equality).
			\item[Case II] Assume $\bar{\lambda}\neq w\lambda$. Then $V(\fu(\lambda))$ is not super quasi-rational. Moreover, we have
			\[\bEnd_{\bmG}(V(\fu(\lambda))_\bR)\cong\begin{cases}
				\bC&(d_\lambda\in 2\bZ_{\geq 0})\\
				\bC\left[\epsilon\right]/(\epsilon^2-1)&(d_\lambda\in 2\bZ_{\geq 0}+1)
			\end{cases}\]
			by Examples \ref{ex:nonsuperpure_real} and \ref{ex:type_sep_clo_case}.
		\end{description}
	\end{ex}


\begin{thebibliography}{10}
		
		\bibitem{MR207712}
		A.~Borel and J.~Tits.
		\newblock Groupes r\'{e}ductifs.
		\newblock {\em Inst. Hautes \'{E}tudes Sci. Publ. Math.}, (27):55--150, 1965.
		
		\bibitem{MR1045822}
		S.~Bosch, W.~L{\"u}tkebohmert, and M.~Raynaud.
		\newblock {\em N\'{e}ron models}, volume~21 of {\em Ergebnisse der Mathematik
			und ihrer Grenzgebiete (3) [Results in Mathematics and Related Areas (3)]}.
		\newblock Springer-Verlag, Berlin, 1990.
		
		\bibitem{MR1973576}
		J.~Brundan and A.~Kleshchev.
		\newblock Modular representations of the supergroup {$Q(n)$}. {I}.
		\newblock volume 260, pages 64--98. 2003.
		\newblock Special issue celebrating the 80th birthday of Robert Steinberg.
		
		\bibitem{zbMATH01985814}
		J.~Brundan and J.~Kujawa.
		\newblock A new proof of the {Mullineux} conjecture.
		\newblock {\em J. Algebr. Comb.}, 18(1):13--39, 2003.
		
		\bibitem{E1914}
		E.~Cartan.
		\newblock Les groupes projectifs continus r\'{e}els qui ne laissent invariante
		aucune multiplicit\'{e}.
		\newblock {\em Journal de Math\'{e}matiques Pures et Appliqu\'{e}es},
		10:149--186, 1914.
		
		\bibitem{MR3936085}
		S-J. Cheng, B.~Shu, and W.~Wang.
		\newblock Modular representations of exceptional supergroups.
		\newblock {\em Math. Z.}, 291(1-2):635--659, 2019.
		
		\bibitem{MR3012224}
		S-J. Cheng and W.~Wang.
		\newblock {\em Dualities and representations of {L}ie superalgebras}, volume
		144 of {\em Graduate Studies in Mathematics}.
		\newblock American Mathematical Society, Providence, RI, 2012.
		
		\bibitem{MR1701598}
		P.~Deligne.
		\newblock Notes on spinors.
		\newblock In {\em Quantum fields and strings: a course for mathematicians,
			{V}ol. 1, 2 ({P}rinceton, {NJ}, 1996/1997)}, pages 99--135. Amer. Math. Soc.,
		Providence, RI, 1999.
		
		\bibitem{MR0228502}
		M.~Demazure.
		\newblock Groupes r\'{e}ductifs---g\'{e}n\'{e}ralit\'{e}s.
		\newblock In {\em Sch\'{e}mas en {G}roupes ({S}\'{e}m. {G}\'{e}om\'{e}trie
			{A}lg\'{e}brique, {I}nst. {H}autes \'{e}tudes {S}ci., 1964), {F}asc. 6,
			{E}xpos\'{e} 19}, page~34. Inst. Hautes \'{E}tudes Sci., Paris, 1965.
		
		\bibitem{MR0209401}
		J.~M.~G. Fell.
		\newblock Conjugating representations and related results on semi-simple {L}ie
		groups.
		\newblock {\em Trans. Amer. Math. Soc.}, 127:405--426, 1967.
		
		\bibitem{MR4538712}
		R.~Fioresi and F.~Gavarini.
		\newblock Real forms of complex {L}ie superalgebras and supergroups.
		\newblock {\em Comm. Math. Phys.}, 397(2):937--965, 2023.
		
		\bibitem{zbMATH02646565}
		G.~Frobenius and I.~Schur.
		\newblock {\"U}ber die reellen {Darstellungen} der endlichen {Gruppen}.
		\newblock {\em Berl. Ber.}, 1906:186--208, 1906.
		
		\bibitem{MR4627704}
		T.~Hayashi.
		\newblock Half-integrality of line bundles on partial flag schemes of classical
		{L}ie groups.
		\newblock {\em Bull. Sci. Math.}, 188:Paper No. 103317, 2023.
		
		\bibitem{MR1650134}
		M.~Hovey.
		\newblock {\em Model categories}, volume~63 of {\em Mathematical Surveys and
			Monographs}.
		\newblock American Mathematical Society, Providence, RI, 1999.
		
		\bibitem{MR2066503}
		M.~Hovey.
		\newblock Homotopy theory of comodules over a {H}opf algebroid.
		\newblock In {\em Homotopy theory: relations with algebraic geometry, group
			cohomology, and algebraic {$K$}-theory}, volume 346 of {\em Contemp. Math.},
		pages 261--304. Amer. Math. Soc., Providence, RI, 2004.
		
		\bibitem{MR4581479}
		T.~Ichikawa and Y.~Tachikawa.
		\newblock The super {F}robenius-{S}chur indicator and finite group gauge
		theories on {$\rm pin^-$} surfaces.
		\newblock {\em Comm. Math. Phys.}, 400(1):417--428, 2023.
		
		\bibitem{MR2336486}
		N.~I. Ivanova and A.~L. Onishchik.
		\newblock Parabolic subalgebras and gradings of reductive {L}ie superalgebras.
		\newblock {\em Sovrem. Mat. Fundam. Napravl.}, 20:5--68, 2006.
		
		\bibitem{MR102534}
		N.~Iwahori.
		\newblock On real irreducible representations of {L}ie algebras.
		\newblock {\em Nagoya Math. J.}, 14:59--83, 1959.
		
		\bibitem{MR1864147}
		G.~James and M.~Liebeck.
		\newblock {\em Representations and characters of groups}.
		\newblock Cambridge University Press, New York, second edition, 2001.
		
		\bibitem{MR2015057}
		J.~C. Jantzen.
		\newblock {\em Representations of algebraic groups}, volume 107 of {\em
			Mathematical Surveys and Monographs}.
		\newblock American Mathematical Society, Providence, RI, second edition, 2003.
		
		\bibitem{zbMATH03603460}
		V.~Kac.
		\newblock Representations of classical {Lie} superalgebras.
		\newblock Differ. geom. {Meth}. math. {Phys}. {II}, {Proc}., {Bonn} 1977,
		{Lect}. {Notes} {Math}. 676, 597-626 (1978)., 1978.
		
		\bibitem{MR2177301}
		G.~M. Kelly.
		\newblock Basic concepts of enriched category theory.
		\newblock {\em Repr. Theory Appl. Categ.}, (10):vi+137, 2005.
		\newblock Reprint of the 1982 original [Cambridge Univ. Press, Cambridge;
		MR0651714].
		
		\bibitem{zbMATH05117269}
		J.~Kujawa.
		\newblock The {Steinberg} tensor product theorem for {{\(\text{GL}(m|n)\)}}.
		\newblock In {\em Representations of algebraic groups, quantum groups, and Lie
			algebras. AMS-IMS-SIAM joint summer research conference, Snowbird, UT, USA,
			July 11--15, 2004.}, pages 123--132. Providence, RI: American Mathematical
		Society (AMS), 2006.
		
		\bibitem{MR1500635}
		A.~Loewy.
		\newblock Ueber die {R}educibilit\"{a}t der {R}eellen {G}ruppen linearer
		homogener {S}ubstitutionen.
		\newblock {\em Trans. Amer. Math. Soc.}, 4(2):171--177, 1903.
		
		\bibitem{MR3000482}
		A.~Masuoka.
		\newblock Harish-{C}handra pairs for algebraic affine supergroup schemes over
		an arbitrary field.
		\newblock {\em Transform. Groups}, 17(4):1085--1121, 2012.
		
		\bibitem{MR3605977}
		A.~Masuoka and T.~Shibata.
		\newblock Algebraic supergroups and {H}arish-{C}handra pairs over a commutative
		ring.
		\newblock {\em Trans. Amer. Math. Soc.}, 369(5):3443--3481, 2017.
		
		\bibitem{MR3545265}
		A.~Masuoka and A.~N. Zubkov.
		\newblock Solvability and nilpotency for algebraic supergroups.
		\newblock {\em J. Pure Appl. Algebra}, 221(2):339--365, 2017.
		
		\bibitem{MR2906817}
		I.~M. Musson.
		\newblock {\em Lie superalgebras and enveloping algebras}, volume 131 of {\em
			Graduate Studies in Mathematics}.
		\newblock American Mathematical Society, Providence, RI, 2012.
		
		\bibitem{MR2041548}
		A.~L. Onishchik.
		\newblock {\em Lectures on real semisimple {L}ie algebras and their
			representations}.
		\newblock ESI Lectures in Mathematics and Physics. European Mathematical
		Society (EMS), Z\"{u}rich, 2004.
		
		\bibitem{MR565713}
		M.~Parker.
		\newblock Classification of real simple {L}ie superalgebras of classical type.
		\newblock {\em J. Math. Phys.}, 21(4):689--697, 1980.
		
		\bibitem{MR2276521}
		F.~Pellegrini.
		\newblock Real forms of complex {L}ie superalgebras and complex algebraic
		supergroups.
		\newblock {\em Pacific J. Math.}, 229(2):485--498, 2007.
		
		\bibitem{MR0831047}
		I.~B. Penkov.
		\newblock Characters of typical irreducible finite-dimensional {${\mathfrak
				q}(n)$}-modules.
		\newblock {\em Funktsional. Anal. i Prilozhen.}, 20(1):37--45, 96, 1986.
		
		
		\bibitem{MR0714220}
		V.~V. Serganova.
		\newblock Classification of simple real {L}ie superalgebras and symmetric
		superspaces.
		\newblock {\em Funktsional. Anal. i Prilozhen.}, 17(3):46--54, 1983.
		
		\bibitem{MR2849718}
		V.~Serganova.
		\newblock Quasireductive supergroups.
		\newblock In {\em New developments in {L}ie theory and its applications},
		volume 544 of {\em Contemp. Math.}, pages 141--159. Amer. Math. Soc.,
		Providence, RI, 2011.
		
		\bibitem{MR554237}
		J-P. Serre.
		\newblock {\em Local fields}, volume~67 of {\em Graduate Texts in Mathematics}.
		\newblock Springer-Verlag, New York-Berlin, 1979.
		\newblock Translated from the French by Marvin Jay Greenberg.
		
		\bibitem{MR1867431}
		J-P. Serre.
		\newblock {\em Galois cohomology}.
		\newblock Springer Monographs in Mathematics. Springer-Verlag, Berlin, english
		edition, 2002.
		\newblock Translated from the French by Patrick Ion and revised by the author.
		
		\bibitem{MR4039427}
		T.~Shibata.
		\newblock Borel-{W}eil theorem for algebraic supergroups.
		\newblock {\em J. Algebra}, 547:179--219, 2020.
		
		\bibitem{shibata}
		T.~Shibata.
		\newblock Frobenius kernels of algebraic supergroups and {S}teinberg's tensor
		product theorem.
		\newblock {\em Algebr. Represent. Theory}, 2023.
		
		\bibitem{MR2392322}
		B.~Shu and W.~Wang.
		\newblock Modular representations of the ortho-symplectic supergroups.
		\newblock {\em Proc. Lond. Math. Soc. (3)}, 96(1):251--271, 2008.
		
		\bibitem{MR0252485}
		M.~E. Sweedler.
		\newblock {\em Hopf algebras}.
		\newblock Mathematics Lecture Note Series. W. A. Benjamin, Inc., New York,
		1969.
		
		\bibitem{MR1611385}
		J.~Tits.
		\newblock Sous-alg\`ebres des alg\`ebres de {L}ie semi-simples (d'apr\`es {V}.
		{M}orozov, {A}. {M}al\v{c}ev, {E}. {D}ynkin et {F}. {K}arpelevitch).
		\newblock In {\em S\'{e}minaire {B}ourbaki, {V}ol. 3}, pages Exp. No. 119,
		197--214. Soc. Math. France, Paris, 1956.
		
		\bibitem{MR0224710}
		J.~Tits.
		\newblock Classification of algebraic semisimple groups.
		\newblock In {\em Algebraic {G}roups and {D}iscontinuous {S}ubgroups ({P}roc.
			{S}ympos. {P}ure {M}ath., {B}oulder, {C}olo., 1965)}, pages 33--62. Amer.
		Math. Soc., Providence, R.I., 1966.
		
		\bibitem{MR2069561}
		V.~S. Varadarajan.
		\newblock {\em Supersymmetry for mathematicians: an introduction}, volume~11 of
		{\em Courant Lecture Notes in Mathematics}.
		\newblock New York University, Courant Institute of Mathematical Sciences, New
		York; American Mathematical Society, Providence, RI, 2004.
		
		\bibitem{MR4279905}
		J.~Voight.
		\newblock {\em Quaternion algebras}, volume 288 of {\em Graduate Texts in
			Mathematics}.
		\newblock Springer, Cham, [2021] \copyright 2021.
		
		\bibitem{MR0167498}
		C.~T.~C. Wall.
		\newblock Graded {B}rauer groups.
		\newblock {\em J. Reine Angew. Math.}, 213:187--199, 1963/64.
		
	\end{thebibliography}
\end{document}